\documentclass[12pt,a4paper]{amsart}
\usepackage[a4paper]{geometry}
\usepackage[utf8]{inputenc}
\usepackage{mathtools}
\usepackage{amssymb}
\usepackage{amsthm}
\usepackage{amsmath}
\usepackage{mathrsfs}
\usepackage[pagewise,mathlines]{lineno}
\usepackage{array}
\usepackage{booktabs}
\usepackage{multirow}
\usepackage{hyperref}
\usepackage[
type={CC},
modifier={by},
version={4.0},
]{doclicense}
\usepackage{bbm}
\usepackage[backend=biber,style=numeric,maxbibnames=10
]{biblatex}
\usepackage{comment}
\usepackage{enumitem}
\usepackage{xcolor}
\usepackage{mathtools}
\mathtoolsset{showonlyrefs}

\makeatletter
\newcommand{\namedlabel}[2]{%
	\phantomsection
	#1\def\@currentlabel{#1}\label{#2}%
}
\makeatother

\newtheorem{theorem}{Theorem}[section]
\newtheorem*{theorem*}{Theorem}
\newtheorem{corollary}[theorem]{Corollary}
\newtheorem{lemma}[theorem]{Lemma}
\newtheorem{prop}[theorem]{Proposition}
\newtheorem{definition}[theorem]{Definition}
\newtheorem{remark}[theorem]{Remark}
\newtheorem*{remark*}{Remark}

\numberwithin{equation}{section}

\newcommand{\ulu}{\boldsymbol{u}}
\newcommand{\ulv}{\boldsymbol{v}}
\newcommand{\ulg}{\boldsymbol{g}}
\newcommand{\ulpsi}{\boldsymbol{\psi}}
\newcommand{\ulphi}{\boldsymbol{\varphi}}

\newcommand{\ulf}{\boldsymbol{f}}
\newcommand{\ulzero}{\boldsymbol{0}}
\newcommand{\ulw}{\boldsymbol{w}}
\newcommand*\diff{\mathop{}\!\mathrm{d}}
\newcommand{\twoasts}{{2^*_s}}
\newcommand{\twoastsone}{{2^*_1}}
\newcommand{\Clip}[1]{L_{#1}}
\newcommand{\Nloc}{\mathscr{N}_{s,\mathrm{loc}}}
\newcommand{\Nlocnu}{\mathscr{N}_{s,\nu,\mathrm{loc}}}

\newcommand{\Nglob}{\mathscr{N}_{s,\mathrm{glob}}}

\newcommand{\NStri}{\mathscr{N}_{s,\mathrm{Stri}}}
\newcommand{\psnu}{p_{s,\nu}}
\newcommand{\thetasnu}{\theta_{s,\nu}}

\title{Existence of Solutions for time-dependent fractional Kohn-Sham Equations}

\author[S. Breteaux]{S{\'e}bastien Breteaux}
\address[S. Breteaux]{Universit{\'e} de Lorraine, CNRS, IECL, F-57000 Metz, France}
\email{sebastien.breteaux@univ-lorraine.fr}

\author[M. Fantechi]{Michele Fantechi}
\address[M. Fantechi]{Universit{\'e} de Lorraine, CNRS, IECL, F-57000 Metz, France}
\email{michele.fantechi@univ-lorraine.fr}

\author[J. Faupin]{J{\'e}r{\'e}my Faupin}
\address[J. Faupin]{Universit{\'e} de Lorraine, CNRS, IECL, F-57000 Metz, France}
\email{jeremy.faupin@univ-lorraine.fr}

\begin{document}
	\maketitle
	
	\begin{abstract}
		We consider time-dependent Kohn-Sham equations in dimension $3$ with a fractional dispersion relation $(1-\Delta)^s$, $s\in(0,\frac32)$, and a class of interaction terms including, in particular, external potentials, internal potentials associated to Hartree-type non-linearities, and exchange terms described by energy subcritical pure-power non-linearities. We prove the local existence of weak solutions in $H^s$ using an approximation procedure regularizing the non-linearities. Assuming that the interaction energies can be controlled by the kinetic energy, we show that the solutions can be extended to global solutions using energy estimates. If $s\in[1,\frac32)$, we establish in addition the well-posedness of the time-dependent Kohn-Sham equations using Strichartz estimates.
	\end{abstract}
	
	\section{Introduction}
	
	Density Functional Theory (DFT) is the most successful computational tool used in solid state physics, quantum chemistry, and beyond to understand large atoms, molecules, solids, and even large proteins \cite{engel2011density}. The energy of a system of $N$ electrons with interaction potential $w$ and external potential $V$ is typically given by an atomic Hamiltonian of the following type
	\begin{equation}\label{eq:N-body-Ham}
		H_N = \sum_{j=1}^N \Big( -\frac{\hbar^2}{2m} \Delta_{x_j} + V(x_j) \Big) + \sum_{1\leq j_1 < j_2\leq N} w(x_{j_1}-x_{j_2})\, ,
	\end{equation}
	where $\hbar$ is Planck's constant divided by $2\pi$, $m$ is the electron mass and $x_j$ stands for the position of the $j^{\text{th}}$ electron.  The electronic orbitals are computed through an $N$-body wave function~$ \Psi_N \in L^2_{\mathrm{asym}}(\mathbb{R}^{3N})$, where $\mathrm{asym}$ denotes the antisymmetric elements, which enforces Pauli’s principle. Note that any numerical approach to the problem faces an explosion of complexity due to the scaling for a large number of electrons $N$ and the rich structure of quantum many-body correlations; its practical use is therefore severely constrained to systems of limited size.
	
	The advantage of DFT lies in a reduction in the computational cost of determining the electronic orbitals of molecules by substituting the main object of study from the wave functions of each single electron to the (one-body) electronic density 
	\begin{equation}\label{eq:def_density}
		\rho_{\Psi_N}(x)=\int_{\mathbb{R}^{3(N-1)}}|\Psi_N(x,x_2,\cdots,x_N)|^2\diff x_2\cdots\diff x_N.
	\end{equation}
	This comes at the cost of finding a way to express the atomic Hamiltonian’s energy by an alternative energy functional that depends only on the electronic density.
	
	The first instance of a theory of this type was the Thomas–Fermi model \cite{enrico1927metodo,Thomas_1927}, which, however, fails to predict the binding of atoms \cite{Teller1962}. The turning point of DFT was the celebrated Hohenberg–Kohn Theorem \cite{HohenbergKohn}, which, at a physics level of rigor, provides in principle a way to construct an effective potential, dependent only on the electronic density, that reproduces the correct ground state of a many-body system of electrons. This allows for the computation of the ground state and the excited states by employing, in principle, only the electronic density.
	
	Further contributions by Kohn and Sham then provided a practical way to compute the electronic orbitals through the Kohn–Sham equations, which describe a fictional system of non-interacting electrons (consisting of a Slater determinant $ \frac{1}{\sqrt{N!}} \det \psi_j(x_k) $) minimizing the density functional, thereby attaining the correct solution. A rigorous mathematical formulation of the DFT energy functional was given by Levy and Lieb \cite{Levy1979,lieb1983density}.
	
	The fundamental idea is to consider the infimum of the energy $H_N$ as a function of the external potential $V$, expressed as the sum of two contributions
	\begin{equation}\label{eq:minimization_problem}
		\Big( \inf_{\| \Psi_N \|_2 =1} \langle \Psi_N, H_N \Psi_N \rangle \Big) (V) = \inf_{\substack{\sqrt{\rho} \in H^1 \\ \int\rho = 1}} \Big( \mathscr{E}_{\mathrm{LL}}(\rho) + \int_{\mathbb{R}^3} V(x) \rho(x) \diff  x  \Big)\, ,
	\end{equation}
	where $H^1$ is the usual Sobolev space and $ \mathscr{E}_{\mathrm{LL}} $ is the Levy–Lieb functional:	
	\begin{multline}
		\mathscr{E}_{\mathrm{LL}} (\rho) = \inf_{\| \Psi_N \|=1,\, \rho_{\Psi_N} = \rho}  \int_{\mathbb{R}^{3N}} \Big(\sum_{j=1}^N \frac{\hbar^2}{2m} | \nabla_{x_j} \Psi_N(x_1,\dots,x_N)|^2 \\
		+ \sum_{1\leq j_1<j_2\leq N} w(x_{j_1}-x_{j_2}) | \Psi_N(x_1,\dots,x_N)|^2 \Big)\diff x_1\cdots\diff x_N\, ,
	\end{multline}
	and $ \rho_{\Psi_N} $ denotes the density of the wave function $\Psi_N$ as in \eqref{eq:def_density}. Formally, the statement of the Hohenberg–Kohn Theorem can be summarized as the fact that the map $V \mapsto \rho(V) $, where~$\rho(V)$ is the electronic density of the minimizer of the right-hand side of \eqref{eq:minimization_problem}, is one-to-one up to a constant shift of the potential $V$.
	
	The energy functional for the Kohn–Sham equations, where we minimize over Slater determinants $ \Xi = \frac{1}{\sqrt{N!}} \det \psi_j(x_k)$, can be divided into different contributions,
	\begin{equation}
		\mathscr{K}(\rho) = \inf_{\substack{\psi_1, \dots, \psi_N \in H^1 \\ \rho_\Xi = \rho }} \int_{\mathbb{R}^3} \sum_{j=1}^N \frac{\hbar^2}{2m} | \nabla \psi_j (x) |^2 \diff  x
	\end{equation}
	being the kinetic contribution, and
	\begin{equation}
		\int_{\mathbb{R}^3} V(x) \rho_\Xi(x) \diff  x + \frac{1}{2} \int_{\mathbb{R}^3 \times \mathbb{R}^3 } w(x-y) \rho_\Xi(x) \rho_\Xi(y) \diff  x\diff  y
	\end{equation}
	forming, respectively, the external potential and Hartree (uncorrelated or classical energy) contributions to the energy, that can be computed exactly through the density. The last contribution is formed by the exchange–correlation term of the energy, which cannot be expressed exactly through the density and is defined as
	\begin{equation}
		\mathscr{G}_{\mathrm{xc}}(\rho) = \mathscr{E}_{\mathrm{LL}}(\rho) - \mathscr{K}(\rho) - \frac{1}{2} \int_{\mathbb{R}^3 \times \mathbb{R}^3 } w(x-y) \rho(x) \rho(y) \diff  x\diff  y\, .
	\end{equation}
	The exchange–correlation part of the energy in its general form is not known, and in practice, a large number of approximations are used instead. The basic approximation used in this context is the Local Density Approximation (LDA), where the exchange–correlation energy is approximated by a function of the electronic density
	\begin{equation}
		\mathscr{G}_{\mathrm{xc}}(\rho) \approx \int_{\mathbb{R}^3} g_{xc}(\rho(x)) \diff  x\, .
	\end{equation}
	The LDA is expected to hold under the assumption that the density varies slowly \cite{engel2011density}.
	The most basic approximation of the exchange term was computed formally in the original paper by Kohn and Sham \cite{kohnsham1965}, from the expansion of the exchange term of the uniform electron gas, in the case of a Coulomb interaction potential. The obtained non-linear term is a pure power of the electronic density,
	\begin{equation}\label{eq:exchange_power}
		\mathscr{G}_{\mathrm{xc}}(\rho) \approx \mu \int_{\mathbb{R}^3} \rho^{4/3}(x) \diff  x,
	\end{equation}
	for some $\mu \in \mathbb{R}$. Treating approximations of the exchange–correlation functional, including this case, is one of the motivations of the present work. This approximation is usually corrected further by terms dependent on the gradient of the density $ \nabla_x \rho(x) $, as well as higher-order terms, in practical applications of quantum chemistry. Approximations of the correlation term of the energy can be expressed in the LDA in high- or low-density regimes, yielding a contribution~$ g_{\mathrm{c}}(\rho(x)) = c \cdot \rho(x) $ at leading order \cite[Section 4.3]{engel2011density}. A rigorous mathematical analysis of the LDA arising as an approximation of the energy using the uniform electron gas was initiated in a series of works by Lewin, Lieb, and Seiringer \cite{lewin2018statistical,lewin2019local}. A comprehensive review of the mathematical results on DFT is given by the same authors in \cite{lewin2022universal}.
	
	The mathematical problem of existence, uniqueness, and regularity of solutions of the equations arising in the dynamical version of DFT, the time-dependent density functional theory (TDDFT), is the focus of the present paper. A dynamical analogue of the Hohenberg–Kohn Theorem is known in physics as the Runge–Gross Theorem and the van Leeuwen construction \cite{RungeGross1984,Leeuwen1999}. In a similar manner, these results construct a solution to the evolution of a molecule by constructing an effective time-dependent potential that is a function of the time-dependent electronic density only \cite[Section 7.1]{engel2011density}. A practical approach to the time-dependent problem is given by the dynamical version of the Kohn–Sham equations -- called the time-dependent Kohn–Sham equations -- in which the evolution is determined by the Euler–Lagrange equations derived from the density energy functional. The mathematical analysis of the derivation of the time-dependent Kohn-Sham equations is still in its early stages, and a number of problems have been identified \cite{Schrimer2025,FournaisetAl2016}.
	
	We consider here time-dependent Kohn–Sham equations as a starting point and do not address the problem of a rigorous derivation of the equations from first principles. In particular, we focus on the dynamical version of the LDA, the Adiabatic Local Density Approximation (ALDA), in which the additional assumption is that the effective potential arising in the density functional is time-adiabatic. In other words, we assume that the effective terms of the functional depend only on the time-dependent electronic density \cite[Section 7.3]{engel2011density}. On the other hand, we consider a generalized ``fractional'' dispersion relation $\omega(-i\nabla)=(1-\Delta)^s$ with $s\in(0,\frac{d}{2})$ (we work in dimension $d=3$ to simplify the exposition but our analysis straightforwardly extend to any dimension $d\ge3$), including both the usual non-relativistic dispersion relation (up to a trivial gauge transformation) $\omega(-i\nabla)=-\Delta$ as in \eqref{eq:N-body-Ham}, and the pseudo-relativistic dispersion relation $\omega(-i\nabla)=(1-\Delta)^{\frac12}$.

	\subsection{The Time-Dependent Fractional Kohn-Sham Equations}
	
	Let $\gamma$  be the $1$-particle reduced density matrix of the many-body wavefunction provided by a Slater determinant for any number of fictitious electrons,
	\begin{equation}\label{def:density_matrix}
		\gamma(t) = \sum_{k \geq 0} | \psi_k(t) \rangle \langle \psi_k(t) | , \;\ \;\ \psi_k (t) \in L^2(\mathbb{R}^3),
	\end{equation}
	and let  
	\begin{equation}\label{def:density}
		\rho_{\ulpsi(t)}(x) := \sum_{k \geq 0} | \psi_{k}(t,x) |^2 ,
	\end{equation}
	be the associated one body electronic density. The effective, dynamical, fractional Kohn-Sham equations associated with $(\psi_k(t,x))_k$, derived from a Density Functional Theory in the Adiabatic Local Density Approximation, as described in the introduction, usually read
	\begin{equation}
		\label{eq:Kohn_Sham}
		\begin{dcases}
			i \partial_t \psi_{k}(t,x) = \big((1- \Delta)^s  + \mu \rho_{\ulpsi(t)}(x)^\alpha +(w*\rho_{\ulpsi(t)})(x) +V(x)\big) \psi_{k}(t,x) \, , \quad \, \forall k \in \mathbb{N}\, ,\\
			\psi_k(0,x) = \varphi_k(x)\,,\, \quad\forall k \in \mathbb{N},
		\end{dcases}
	\end{equation}
	with $\mu\in\mathbb{R}$ a constant, $\alpha>0$ ``not too large'' in a sense we will precise later, $w$ an internal interaction potential as in the Hartree equation, and $V$ and external potential. Some regularity will be required on the initial states $(\varphi_k(0,x))_k$, see below.
	
	As eluded to above, the kinetic term in \eqref{eq:Kohn_Sham} is given by the generalized dispersion relation~$ \omega(-i \nabla) = (1-\Delta)^{s} $ (with $s\in(0,\frac{3}{2})$, recalling that we work in dimension $d=3$), including both the non-relativistic case $s=1$, with the Laplacian kinetic term (upon a simple gauge transformation to absorb $1$ in a global phase term) and the pseudo-relativistic case $s=\frac12$, with the kinetic term $ \langle \nabla \rangle =\sqrt{1-\Delta}$. The restriction to $s \in (0,\tfrac{3}{2})$ is a matter of convenience. In fact a similar analysis could be extended to the case $s \geq \frac{3}{2}$ taking care of modifying our assumptions to account for the estimates available beyond the critical Sobolev embedding.
	
	We in fact consider generalized time-dependent fractional Kohn-Sham equations of the form
	\begin{equation}
		\label{eq:Kohn_Sham2}
		\begin{dcases}
			i \partial_t \psi_{k}(t,x) = (1- \Delta)^s\psi_{k}(t,x)  + g_k\big((\psi_k(t))_k\big)(x)  \, , \quad\, \forall k \in \mathbb{N}\, ,\\
			\psi_k(0,x) = \varphi_k(x)\,,\, \quad \forall k \in \mathbb{N},
		\end{dcases}
	\end{equation}
	with the abstract ``non-linearities'' $g_k$ satisfying suitable assumptions that will be stated below. We restrict ourselves to the ALDA and consider functionals which depend only on the density. As remarked in the introduction, the non-linear terms of the density functional are not all known explicitly due the presence of the exchange-correlation functional. We thus make an effort to keep the model reasonably general. Recall that the explicitly known terms of the Levy-Lieb functional include the last two terms appearing on the right-hand side of \eqref{eq:Kohn_Sham}, namely the external potential contribution
	\begin{equation}\label{eq:NL-V}
		V(x) \, \psi_{k}(t,x),
	\end{equation}
	which is actually linear, and does not depend on the density $\rho_{\ulpsi}$, and the classical correlation term
	\begin{equation}\label{eq:NL-w}
		(w*\rho_{\ulpsi(t)})(x) \, \psi_{k}(t,x),
	\end{equation}
	also called the Hartree term.  We want to cover in addition local pure power non-linearities as in \eqref{eq:Kohn_Sham}, of the type 
	\begin{equation}\label{eq:NL-alpha}
		\mu \,\rho_{\ulpsi(t)}(x)^\alpha \, \psi_{k}(t,x),
	\end{equation}
	as proposed originally by Kohn and Sham, see \eqref{eq:exchange_power}.
	
	We will prove that our assumptions on the abstract non-linearities $g_k$ in \eqref{eq:Kohn_Sham2} are satisfied for sums of non-linearities given as in \eqref{eq:NL-V}, \eqref{eq:NL-w} or \eqref{eq:NL-alpha}, provided that $V$ and $w$ belong to suitables sums of $L^p$ spaces, and that the pure power non-linearity parameter $\alpha$ belongs to $[0 , \frac{2s}{3-2s} ) $ (which includes the case $\alpha=\frac13$ for $s\in[\frac38,\frac32)$). We will only consider non-linearities that can be controlled by the kinetic energy.
	
	We will assume that the initial condition is such that $(\|\varphi_k\|_{L^2})_k \in \ell^2 $ (we do \emph{not} require that $\|\varphi_k\|_{L^2}=1$) and we will show that the conservation law $ \| \psi_k(t) \|_{L^2} = \| \varphi_k \|_{L^2} $ holds for all integer $k$. Let us already observe that Eq.~\eqref{eq:Kohn_Sham2} has some formally conserved quantities, such as an analogue of the trace norm for the Hartree-Fock equation
	\begin{equation}
		\sum_{k \geq 0} \| \psi_k(t) \|^2_{L^2} 
		= \mathrm{Tr}(\gamma(t))\,,
	\end{equation}
	and an energy  
	\begin{equation}
		\frac{1}{2} \Big(\sum_{k \geq 0} \| \psi_k(t) \|_{H^s}^2\Big) + \mathscr{G}\big((\psi_k(t))_k\big),
	\end{equation}
	for some functional $\mathscr{G}$ that will be defined below. These conservation laws will play an important role both in our proof of the existence of local solutions to \eqref{eq:Kohn_Sham2} through the well-posedness of a regularized problem and for the extension of local solutions to global ones.

	\subsection{Notations}
	
	Before stating our assumptions, we introduce some notations that we will use throughout this text. 
	
	The symbol $\mathbb{N}$ stands for set of natural numbers, including $0$. The positive and negative parts of a real-valued function $f$ are defined by $f^{\pm}(x)=\max\{0,\pm f(x)\}$. 
	
	The open ball centered at $a\in X$ and of radius $r>0$ in a normed vector space $X$ is denoted by~$B_X(a,r)$. 
	
	When we consider a Banach space $B$ as a topological space endowed with the norm topology, we just write~$B$, but if we consider $B$ endowed with the weak topology, we write $B_{\mathrm{weak}}$. For any topological spaces~$X_1$ and~$X_2$, the space of continuous functions from~$X_1$ to~$X_2$ is denoted by $\mathscr{C}(X_1;X_2)$.  We will also consider, when~$X_1$ is an open subset of~$\mathbb{R}^m$ and~$X_2$ is a normed vector space, the space~$\mathscr{C}^k(X_1;X_2)$, $k\in\mathbb{N}\cup\{\infty\}$, of functions of class $\mathscr{C}^k$ from~$X_1$ to~$X_2$,  the subspace $\mathscr{C}^k_{\mathrm{c}}(X_1;X_2)$ of~$\mathscr{C}^k(X_1;X_2)$ of functions with compact support, and the space~$\mathscr{C}^{0,\alpha}(X_1;X_2)$, $\alpha>0$, of H\"older continuous functions of order $\alpha$ from $X_1$ to $X_2$. 
	
	The usual $\mathbb{C}$-valued Sobolev spaces on $\mathbb{R}^3$ are written as $H^s=H^s(\mathbb{R}^3;\mathbb{C})$ with $s \in \mathbb{R}$, while the Lebesgue spaces are denoted by $L^p=L^p(\mathbb{R}^3;\mathbb{C})$ with $p \in \left[1,\infty \right]$. The notation $L^{p^+}+L^\infty$ means
	\begin{equation*}
		L^{p^+}+L^\infty=\bigcup_{\varepsilon\in(0,\infty)}\big(L^{p+\varepsilon}+L^\infty\big),
	\end{equation*}
	and $\mathcal{M}_0$ denotes the space of finite, signed Radon measures over $\mathbb{R}^3$ endowed with the total variation norm.
	
	Likewise, for any real interval $I$ and Banach space $B$, the usual Lebesgue spaces are denoted by~$L^p(I;B)$, for~$p\in[1,\infty]$, and the Lebesgue-Sobolev spaces are denoted by~$W^{k,p}(I;B)$, with~$k\in\mathbb{N}$, $p\in[1,\infty]$. 
	
	Our normalization of the Fourier transform is the unitary version
	\begin{equation}
		\mathcal{F}(f)(\xi) \coloneq \hat{f}(\xi) = \frac{1}{(2 \pi)^{3/2}} \int_{\mathbb{R}^3}  e^{-i \xi \cdot x} f(x) \diff x.
	\end{equation}
	
	For sequences depending on the index $k \in \mathbb{N}$, that is corresponding to the different components of \eqref{eq:Kohn_Sham}, we will use bold letters. For example: 
	$\ulpsi=(\psi_k)_{k\in\mathbb{N}} $, $\ulphi=(\varphi_k)_{k\in\mathbb{N}}$, or more generally, any sequence depending on the parameter $k$ will be written with a bold font if we do not write the index $k$. For example the non-linearities in \eqref{eq:Kohn_Sham2} will also be denoted by
	\begin{equation}\label{eq:def-interaction}
		\boldsymbol{g}(\boldsymbol{\psi})=(g_{k}(\boldsymbol{\psi}))_{k\in\mathbb{N}} .
	\end{equation}
	Likewise, an operator $A$ acting on each element of the sequence will also be denoted by $A \ulu $, that is $A \ulu = (A u_k)_{k\in\mathbb{N}} $.
	Equation \eqref{eq:Kohn_Sham2} can be written handily using these notations:
	\begin{equation}\label{eq:Kohn-Sham-vectorised}
		\begin{dcases}
			i \partial_t \ulpsi(t) = (1-\Delta)^s \ulpsi(t) + \ulg (\ulpsi(t))  \\
			\ulpsi(0,x)=\ulphi(x) \,.
		\end{dcases}
	\end{equation}
	The typical non-linearities given by \eqref{eq:NL-V}, \eqref{eq:NL-w} and \eqref{eq:NL-alpha} can also be rewritten using these notations as
	\begin{equation}
		\ulg_V (\ulpsi) := V\,\ulpsi, \quad    \ulg_{\mathrm{H}} (\ulpsi) := \big(w \ast \rho_{\ulpsi}  \big) \, \ulpsi,\quad    \ulg_{\mathrm{KS}}(\ulpsi) := \mu\, \rho_{\ulpsi}^{\alpha}\, \ulpsi.
	\end{equation}
	
	For any Banach space $B$,   $\ell^2 (B)$ is the Banach space of sequences $ \ulu $ in $B^\mathbb{N}$ such that the norm
	\begin{equation}
		\|  \ulu  \|_{	\ell^2(B)} \vcentcolon = \sqrt{ \sum_{k \geq 0}  \| u_k \|_{B}^2 }
	\end{equation}
	is finite.
	In particular we will consider the spaces $ \ell^2 (L^p) $, $ \ell^2 (L^2\cap L^p) $, $ \ell^2 (L^2+ L^p) $, 
	and $ \ell^2 (H^s) $, 
	where $ \ell^2 (H^s) $ is a Hilbert space for the inner product:
	\begin{equation}
		\langle \ulu  , \ulv  \rangle_{\ell^2 (H^s) } \vcentcolon = \sum_{k \geq 0}  \langle u_k, v_k \rangle_{H^s} .
	\end{equation}
	The usual Sobolev embeddings in dimension $ d=3$ can be easily lifted to the previous norms. Namely, considering $s\in(0,\frac32)$, $\nu \in \left[2, \twoasts \right) $, where $$\twoasts=\frac{6}{3-2s},$$ we have
	\begin{equation}
		\ell^2(H^{s}) \xhookrightarrow{} \ell^2(L^2\cap L^\nu), \quad \ell^2(L^2+ L^{\nu^\prime}) \xhookrightarrow{}  \ell^2(H^{-s}),
	\end{equation}
	that is, the following inequalities hold (as they follow from the usual Sobolev inequalities):
	\begin{align}
		\label{eq:sobolev_embedding_plus}
		\| \ulpsi \|_{ \ell^2(L^2\cap L^\nu)} &\leq  C_{s} \| \ulpsi \|_{ \ell^2(H^{s})}, \\
		\label{eq:sobolev_embedding_minus}
		\| \ulpsi \|_{ \ell^2(H^{-s})} &\leq  C_{s} \| \ulpsi \|_{ \ell^2(L^2+ L^{\nu^\prime})},
	\end{align}
	for some positive constant $C_s$. Here and it what follows, $\nu'$ denotes the conjugate exponent of~$\nu$, namely $\frac1\nu+\frac{1}{\nu'}=1$.
	
	\begin{remark}
		Note that the energy norm $ \| \cdot \|_{\ell^2(H^s)}$ is nothing more that the usual trace norm for the density matrix $\gamma$ from \eqref{def:density_matrix} with a weight, that is
		\begin{equation}
			\| \ulpsi \|_{\ell^2(H^s)}^2 = \mathrm{Tr}( (1-\Delta)^{\frac{s}{2}} \gamma (1-\Delta)^{\frac{s}{2}}  ) .
		\end{equation}
		In this work we use the formulation $ \| \ulpsi \|_{\ell^2(H^s)}^2 $  as it is more compatible with the embeddings needed for the non-linear terms we consider. 
	\end{remark}
	
	For $\nu\in[2,\twoasts)$, we set
	\begin{equation}\label{eq:psnu}
		\psnu=\frac{\nu(\twoasts -2)}{2(\twoasts -\nu)}\quad\text{and}\quad\psnu'=\frac{\nu(\twoasts -2)}{\twoasts (\nu-2)},
	\end{equation}
	which satisfy $1/\psnu+1/\psnu'=1$.
	We also set
	\begin{equation}\label{eq:eta-thetasnu}
		\eta_{s,\nu}(x)=x^{1/\psnu'}+1 \quad\text{and}\quad \thetasnu (x)=x^{1/\psnu}+x\,.
	\end{equation}

	\subsection{Assumptions}
	
	We now specify the classes of non-linearities $\ulg(\ulpsi)$ we will consider in this text. The functional~$\ulg$ will be defined as a map $\ulg: \ell^2(L^2\cap L^\nu) \to \ell^2(L^2+L^{\nu'})$, for some $\nu\in[2,\twoasts)$, satisfying some of the following conditions: 
	\begin{itemize}
		\item[\namedlabel{(N1)}{assump1}] The first condition is the vanishing of $\ulg$ at $\ulzero$:
		\begin{equation}
			\ulg(\ulzero)=\ulzero\,; 
		\end{equation}
		\item[\namedlabel{(N2)}{ass2:nonlinearity_Lp_bound}]  The second condition is slightly weaker than a Lipschitz condition:
		\begin{multline}
			\forall M>0\,,\,\exists \Clip{M}>0\,, \forall \ulu,\ulv\in B_{\ell^2(H^s)}(\ulzero,M)\,,\\
			\| \ulg(\ulu) - \ulg (\ulv) \|_{\ell^2(L^2+L^{\nu^\prime})}  \leq \Clip{M} \, \| \ulu -\ulv \|_{\ell^2(L^2\cap L^{\nu})}\,;
		\end{multline}
		\item[\namedlabel{(N2')}{ass2stri:nonlinearity_Lp_bound}] The second condition is also considered in a stronger sense: There exist a set $J$ such that $\mathrm{card}(J)<\infty$ and real numbers $\nu_j\in[2,\min\{\twoasts,6\})$, $j\in J$, satisfying  
		\begin{align}
			& \ulg=\sum_{j\in J}\ulg_j \Big|_{\ell^2(L^2 \cap L^{\nu})}  \text{ where } \nu = \max_{j\in J} \nu_j\,, \quad \ulg_j \in \mathscr{C}\big(\ell^2(L^{\nu_j});\ell^2(L^{\nu'_j})\big)\,, 
		\end{align}
		and
		\begin{multline}
			\forall M>0\,,\,\exists \Clip{M}>0\,, \forall j\in J\,,\,\forall \ulu,\ulv\in B_{\ell^2(H^s)}(\ulzero,M)\,,\\
			\| \ulg_j(\ulu) - \ulg_j(\ulv) \|_{\ell^2(L^{\nu_j^\prime})}  \leq \Clip{M} \, \| \ulu -\ulv \|_{\ell^2( L^{\nu_j})}\,;
		\end{multline}
		\item[\namedlabel{(N3)}{ass3:nonlinear_energy_derivative}] The third condition states the existence of a function $\mathscr{G}$, called a non-linearity energy, whose Gateaux derivative is $\ulg$:
		\begin{align}
			\exists\, \mathscr{G} & : \ell^2(L^2\cap L^{\nu}) \to \mathbb{R} \text{ such that } \forall \ulu,\ulv \in \ell^2(L^2\cap L^{\nu})\,, \nonumber \\
			\frac{\diff}{\diff \tau}\mathscr{G}(\ulu+\tau \ulv) \Big|_{\tau=0} 
			&= \langle \ulg (\ulu), \ulv \rangle_{ \ell^2(L^2+L^{\nu'}),\ell^2(L^2\cap L^{\nu})}+\langle  \ulv,\ulg (\ulu) \rangle_{\ell^2(L^2\cap L^{\nu}), \ell^2(L^2+L^{\nu'})} \nonumber\\
			&=\vcentcolon \mathrm{Re}\langle \ulg (\ulu), \ulv \rangle_{ \ell^2(L^2+L^{\nu'}),\ell^2(L^2\cap L^{\nu})}\,; 
		\end{align}
		\item[\namedlabel{(N4)}{ass4:nonlinearity_global_estimate}] The fourth condition states that the negative part of the non-linearity energy $\mathscr{G}$ can be controlled by the energy and mass norms:
		\begin{align}
			\exists \, a\in(0,1),\, D:(0,\infty)\to(0,\infty) \text{ locally bounded, such that } \\
			\forall \ulu \in \ell^2(H^s)\,,\quad 
			\mathscr{G}^{-}(\ulu) \leq \frac{a}{2} \| \ulu \|_{\ell^2(H^s)}^2 + D( \| \ulu \|_{\ell^2(L^2)}) .
		\end{align}
	\end{itemize}
	
	
	\begin{remark}\label{rmk:N2prime_implies_N2}
		Note that the condition \ref{ass2stri:nonlinearity_Lp_bound} implies the condition \ref{ass2:nonlinearity_Lp_bound} because, for $\nu_j \in [2,\nu]$
		\begin{equation}
			\|\ulf\|_{\ell^2(L^2+L^{\nu'})} \lesssim \|\ulf\|_{\ell^2(L^{\nu_j'})} \quad \text{and}\quad \|\ulu\|_{\ell^2( L^{\nu_j})} \lesssim \|\ulu\|_{\ell^2(L^2\cap L^\nu)}\,.
		\end{equation}
	\end{remark}
	For $s\in(0,\frac32)$, we introduce the class
	\begin{align*}
		\Nloc =   \bigcup_{\nu \in \left[2, \twoasts \right)} \Nlocnu,
	\end{align*}
	where
	\begin{align*}
		\Nlocnu =  \left\{\ulg\in\mathscr{C}\big(\ell^2(L^2\cap L^{\nu});\ell^2(L^2+L^{\nu'})\big)
		\mid
		\ulg \text{ satisfies \ref{assump1}, \ref{ass2:nonlinearity_Lp_bound} and  \ref{ass3:nonlinear_energy_derivative}}\right\},
	\end{align*}
	and two distinct sub-classes of $\Nloc$:
	\begin{align}
		\Nglob = & \Big\{\ulg\in\Nloc \mid \ulg \text{ satisfies \ref{ass4:nonlinearity_global_estimate}}  \Big\}\,, \\
		\NStri = & \Big\{ \ulg \in \Nloc \mid \ulg \text{ satisfies \ref{ass2stri:nonlinearity_Lp_bound}} \Big\} .
	\end{align}
	
	We will suppose that $\ulg\in\Nloc$ throughout the text, in particular to prove the local existence of solutions to \eqref{eq:Kohn-Sham-vectorised} (see Theorem \ref{thm:local_existence} below), while $\ulg\in\Nglob$ will only be used to prove the global existence (see Theorem \ref{thm:global_existence}) and $\ulg \in \NStri $ will be assumed for $s \in [1,\frac32) $ to prove the uniqueness and well-posedness of solutions using Strichartz estimates (see Theorem \ref{thm:well_posedness}). Assuming both $ \ulg\in\Nglob $ and~$ \ulg \in \NStri $ yields the global well-posedness result Corollary \ref{cor:global_wellposedness}.

	In Section \ref{sec:examples} below, we will prove that the typical non-linearities given by \eqref{eq:NL-V}, \eqref{eq:NL-w} and~\eqref{eq:NL-alpha} belong to $\Nloc$, $\NStri$ or $\Nglob$ under suitable conditions on the external and internal potentials $V$ and $w$, and on the exponent $\alpha$ of the pure power non-linearity, as given in the following Table \ref{tab:non-linearities_classes}. Note that in the table, the class $\NStri$ is considered only in the relevant case $s \in [1,\tfrac{3}{2})$. 
	
	\begin{table}[ht]
		\begin{center}
			\begin{tabular}{@{}lcccc@{}}
				\toprule
				& & \multicolumn{3}{c}{Examples}\\ \cmidrule(r){3-5}
				\multicolumn{2}{l}{Class of} & Power & Hartree & External\tabularnewline
				\multicolumn{2}{l}{non-linearities} & $\rho^\alpha_{\ulpsi} \, \ulpsi $ & $(w * \rho_{\ulpsi})\,\ulpsi $  & potential $V \ulpsi $\\
				& $s\in$ & $\alpha\in$ & $w\in$ & $V\in$\\
				\cmidrule(r){1-2} 
				\cmidrule(r){3-5}
				$\Nloc$ & $\left(0,\dfrac{3}{4}\right]$ & $\left[0,\dfrac{2s}{3-2s}\right)$ & $L^{\frac{3^{+}}{4s}}+L^{\infty}$ & $L^{\frac{3^{+}}{2s}}+L^{\infty}$\\ \addlinespace
				& $\left(\dfrac{3}{4},\dfrac{3}{2}\right)$ &  $\left[0,\dfrac{2s}{3-2s}\right)$ & $\mathcal{M}_{0}+L^{\infty}$ & $L^{\frac{3^{+}}{2s}}+L^{\infty}$\\ \addlinespace
				$\Nglob$ & $\left(0,\dfrac{3}{2}\right)$ & $\left[0,\dfrac{2s}{3}\right)$ & $L^{\frac{3^{+}}{2s}}+L^{\infty}$ & $L^{\frac{3^{+}}{2s}}+L^{\infty}$\\ \addlinespace
				$\NStri$ & $\left[1,\dfrac{3}{2}\right)$ & $[0,2)$ & $\mathcal{M}_{0}+L^{\infty}$ & $L^{\frac{3^{+}}{2s}}+L^{\infty}$\\
				\bottomrule
			\end{tabular}
		\end{center}
		\caption{Examples of parameters yielding ``non-linearities'' in the $\mathscr{N}_{s,x}$ classes, depending on~$s$. Note that actually the conditions for $\Nglob$ are only relevant for $\mu \rho_{\ulpsi}^{\alpha}$ with $\mu<0$ and $w^{-}$ while the ``non-negative'' parts only need to satisfy the conditions to yield non-linearities in $\Nloc$.}
		\label{tab:non-linearities_classes}
	\end{table}
	The example non-linearities treated by our results are summarized in the next proposition. The proof is obtained by the results in Section \ref{sec:examples} and Remark \ref{rmk:N2prime_implies_N2}.
	
	\begin{prop}
		Let $s\in (0,\frac{3}{2})$. If
		\begin{enumerate}[label=(\roman*)]
			\item\label{range_alpha_loc} $\mu_1,\,\mu_2\in\mathbb{R}$ and $\alpha_1,\,\alpha_2\in \left[0,\frac{2s}{3-2s}\right)$, and
			\item\label{range_w_loc} $w$ is real and even, with $w\in L^{\frac{3^+}{4s}}+L^\infty$ when~$s\leq \frac{3}{4}$, or 
			$w\in \mathcal{M}_0+L^\infty$ when~$s> \frac{3}{4}$, and
			\item\label{range_V_loc} $V$ is real and such that $V\in L^{\frac{3^+}{2s}}+L^\infty$,
		\end{enumerate}
		and the term $\ulg$ is defined through
		\begin{equation}
			\ulg(\ulpsi) = \mu_1 \rho_{\ulpsi}^{\alpha_1} \, \ulpsi + \mu_2 \rho_{\ulpsi}^{\alpha_2} \, \ulpsi + (w*\rho_{\ulpsi}) \, \ulpsi + V\ulpsi\,,
		\end{equation}
		then $\ulg\in\Nloc$.
		
		If \ref{range_alpha_loc}, \ref{range_w_loc}, \ref{range_V_loc}, ($\mu_{j}>0$ or $\alpha_j<\frac{2s}{3}$, $j=1,2$) and $w^{-} \in L^{\frac{3^+}{2s}}+L^\infty$
		then $\ulg\in\Nglob$.
		
		Similarly if \ref{range_alpha_loc}, \ref{range_w_loc}, \ref{range_V_loc}, $s\geq 1$ and $\alpha_j<2$, $j=1,2$, then $\ulg\in\NStri$.
		
	\end{prop}


	\subsection{Main results}
	
	The notion of solutions we consider in our results is the following:
	
	\begin{definition}[Weak and strong solutions to the Kohn-Sham equations]\label{def:solutions}Let $s\in(0,\frac{3}{2})$ and~$ \ulg \in \Nloc$.
		\begin{itemize}
			\item A weak solution in $\ell^2(H^s)$ to the Kohn-Sham equations \eqref{eq:Kohn-Sham-vectorised} on an open time interval $I $ containging $0$ with initial condition $ \ulphi \in \ell^2(H^s) $ is a sequence $ \ulpsi \in L^\infty(I; \ell^2(H^s)) \cap W^{1,\infty}(I; \ell^2(H^{-s}))$ such that $ \ulpsi(0) = \ulphi $ and
			\begin{equation}
				i \partial_t \ulpsi(t) = (1-\Delta)^s \ulpsi(t) + \ulg(\ulpsi(t)) \;\ \text{ in } \ell^2(H^{-s}),
			\end{equation}
			for almost every $t \in I$, 
			or equivalently, 
			\begin{equation}
				\ulpsi(t) = e^{-it (1-\Delta)^s} \ulphi - i \int_{0}^{t}  e^{-i(t-\tau)(1-\Delta)^s} \ulg(\ulpsi(\tau)) \diff  \tau \;\ \text{ in } \ell^2(H^{-s}) ,
			\end{equation}
			for almost every $t \in I$.
			\item For $\sigma\ge0$ (and assuming that $\ulg$ is well-defined on $\ell^2(H^\sigma)$), a strong solution in $\ell^2(H^\sigma)$ to the Kohn-Sham equations \eqref{eq:Kohn-Sham-vectorised} on an open time interval $I $ containing $0$ is a sequence~$ \ulpsi$ in the space~$\mathscr{C}(I; \ell^2(H^\sigma)) \cap \mathscr{C}^1(I; \ell^2(H^{\sigma-2s}))$ such that $ \ulpsi(0) = \ulphi $ and
			\begin{equation}
				i \partial_t \ulpsi(t) = (1-\Delta)^s \ulpsi(t) + \ulg(\ulpsi(t)) \;\ \text{ in } \ell^2(H^{\sigma-2s}) ,
			\end{equation}
			for all $t\in I$. 
		\end{itemize}
	\end{definition}

	We can now state our results on the existence of weak solutions in $ \ell^2(H^{s}) $ for the Kohn-Sham equations \eqref{eq:Kohn-Sham-vectorised}. 
	
	
	\begin{theorem}\label{thm:local_existence}
		Let $s\in(0,\frac32)$. Suppose that $\ulg\in\Nloc$ and consider $M>0$. For any $ \ulphi$  in the closed ball $\bar{B}_{\ell^2(H^{s})}(0,M)$ there exists $T(M) >0 $ such that the Kohn-Sham equations \eqref{eq:Kohn-Sham-vectorised} with initial condition~$ \ulphi  $ have a local weak solution $\ulpsi$ in $\ell^2(H^s)$ on the time interval $I(M)=\big(-T(M),T(M)\big)$. This solution satisfies in addition 
		\begin{equation}
			\ulpsi \in \mathscr{C}\big(I(M);\ell^2(H^{s})_{\mathrm{weak}}\big) \cap \bigcap_{r \in [2,\twoasts ) } \mathscr{C}^{0,1/(2p_{s,r})}(I(M);\ell^2(L^r)) .
		\end{equation}
		Moreover,
		\begin{enumerate}[label=\roman*)]
			\item \label{it:cons-l2L2} The conservation of the mass of each particle, $\| \varphi_k \|_{L^2} = \| \psi_k(t) \|_{L^2} $, holds for all $k \in \mathbb{N}$ and every $ t \in I(M)$;
			\item \label{it:kin-bound} The kinetic energy bound $ \| \ulpsi \|_{L^{\infty}(I(M) ; \ell^2(H^{s}))}  \leq 2 M $ holds;
			\item \label{it:en-ineq} The energy inequality  $ \ \mathcal{E}(\ulpsi(t)) \leq \mathcal{E}(\ulphi ) $ holds for almost every  $t \in I(M)$.
		\end{enumerate}
	\end{theorem}
	
	Our proof of Theorem \ref{thm:local_existence} employs an approximation scheme similar to that used, e.g., in \cite[Theorem 3.3.5]{cazenave2003semilinear} for nonlinear Schrödinger equations. It is here suitably adapted to the framework of the Kohn-Sham equations. We emphasize that, due to this approximation procedure, we only obtain the existence of solutions, not their uniqueness, and only the energy inequality stated in \ref{it:en-ineq}, not the conservation of energy.
	
	Under additional assumptions on the non-linearity, the weak solutions in $\ell^2(H^s)$ constructed in the previous theorem can be extended to be global. 
	
	\begin{theorem}\label{thm:global_existence}
		Let $s\in(0,\frac32)$. Suppose that $\ulg\in\Nglob$ and consider $ \ulphi  \in \ell^2(H^{s}) $. There exists a global weak solution $ \ulpsi$ in~$\ell^2(H^s)$ to \eqref{eq:Kohn-Sham-vectorised}  with initial condition $ \ulphi $. This solution satisfies in addition
		\begin{equation}
			\ulpsi \in \mathscr{C}\big(\mathbb{R};\ell^2(H^{s})_{\mathrm{weak}}\big) \cap \bigcap_{r \in [2,\twoasts ) } \mathscr{C}_{\mathrm{loc}}^{0,1/(2p_{s,r})}(\mathbb{R};\ell^2(L^r)) ,
		\end{equation}
		and
		\begin{enumerate}[label=\roman*)]
			\item The conservation of the mass of each particle, $\| \varphi_k \|_{L^2} = \| \psi_k(t) \|_{L^2} $, holds for all $k \in \mathbb{N}$ and every $ t \in \mathbb{R}$;
			\item \label{it:en-ineq-glob} The energy inequality  $ \ \mathcal{E}(\ulpsi(t)) \leq \mathcal{E}(\ulphi ) $ holds for almost every  $t \in \mathbb{R}$.
		\end{enumerate}
	\end{theorem}
	
	To deduce Theorem \ref{thm:global_existence} from Theorem \ref{thm:local_existence}, we use, in particular, the fact that the condition~$\ulg\in\Nglob$ on the non-linearity implies that the non-linear terms in the energy can be controlled by the kinetic term. Hence the energy inequality \ref{it:en-ineq-glob} from Theorem \ref{thm:global_existence} also implies the kinetic energy bound $ \| \ulpsi \|_{L^{\infty}(\mathbb{R} ; \ell^2(H^{s}))}  \leq C $ (compare to \ref{it:kin-bound} of Theorem \ref{thm:local_existence}).
	
	\begin{remark}
		This global existence result could be extended to the case when a critical scaling in the energy is present in the following sense:
		\begin{equation*}
			\mathscr{G}^-(\ulu)  \leq K \| \ulu \|_{\ell^2(H^s)}^{2} D(\| \ulu \|_{\ell^2(L^2)})
		\end{equation*}
		for some constant $K>0$ and some positive locally bounded function $D$ such that $D(x) \rightarrow 0$ as~$x \rightarrow 0$. This estimate is critical in the sense that the non-linear contribution to the energy is controlled by the kinetic term but with a prefactor that depends on the size of the $\ell^2(L^2)$ norm of the solution. Assuming that the norm $ \| \ulphi \|_{\ell^2(L^2)}$ is small enough, the same type of argument can be used to obtain a global existence result.
	\end{remark}
	
	In the particular case where $s \in [1,\frac32) $ (including the non-relativistic case $s=1$), Strichartz estimates yield the uniqueness of the local solutions constructed in Theorem \ref{thm:local_existence}, which in turn allows us to lift them to strong $\ell^2(H^s)$ solutions. 
	
	\begin{theorem}\label{thm:well_posedness}
		Let $s\in[1,\frac32)$. Suppose that $\ulg\in\NStri$ and consider $ \ulphi  \in \ell^2(H^{s}) $. Then \eqref{eq:Kohn-Sham-vectorised} is well-posed, that is there exist unique minimal and maximal times of existence, $T_{\mathrm{min}}\in[-\infty,0)$ and $T_{\mathrm{max}}\in(0,\infty]$, and a unique strong $ \ell^2(H^s) $ solution $\ulpsi$ defined on 
		$ I_{\mathrm{max}}(\ulphi) = \left(T_{\mathrm{min}}, T_{\mathrm{max}} \right)$ such that
		\begin{enumerate}[label=\roman*)]
			\item The blow-up alternative holds:
			\begin{itemize}
				\item If $T_{\mathrm{min}} > -\infty $, then $  \| \ulpsi(t) \|_{\ell^2(H^s)} \xrightarrow[t \rightarrow T_{\mathrm{min}}]{} \infty$ and likewise,
				\item  if $T_{\mathrm{max}} < \infty $, then $ \| \ulpsi(t) \|_{\ell^2(H^s)} \xrightarrow[t \rightarrow T_{\mathrm{max}}]{} \infty$;
			\end{itemize}
			\item The mass of each particle is preserved:
			\begin{equation}
				\forall t \in \left(T_{\mathrm{min}}, T_{\mathrm{max}} \right)\,,\, \forall k \in \mathbb{N}, \;\  \| \psi_k(t) \|_{L^2} = \| \varphi_k \|_{L^2},
			\end{equation}
			and the energy is preserved:
			\begin{equation}
				\forall t \in \left(T_{\mathrm{min}}, T_{\mathrm{max}} \right) , \quad  \mathcal{E}(\ulpsi(t)) = \mathcal{E}(\ulphi);
			\end{equation}
			\item Let $\bar{I}$ be any closed finite interval which contains $0$ and $\Omega(\bar{I})\subset \ell^2(H^s)$ be the set of $\ulphi$ such that $\bar{I} \subset I_{\mathrm{max}}(\ulphi)$. Then $\Omega(\bar{I})$ is an open subset of $\ell^2(H^s)$ and the map 
			\begin{equation*}
				\Omega(\bar{I}) \ni \ulphi \mapsto \ulpsi \in \mathscr{C}(\bar{I};\ell^2(H^s))
			\end{equation*}
			is continuous (with the norm topology of $\ell^2(H^s)$ on $\Omega(\bar{I})$).
		\end{enumerate}
	\end{theorem}

	The use of Strichartz estimates (with no loss of derivatives since $s\in[1,\frac32)$) to prove the uniqueness of solutions is adapted from e.g. \cite[Section 4.2]{cazenave2003semilinear}. Strichartz estimates could in fact be used to prove both existence and uniqueness by a contraction map method in appropriate norms. We refrain to do so in this paper.
	
	Combining the previous two theorems, we obtain:
	
	\begin{corollary}\label{cor:global_wellposedness}
		Let $s\in[1,\frac32)$. Suppose that $\ulg\in\Nglob\cap\NStri$ and consider $ \ulphi  \in \ell^2(H^{s}) $. Then the Kohn-Sham equations \eqref{eq:Kohn-Sham-vectorised} are well-posed in $\ell^2(H^s)$ and admit global solutions on the whole real line. 
	\end{corollary}
	
	Note that in the case $s=1$, the Kohn-Sham equations have a dispersion relation of the form~$ 1-\Delta $, and we can obtain a set of equations with the usual non-relativistic dispersion relation $ -\Delta $ by the simple unitary transformation $ U_t = e^{it \mathbbm{1}} $ which amounts of a global phase term:
	\begin{equation}
		i \partial_t (U_t \ulpsi)(t) = - U_t \ulpsi(t) + U_t (1- \Delta  ) \ulpsi(t) +  U_t \, \ulg(\ulpsi(t)) 
		= - \Delta U_t \ulpsi(t) + \ulg(U_t \ulpsi(t)).
	\end{equation}
	Assuming the following ``gauge'' condition to hold: 
	\begin{equation}\label{eq:gauge}
		\forall t\in \mathbb{R},\,\quad \ulg(e^{it}\ulu) = e^{it} \ulg(\ulu),
	\end{equation} 
	we deduce that $U_t \ulpsi$ is a solution to the usual non-relativistic Kohn-Sham equations. We will show in Section \ref{sec:examples} below that \eqref{eq:gauge} is satisfied for any non-linearity of the form given by \eqref{eq:NL-V}, \eqref{eq:NL-w} or \eqref{eq:NL-alpha}.

	\subsection{Comparison with the literature}
	
	Considering the well-posedness of the time-dependent Kohn–Sham equations in the local density approximation, a number of results are present in the literature. The case of the Kohn–Sham equations constrained to an open bounded domain is considered in \cite{JEROME2015995,BorzietAl2017}. In \cite{JEROME2015995}, a unique global weak solution is constructed under the hypotheses of a non-time-adiabatic memory term for the exchange–correlation energy contribution, with rather restrictive conditions: the term must be uniformly Lipschitz in time and densities in the energy space, the external potential must be positive and $\mathscr{C}^1$ in time and space, and the Coulomb interaction potential must be used. Borzi et al. \cite{BorzietAl2017} construct a unique strong solution in the case of an exchange–correlation term in the ALDA, which is again assumed to be Lipschitz in the energy space and in $L^2$, and in the presence of a time-dependent external potential assumed to be $W^{2,\infty}$.
	
	The case of electrons not constrained to a bounded domain is considered in \cite{pusateri2021long,dupuy2025linear} with different goals. In \cite{pusateri2021long}, global well-posedness results are provided for ALDA exchange–correlation terms of the pure power type, as in \eqref{eq:exchange_power}, but with a constraint on the power that excludes this relevant example. The analysis is carried out for half-density operators, which are the square root of the density operator of the Slater determinants of the Kohn–Sham equations. The asymptotic behavior in time for the resulting nonlinear equations is analyzed, providing an asymptotic completeness result in the noncritical case, in analogy with the scattering theory for the nonlinear Schrödinger equation. As we were finalizing this paper the very relevant pre-print \cite{kawamoto2026} appeared, where the authors prove global existence and modified scattering in the non-relativistic case $s=1$, for small initial data in dimensions $2$ and $3$ with the critical Coulomb potential $1/|\cdot|$ for $w$, the Kohn-Sham expression $ \mu \rho^{\frac{1}{d}} $ for the exchange term and no external potential $V$. 
	
	The approach of \cite{dupuy2025linear} is focused on the linear response theory arising from the time-dependent perturbation of the ground state of the density functional, and the well-posedness result is obtained under a number of hypotheses of regularity on the perturbation potential, which must be $H^2$, and on the exchange–correlation energy term, notably that it must be of class $\mathscr{C}^4$.
	
	In the current paper, we relax these assumptions, allowing for general time-independent external potentials and interaction potentials, and for an exchange–correlation term that is of class $\mathscr{C}^1$, allowing for instance for \eqref{eq:exchange_power} and therefore relaxing the requirements for a term whose mathematical properties are not completely understood.
	
	We briefly discuss here the technical difficulties of the current paper. We consider a reduced density matrix possibly with infinite rank, that is, a general positive trace-class operator, as the state of the system of electrons.
	The natural norm to use would then be the trace-class norm, which we however write in an alternative way, fixing the basis and exploiting the structure of the equations to help us overcome some technical problems. These technical issues arise from the interplay of the generalized kinetic term we consider and the regularity properties of the nonlinear terms. The general dispersion relation we consider here implies a loss of derivatives in the associated dispersive estimates for $s < 1$, which poses a problem in the presence of a nonlinear term that is not regular enough. The typical fixed-point argument in a Banach space can be adapted only for some terms. The classical correlation or Hartree term can be treated in the presence of the loss of derivatives \cite{breteaux2025propagation}, and some strategies allow us to treat an external potential, assuming for instance Kato smallness with respect to the kinetic term. In the case of a general exchange–correlation term, taking as an example~\eqref{eq:exchange_power}, the term is not regular enough to close the estimates in energy space without using the “good” dispersive properties of the kinetic term that are only available through the celebrated Strichartz estimates for the Laplacian operator (or more generally for $s\ge1$). Treating terms arising from the exchange–correlation functional that may lack the required regularity requires an approximation procedure to construct the solutions, which is the main technical problem overcome in this paper.

	\subsection{Outline of the article}
	Our paper is organized as follows. In Section \ref{sec:examples}, we prove that the conditions given in Table~\ref{tab:non-linearities_classes} are sufficient so that the non-linearities given by~\eqref{eq:NL-V}, \eqref{eq:NL-w} and \eqref{eq:NL-alpha} belong to $\Nloc$, $\Nglob$ or $\NStri$. In Section \ref{sec:prelim}, we establish several preliminary estimates allowing us to control the classes of non-linearities we consider in this paper and the regularization procedure employed in our construction of solutions to the Kohn-Sham equations. Sections \ref{sec:local-ex}, \ref{sec:global} and \ref{sec:Strichartz} are devoted to the proofs of Theorems \ref{thm:local_existence}, \ref{thm:global_existence} and \ref{thm:well_posedness} respectively. 
	
	\subsection*{Acknowledgements}
	This research was funded, in whole or in part, by l’Agence Nationale de la Recherche
	(ANR), project ANR-22-CE92-0013. We are grateful to Viviana Grasselli and Israel Michael Sigal for useful discussions. 
	
	\section{Examples of non-linear terms}\label{sec:examples}

	In this section, we prove that the main examples of ``non-linearities'' we have in mind, namely those associated to an external potential \eqref{eq:NL-V}, an internal convolution potential \eqref{eq:NL-w} and a local pure power non-linearity \eqref{eq:NL-alpha}, belong to the class $\Nloc$, $\NStri$ or $\Nglob$ under the conditions given in Table \ref{tab:non-linearities_classes}. Note that by Remark \ref{rmk:N2prime_implies_N2} it suffices to prove \ref{ass2stri:nonlinearity_Lp_bound} to obtain \ref{ass2:nonlinearity_Lp_bound}.
	We begin with the class $\NStri$ for pure power non-linearities in Subsection \ref{subsec:pure}, for convolution non-linearities in Subsection \ref{subsec:convol} and for external potentials in Subsection \ref{subsec:ext}. Subsection \ref{subsec:Nglob} concerns the subclass $\Nglob$ for the three previous examples. Subsection \ref{subsec:gauge} concerns the gauge condition \eqref{eq:gauge} mentioned in the introduction.

	\subsection{Local pure power non-linearity}\label{subsec:pure}
	
	Recall that  the local pure power non-linearity we consider is given by
	\begin{equation}
		\ulg_{\mathrm{KS}}(\ulu) = 	\mu \rho_{\ulu}^\alpha \, \ulu\, .
	\end{equation} 
	Obviously, we have $\ulg_{\mathrm{KS}}(\ulzero)=\ulzero$ and hence \ref{assump1} is satisfied.
	
	Now we verify that \ref{ass2stri:nonlinearity_Lp_bound} is satisfied. We start by considering pointwise estimates.
	\begin{lemma}
		Let $ \ulu = (u_k)_{k\in\mathbb{N}} $ and $ \ulv = (v_k)_{k\in\mathbb{N}} $, with $u_k,v_k:\mathbb{R}^3\to\mathbb{C}$ for all $k\in\mathbb{N}$. Let $x\in\mathbb{R}^3$ be such that $\ulu(x),\ulv(x)\in\ell^2(\mathbb{C})$. Then
		\begin{equation}
			\label{est:pointwise_densities_1/2}
			| \rho_{\ulu }^{\frac{1}{2}} (x) - \rho_{\ulv }^{\frac{1}{2}} (x) | \leq \rho_{\ulu -\ulv }^{\frac{1}{2}}(x).
		\end{equation}
		
		Moreover, if  $ \ulu \in \ell^2 (L^q) $ with $ q \in [2,\infty] $, then
		\begin{equation}
			\label{est:densities_Lp}
			\| \rho_{\ulu }^{\frac{1}{2}}  \|_{L^q} \leq \| \ulu  \|_{\ell^2(L^q)} .
		\end{equation}
	\end{lemma}
	\begin{proof}
		The proof of \eqref{est:pointwise_densities_1/2} follows from observing that from the definition of the density $\rho_{\ulu }$, the square root of $\rho_{\ulu }$ is pointwise equal to the $ \ell^2 (\mathbb{C}) $ norm of the family $ \ulu  $:
		\begin{equation*}
			\forall x \in \mathbb{R}^3, \,\ \rho_{\ulu }^{\frac{1}{2}}(x) = \Bigg( \sum_{k \geq 0 }  |u_k(x)|^2 \Bigg)^{\frac{1}{2}} = \vcentcolon \| \ulu (x) \|_{ \ell^2 (\mathbb{C})}.
		\end{equation*}
		As such, the inverse triangle inequality must hold pointwise:
		\begin{equation*}
			| \rho_{\ulu }^{\frac{1}{2}} (x) - \rho_{\ulv }^{\frac{1}{2}} (x) |  = 	\big| \| \ulu (x) \|_{ \ell^2 (\mathbb{C})} - \| \ulv (x) \|_{ \ell^2 (\mathbb{C})} \big| \leq \| \ulu (x) - \ulv (x) \|_{ \ell^2(\mathbb{C})} = \rho_{\ulu -\ulv }^{\frac{1}{2}}(x)  \,,
		\end{equation*}
		which is \eqref{est:pointwise_densities_1/2}.
		
		To prove \eqref{est:densities_Lp},	it suffices to estimate, by Minkowski inequality (notice that $ q \geq 2$),
		\begin{equation*}
			\| \rho_{\ulu }^{\frac{1}{2}} \|_{L^q} = \| \rho_{\ulu } \|_{L^{\frac{q}{2}}}^{\frac{1}{2}} = \Big\| \sum_{k \geq 0}  |u_k(x) |^2  \Big\|_{L^{\frac{q}{2}}}^{\frac{1}{2}} \leq \Big( \sum_{k \geq 0}  \| u_k \|_{L^{q}}^2 \Big)^{\frac{1}{2}} = \|\ulu  \|_{\ell^2 (L^q) } .
		\end{equation*}
		
	\end{proof}
	
	We will need another pointwise estimate stated in the following lemma. 
	
	\begin{lemma} \label{lem:power_pointwise_estimate}	Let $ \ulu = (u_k)_{k\in\mathbb{N}} $ and $ \ulv = (v_k)_{k\in\mathbb{N}} $, with $u_k,v_k:\mathbb{R}^3\to\mathbb{C}$ for all $k\in\mathbb{N}$. Let $x\in\mathbb{R}^3$ be such that $\ulu(x),\ulv(x)\in\ell^2(\mathbb{C})$. Let $\alpha\ge0$. Then
		\begin{multline}
			| \rho_{\ulu }^\alpha (x) u_k(x) - \rho_{\ulv }^\alpha (x) v_k(x) | \\
			\leq \rho_{\ulv }^\alpha (x) \, | u_k(x) - v_k(x) | + (2\alpha+2) ( \rho_{\ulu }^\alpha (x) + \rho_{\ulv }^\alpha (x) )\, \rho_{\ulu -\ulv }^{\frac{1}{2}} (x) \frac{|u_k(x)|}{\rho_{\ulu }^{\frac{1}{2}}(x)} ,
		\end{multline}
		for all $k$ in $\mathbb{N}$, with the convention that the quotient is $0$ if $\rho_{\ulu }(x)=0$.
	\end{lemma}
	\begin{proof}
		If $\rho_{\ulv}(x)=0$ then $v_k(x)=0$ for all $k\in\mathbb{N}$ and the statement is obvious. Assume then that $\rho_{\ulv}(x)\neq0$. Likewise, il $\rho_{\ulu}(x)=0$ then the statement is obvious and we therefore also assume that $\rho_{\ulu}(x)\neq0$. We omit the $x$ in the notation for simplicity.
		
		We begin by a basic estimate, using the triangle inequality:
		\begin{equation}
			| \rho_{\ulu}^\alpha  u_k  - \rho_{\ulv}^\alpha  v_k |  
			\leq   | \rho_{\ulu}^{\alpha+\frac{1}{2}}  - 			\rho_{\ulv}^{\alpha+\frac{1}{2}}  | \Bigg| 			\frac{u_k}{\rho_{\ulu}^{\frac{1}{2}}}  \Bigg| + \rho_{\ulv}^{\alpha+\frac{1}{2}}  \Bigg| 	\frac{u_k}{\rho_{\ulu}^{\frac{1}{2}}} - \frac{v_k}{\rho_{\ulv}^{\frac{1}{2}}} \Bigg| . \label{est:split_step_pointwise}
		\end{equation}
		The first term of the sum can be estimated by considering that $ \rho_{\sharp}^{\alpha+\frac{1}{2}}  =  ( \rho_{\sharp}^{\frac{1}{2}}  )^{2 \alpha + 1} $ and, 
		if $c\geq1 $, by Lagrange's mean value theorem
		\begin{equation*}
			| a^c - b^c | \leq c (a^{c-1} + b^{c-1})\, |a-b| .
		\end{equation*}
		From these considerations, we get
		\begin{equation*}
			| \rho_{\ulu}^{\alpha+\frac{1}{2}}  - \rho_{\ulv}^{\alpha+\frac{1}{2}}  | \leq (2\alpha +1) ( \rho_{\ulu}^{\alpha}  +  \rho_{\ulv}^{\alpha} ) \, |  \rho_{\ulu}^{\frac{1}{2}}  - \rho_{\ulv}^{\frac{1}{2}}  | .
		\end{equation*}
		Applying the last inequality, $|u_k|\le \rho_{\ulu}^{\frac{1}{2}}$ and \eqref{est:pointwise_densities_1/2} into the first term of \eqref{est:split_step_pointwise}, we get
		\begin{equation}\label{est:split_step_pointwise2}
			| \rho_{\ulu}^{\alpha+\frac{1}{2}}  - \rho_{\ulv}^{\alpha+\frac{1}{2}}  | \Bigg| 			\frac{u_k}{\rho_{\ulu}^{\frac{1}{2}}}  \Bigg|\leq  (2\alpha+1)( \rho_{\ulu}^{\alpha}  +  \rho_{\ulv}^{\alpha} ) \, \rho_{\ulu-\ulv}^{\frac{1}{2}} .
		\end{equation}
		To estimate the second term in \eqref{est:split_step_pointwise}, we simply group the denominators and add and subtract appropriately,
		\begin{multline*}
			\rho_{\ulv}^{\alpha+\frac{1}{2}}  \Bigg|  \frac{u_k}{\rho_{\ulu}^{\frac{1}{2}}} - \frac{v_k}{\rho_{\ulv}^{\frac{1}{2}}} \Bigg| = \frac{\rho_{\ulv}^{\alpha} }{\rho_{\ulu}^{\frac{1}{2}}} \big| u_k \rho_{\ulv}^{\frac{1}{2}} - v_k \rho_{\ulu}^{\frac{1}{2}} \big| \\
			\leq   \frac{\rho_{\ulv}^{\alpha} }{\rho_{\ulu}^{\frac{1}{2}}} \big( | u_k \rho_{\ulv}^{\frac{1}{2}} - u_k \rho_{\ulu}^{\frac{1}{2}} | + | u_k \rho_{\ulu}^{\frac{1}{2}} - v_k \rho_{\ulu}^{\frac{1}{2}} | \big) 
			=  \rho_{\ulv}^{\alpha}  \Bigg| \frac{ u_k}{ \rho_{\ulu}^{\frac{1}{2}}  } \Bigg| \big|  \rho_{\ulu}^{\frac{1}{2}}  - \rho_{\ulv}^{\frac{1}{2}}  \big| +  \rho_{\ulv}^{\alpha}  | u_k-v_k | .
		\end{multline*}
		Applying again \eqref{est:pointwise_densities_1/2}, we obtain
		\begin{equation}\label{est:split_step_pointwise3}
			\rho_{\ulv}^{\alpha+\frac{1}{2}}  \Bigg| \frac{u_k}{\rho_{\ulu}^{\frac{1}{2}}} - \frac{v_k}{\rho_{\ulv}^{\frac{1}{2}}} \Bigg| \leq \rho_{\ulv}^{\alpha}  \Bigg| \frac{u_k}{ \rho_{\ulu}^{\frac{1}{2}}  } \Bigg| \rho_{\ulu-\ulv}^{\frac{1}{2}}   +  \rho_{\ulv}^{\alpha}  | u_k-v_k | .
		\end{equation}
		The desired estimate follows from \eqref{est:split_step_pointwise}, \eqref{est:split_step_pointwise2} and \eqref{est:split_step_pointwise3}.
	\end{proof}
	
	Now we restrict the possible powers to handle only long range interactions. The next proposition shows that \ref{ass2stri:nonlinearity_Lp_bound} holds for the local pure power non-linearity. 
	
	\begin{prop}\label{lem:power_nonlinear_estimates}
		Let $\alpha\geq 0$. If $ \nu(\alpha)$ 
		is defined by 
		\begin{equation}\label{eq:nu(alpha)}
			\nu(\alpha) := 2(\alpha+1), 
		\end{equation}
		then, for all $ \ulu,\ulv\in\ell^2(L^{\nu(\alpha)}) $, 
		\begin{equation}
			\label{est:power_Lp}
			\big\| \rho_{\ulu}^\alpha \ulu - \rho_{\ulv}^\alpha \ulv \big\|_{\ell^2(L^{\nu^\prime(\alpha)})} \leq 2(2\alpha+2)  \big( \| \ulu \|_{\ell^2(L^{\nu(\alpha)})}^{2\alpha} + \| \ulv \|_{\ell^2(L^{\nu(\alpha)})}^{2\alpha} \big) \| \ulu-\ulv \|_{\ell^2(L^{\nu(\alpha)})} ,
		\end{equation}	
		with $\frac{1}{\nu(\alpha)}+\frac{1}{\nu'(\alpha)}=1$.
		
		If $s\in(0,\frac32)$ and $ \alpha \in \left[ 0, \frac{2s}{3-2s} \right) $, then  \ref{ass2stri:nonlinearity_Lp_bound} holds with $\nu=\nu(\alpha)\in[2,\twoasts)$ for the local pure power non-linearity.
	\end{prop}
	\begin{proof}
		The first step is to use Minkowski inequality: observing that $ 1\le\nu^\prime(\alpha)\le2 $, we have
		\begin{align}
			\big\| \rho_{\ulu}^\alpha \ulu - \rho_{\ulv}^\alpha \ulv  \big\|_{\ell^2(L^{\nu^\prime(\alpha)})} 
			& \leq \big\| \| \rho_{\ulu}^\alpha \ulu - \rho_{\ulv}^\alpha \ulv \|_{\ell^2(\mathbb{C})} \big\|_{L^{\nu^\prime(\alpha)}} . \label{eq:Minkowski}
		\end{align}
		We are left to consider the $ \ell^2(\mathbb{C}) $ norm of the non-linearity for almost every $ x $ in $\mathbb{R}^3 $. To do this, we consider the pointwise estimate obtained in Lemma~\ref{lem:power_pointwise_estimate} and apply it to this term:
		\begin{align*}
			\| & \rho_{\ulu}^\alpha (x) \ulu (x) -  \rho_{\ulv}^\alpha (x) \ulv (x) \|_{\ell^2(\mathbb{C})}^2 = \sum_{k \geq 0}  | \rho_{\ulu}^\alpha (x) u_k (x) - \rho_{\ulv}^\alpha (x) v_k (x) |^2  \\
			&\leq  \sum_{k \geq 0}  \big( \rho_{{\ulv}}^\alpha (x) | u_k(x) - v_k(x) | + (2\alpha+2) ( \rho_{{\ulu}}^\alpha (x) + \rho_{{\ulv}}^\alpha (x) ) \rho_{{\ulu}-{\ulv}}^{\frac{1}{2}} (x) \frac{|u_k(x)|}{\rho_{{\ulu}}^{\frac{1}{2}}(x)} \big)^2 \\
			&\leq  2 \rho_{{\ulv}}^{2\alpha} (x) \sum_{k \geq 0}  | u_k(x) - v_k(x) |^2 + 2 (2\alpha+2)^2 ( \rho_{{\ulu}}^\alpha (x) + \rho_{{\ulv}}^\alpha (x) )^2 \frac{ \rho_{{\ulu}-{\ulv}} (x)}{\rho_{{\ulu}}(x)} \sum_{k \geq 0}  |u_k(x)|^2  \\
			&\leq  2 \rho_{{\ulv}}^{2\alpha} (x) \rho_{{\ulu}-{\ulv}}(x) +2 (2\alpha+2)^2 ( \rho_{{\ulu}}^\alpha (x) + \rho_{{\ulv}}^\alpha (x) )^2\rho_{{\ulu}-{\ulv}} (x)  \,,
		\end{align*}
		where the convexity inequality of the square and the definition of the density \eqref{def:density} have been used. Taking the square root and estimating, we get
		\begin{equation*}
			\| \rho_{\ulu}^\alpha (x) \ulu (x) - \rho_{\ulv}^\alpha (x) \ulv (x) \|_{\ell^2(\mathbb{C})} \leq 2(2\alpha+2)  \big( \rho_{{\ulu}}^\alpha (x) + \rho_{{\ulv}}^\alpha (x) \big) \rho_{{\ulu}-{\ulv}}^{\frac{1}{2}} (x) 
		\end{equation*}
		for almost every $ x$ in $\mathbb{R}^3 $.
		
		Next, we take the $L^{\nu^\prime(\alpha)} $ norm and use H{\"o}lder's inequality with 
		\begin{equation*}
			\frac1{\nu'(\alpha)} = \frac{1}{p(\alpha)} + \frac1{\nu(\alpha)}, \text{  that is  } p(\alpha) = \frac{1+\alpha}{ \alpha} .
		\end{equation*} 
		Together with \eqref{eq:Minkowski}, this yields
		\begin{align*}
			\big\| \rho_{\ulu}^\alpha \ulu - \rho_{\ulv}^\alpha  \ulv \big\|_{\ell^2(L^{\nu^\prime(\alpha)})} & \leq 2(2\alpha+2)\| \big( \rho_{{\ulu}}^\alpha  + \rho_{{\ulv}}^\alpha  \big) \rho_{{\ulu}-{\ulv}}^{\frac{1}{2}}  \|_{L^{\nu^\prime(\alpha)}}   \\
			& \leq
			2(2\alpha+2)  \big( \| \rho_{{\ulu}}^\alpha  \|_{L^{p(\alpha)}} +  \| \rho_{{\ulv}}^\alpha  \|_{L^{p(\alpha)}} \big) \| \rho_{{\ulu}-{\ulv}}^{\frac{1}{2}}  \|_{L^{\nu(\alpha)}}\\
			& =
			2(2\alpha+2)  \big( \| \rho_{{\ulu}}^\frac12  \|_{L^{\nu(\alpha)}}^{2\alpha} +  \| \rho_{{\ulv}}^\frac12  \|_{L^{\nu(\alpha)}}^{2\alpha} \big) \| \rho_{{\ulu}-{\ulv}}^{\frac{1}{2}}  \|_{L^{\nu(\alpha)}},
		\end{align*} 
		where in the equality we have used that $2\alpha p(\alpha)=\nu(\alpha)$.
		Using the estimate for the $L^q$ norm of the densities \eqref{est:densities_Lp}, we obtain \eqref{est:power_Lp}. 
		
		By the Sobolev embedding \eqref{eq:sobolev_embedding_plus}, observing that $ \nu(\alpha) \in \left[ 2 , \twoasts  \right) $, we see that \eqref{est:power_Lp} implies
		\begin{equation}
			\label{est:power_Hs}
			\big\| (\rho_{\ulu}^\alpha u_k - \rho_{\ulv}^\alpha v_k )_k \big\|_{\ell^2(L^{\nu^\prime(\alpha)})} \leq 2(2\alpha+2)  C_s^{2 \alpha} \big( \| {\ulu} \|_{\ell^2(H^s)}^{2\alpha} + \| {\ulv} \|_{\ell^2(H^s)}^{2\alpha} \big) \| {\ulu} - {\ulv} \|_{\ell^2(L^{\nu(\alpha)})},
		\end{equation}
		which in turn shows that \ref{ass2stri:nonlinearity_Lp_bound} holds.
	\end{proof}
	
	Now we turn to the verification of \ref{ass3:nonlinear_energy_derivative}. The energy of the local power non-linearity is the function
	\begin{equation}
		\mathscr{G}_{\mathrm{KS}}(\ulu) = \frac{\mu}{2(\alpha+1)} \int_{\mathbb{R}^3} \rho_{\ulu}^{\alpha +1} (x) \diff  x .
	\end{equation}
	Note that it follows from \eqref{est:densities_Lp} with $q=\nu(\alpha)=2(\alpha+1)$ that $\mathscr{G}_{\mathrm{KS}}(\ulu)$ is well-defined for~$\ulu$ in~$\ell^2(L^{\nu(\alpha)})\supset \ell^2(L^2\cap L^{\nu(\alpha)})$. 
	
	\begin{prop}
		Let $s\in(0,\frac32)$, $ \alpha \in \left[0, \frac{2s}{3-2s} \right) $ and $\nu=\nu(\alpha)$ be defined by \eqref{eq:nu(alpha)}. Then, for all~$ \ulu,\ulv\in\ell^2(L^{\nu}) $, 
		\begin{equation}
			\frac{\diff }{\diff \tau} \mathscr{G}_{\mathrm{KS}}(\ulu + \tau \ulv) |_{\tau = 0} = \mu \mathrm{Re} \langle \rho_{\ulu}^\alpha \ulu, \ulv \rangle_{\ell^2(L^{\nu^\prime}) , \ell^2(L^{\nu})} .
		\end{equation}
		Hence \ref{ass3:nonlinear_energy_derivative} holds with $\nu=\nu(\alpha)\in[2,\twoasts)$ for the local pure power non-linearity.
	\end{prop}
	\begin{proof}
		For a.e. $x\in\mathbb{R}^3$, we have 
		\begin{equation}
			\frac{\diff }{\diff \tau}\rho_{\ulu+\tau\ulv}^{\alpha +1} (x)=(\alpha + 1) \rho_{\ulu + \tau \ulv}^{\alpha} (x) \frac{\diff }{\diff \tau} \rho_{\ulu + \tau \ulv}(x),
		\end{equation}
		with 
		\begin{equation*} 
			\frac{\diff }{\diff \tau} \rho_{\ulu + \tau \ulv}(x) = \sum_{k \geq 0} ( 2 \mathrm{Re} ( \overline{u}_k(x) v_k(x) ) + 2 \tau | v_k |^2(x) ) = 2 \mathrm{Re} \langle \ulu(x) + \tau \ulv(x) , \ulv(x)  \rangle_{\ell^2(\mathbb{C})} ,
		\end{equation*}
		Therefore we can estimate, for $\tau\in(-1,1)$,
		\begin{align}
			\Big|\frac{\diff }{\diff \tau}\rho_{\ulu+\tau\ulv}^{\alpha +1} (x)\Big|\le2(\alpha + 1) \langle \rho_{|\ulu|+ |\ulv|}^{\alpha} (x)(|\ulu|(x) + |\ulv|(x)) , |\ulv|(x)  \rangle_{\ell^2(\mathbb{C})}.
		\end{align}
		We claim that the right-hand side is integrable over $\mathbb{R}^3$. Indeed we have
		\begin{multline*}
			\int_{\mathbb{R}^3}\langle \rho_{|\ulu|+ |\ulv|}^{\alpha} (x)(|\ulu|(x) + |\ulv|(x)) , |\ulv|(x)  \rangle_{\ell^2(\mathbb{C})}\diff x\\
			\le\big\|\|\rho_{|\ulu|+ |\ulv|}^{\alpha}(x)(|\ulu| + |\ulv|)(x)\|_{\ell^2(\mathbb{C})}\big\|_{L^{\nu'}}\big\|\|\ulv(x)\|_{\ell^2(\mathbb{C})}\big\|_{L^\nu}.
		\end{multline*}
		From the proof of Proposition \ref{lem:power_nonlinear_estimates}, we know that the first factor is finite, while Minkowski's inequality (since $\nu\ge2$) gives that $\big\|\|\ulv(x)\|_{\ell^2(\mathbb{C})}\big\|_{L^\nu}\le\|\ulv\|_{\ell^2(L^\nu)}<\infty$.
		
		We have thus proven that $\tau\mapsto \mathscr{G}_{\mathrm{KS}}(\ulu + \tau \ulv)$ is differentiable on $(-1,1)$, with
		\begin{align*}
			\frac{\diff }{\diff \tau} \mathscr{G}_{\mathrm{KS}}(\ulu + \tau \ulv) 
			& = \mu \int_{\mathbb{R}^3} \mathrm{Re} \langle \rho_{\ulu + \tau \ulv}^{\alpha} (x) (\ulu + \tau \ulv)(x) , \ulv(x)  \rangle_{\ell^2(\mathbb{C})} \diff  x \\
			& =  \mathrm{Re} \langle \ulg_{\mathrm{KS}}(\ulu + \tau \ulv), \ulv \rangle_{\ell^2(L^{\nu^\prime}) , \ell^2(L^{\nu}) } .
		\end{align*}
		This yields the result by setting $\tau=0$.
	\end{proof}

	\subsection{Hartree non-linearity}\label{subsec:convol}
	Recall that the classical correlation term in the density functional gives rise to a non-linear term of the Hartree type,    
	\begin{equation}\label{eq:GH}
		\ulg_{\mathrm{H}}(\ulu) = (w \ast \rho_{\ulu} )\ulu .
	\end{equation}
	Here, depending on the value of $s$, $ w $ can be either a real function or a finite signed Radon measure over $\mathbb{R}^3$ both assumed to be even. Note that in the latter case, taking the Dirac delta measure $w=\delta_0$, we recover a pure power non-linearity as in the previous subsection, with~$\alpha=1$.
	
	From the expression \eqref{eq:GH} of $\ulg_{\mathrm{H}}$, we obviously see that \ref{assump1} is satisfied. The next lemma will allow us to verify that \ref{ass2stri:nonlinearity_Lp_bound} holds under suitable conditions.
	
	Recall that $\mathcal{M}_0$ stands for the set of finite, signed Radon measures over $\mathbb{R}^3$, equipped with the total variation norm $\|\cdot\|_{\mathcal{M}_0}$.
	\begin{lemma}\label{lm:Hartree}
		Let $\nu\in[2,4]$. For all $ w \in L^{\frac{\nu}{2(\nu-2)}}$ and $\ulu, \, \ulv \in \ell^2(L^\nu) $, 
		\begin{equation}\label{est:Hartree_Lp_1}
			\| (w \ast \rho_{\ulu}) \ulu - (w \ast \rho_{\ulv}) \ulv \|_{\ell^2(L^{\nu^\prime})} \leq \frac{3}{2} \| w \|_{L^{\frac{\nu}{2(\nu-2)}}} ( \| \ulu  \|_{\ell^2(L^\nu)}^2 +\| \ulv \|_{\ell^2(L^{\nu})}^2  ) \| \ulu-\ulv \|_{\ell^2(L^{\nu})} .
		\end{equation}
		For all $ w \in \mathcal{M}_0$ and $\ulu, \, \ulv \in \ell^2(L^4) $,
		\begin{equation}\label{est:Hartree_Lp_2}
			\| (w \ast \rho_{\ulu}) \ulu - (w \ast \rho_{\ulv}) \ulv \|_{\ell^2(L^{4/3})} \leq \frac{3}{2} \| w \|_{\mathcal{M}_0} ( \| \ulu  \|_{\ell^2(L^4)}^2 +\| \ulv \|_{\ell^2(L^{4})}^2  ) \| \ulu-\ulv \|_{\ell^2(L^{4})} .
		\end{equation}
	\end{lemma}
	\begin{proof}
		We begin by proving \eqref{est:Hartree_Lp_1}. Consider $ \nu \in \left[ 2,  \twoasts  \right) $. Using first the H{\"o}lder inequality, as~$1/\nu^{\prime} = (\nu-2)/\nu + 1/\nu$, and then Young's inequality, as $ 1 + (\nu-2)/\nu = 2(\nu-2)/\nu + 2/\nu$, we obtain
		\begin{align*}
			\| (w \ast \rho_{\ulu}) &\ulu - (w \ast \rho_{\ulv})  \ulv \|_{\ell^2(L^{\nu^\prime})} \\
			& \leq \| w \ast \rho_{\ulu} \|_{L^{\frac{\nu}{\nu-2}}}  \| \ulu - \ulv \|_{\ell^2(L^{\nu})} + \| w \ast (\rho_{\ulu} - \rho_{\ulv}) \|_{L^{\frac{\nu}{\nu-2}}} \| \ulv \|_{\ell^2(L^{\nu})}  \\
			& \leq \| w \|_{L^{\frac{\nu}{2(\nu-2)}}} \| \rho_{\ulu} \|_{L^{\frac{\nu}{2}}} \| \ulu - \ulv \|_{\ell^2(L^{\nu})} + \| w \|_{L^{\frac{\nu}{2(\nu-2)}}} \| \rho_{\ulu} - \rho_{\ulv} \|_{L^{\frac{\nu}{2}}} \| \ulv \|_{\ell^2(L^{\nu})} .
		\end{align*}
		The identity 
		\begin{equation*}
			\rho_{\ulu} - \rho_{\ulv} = ( \rho_{\ulu}^{1/2} + \rho_{\ulv}^{1/2} ) ( \rho_{\ulu}^{1/2} - \rho_{\ulv}^{1/2} ), 
		\end{equation*}
		together with \eqref{est:pointwise_densities_1/2}, \eqref{est:densities_Lp} and the H{\"o}lder inequality, yields
		\begin{equation*}
			\| \rho_{\ulu} \|_{L^{\frac{\nu}{2}}} \leq \| \ulu  \|_{\ell^2(L^{\nu})}^2 , \quad \| \rho_{\ulu} - \rho_{\ulv} \|_{L^{\frac{\nu}{2}}} \leq (\| \ulu  \|_{\ell^2(L^{\nu})}^2 +\| \ulv \|_{\ell^2(L^{\nu})}^2  ) \| \ulu-\ulv \|_{\ell^2(L^{\nu})} .
		\end{equation*}
		Composing the estimates, we obtain \eqref{est:Hartree_Lp_1}.
		
		Now we prove \eqref{est:Hartree_Lp_2}. Consider $ s > \frac{3}{4} $, $ w \in \mathcal{M}_0 $ and take $\nu=4$. We proceed as in the previous case, the only difference being that the convolutions $ (w \ast f) (x) \coloneq \int_{\mathbb{R}^3} f (x-y) \diff w (y) $ are bounded linear operators on $L^p, \, p \in \left[1,\infty \right] $ as
		\begin{equation*}
			\| w \ast f \|_{L^p_x} \leq \int_{\mathbb{R}^3} \| f(x-y) \|_{L^p_y} \diff |w| (x) \leq \| w \|_{\mathcal{M}_0} \| f \|_{L^p},
		\end{equation*}
		by Minkowski inequality.
		Therefore,
		\begin{equation*}
			\| w \ast \rho_{\ulu} \|_{L^2} \leq \| w \|_{\mathcal{M}_0} \| \rho_{\ulu} \|_{L^2}, \quad \| w \ast (\rho_{\ulu} - \rho_{\ulv} ) \|_{L^2} \leq \| w \|_{\mathcal{M}_0} \| \rho_{\ulu} - \rho_{\ulv} \|_{L^2} .
		\end{equation*}
		Following the same steps as before, \eqref{est:Hartree_Lp_2} follows. 
	\end{proof}
	
	Using Lemma \ref{lm:Hartree}, it is not difficult to deduce the following proposition.
	
	\begin{prop}\label{prop:Hartree(N2)}$ $
		\begin{enumerate}[label=\roman*)]
			\item Let $s\in(\frac34,\frac32)$. For all $w=w_1+w_2\in\mathcal{M}_0+L^\infty$, $\ulg_{\mathrm{H}}=\ulg_{\mathrm{H},1}+\ulg_{\mathrm{H},2}$ satisfies \ref{ass2stri:nonlinearity_Lp_bound} with~$\nu_1=4$ and~$\nu_2=2$.
			\item Let $s\in(0,\frac34]$ and $\varepsilon>0$. For all $w=w_1+w_2\in L^{\frac{3}{4s}+\varepsilon}+L^\infty$, $\ulg_{\mathrm{H}}=\ulg_{\mathrm{H},1}+\ulg_{\mathrm{H},2}$ satisfies \ref{ass2stri:nonlinearity_Lp_bound} with 
			\begin{equation}\label{eq:nu(epsilon)}
				\nu_1=\nu(\varepsilon) = \frac{2(3+4s\varepsilon)}{3+4s\varepsilon-2s}
			\end{equation}
			and $\nu_2=2$.
		\end{enumerate}
	\end{prop}
	
	\begin{proof}
		In both cases, it suffices to split the Hartree non-linearity in two convolution terms of the form~$ \ulg_1(\ulu) = (w_{\mathrm{H},1} \ast \rho_{\ulu}) \ulu$ and $\ulg_2(\ulu) = (w_{\mathrm{H},2} \ast \rho_{\ulu}) \ulu $ and then apply the previous lemma, yielding the estimates 
		\begin{equation}    
			\| (w_2 \ast \rho_{\ulu}) \ulu - (w_2 \ast \rho_{\ulv}) \ulv \|_{\ell^2(L^{2})} \leq  3 C_{s}^2 \| w_2 \|_{L^\infty} ( \| \ulu  \|_{\ell^2(H^s)} +\| \ulv \|_{\ell^2(H^s)}  )^2 \| \ulu-\ulv \|_{\ell^2(L^{2})} ,
		\end{equation}
		and either
		\begin{equation}
			\| (w_1 \ast \rho_{\ulu}) \ulu - (w_1 \ast \rho_{\ulv}) \ulv \|_{\ell^2(L^{4/3})} \leq  3 C_{s}^2 \| w_1 \|_{\mathcal{M}_0} ( \| \ulu  \|_{\ell^2(H^s)} +\| \ulv \|_{\ell^2(H^s)}  )^2 \| \ulu-\ulv \|_{\ell^2(L^{4})}
		\end{equation}
		if $s\in(\frac34,\frac32)$, or
		\begin{equation}
			\| (w_1 \ast \rho_{\ulu}) \ulu - (w_1 \ast \rho_{\ulv}) \ulv \|_{\ell^2(L^{\nu(\varepsilon)^\prime})} \leq  3 C_{s}^2 \| w_1 \|_{L^{\frac{3}{4s}+\varepsilon}} ( \| \ulu  \|_{\ell^2(H^s)} +\| \ulv \|_{\ell^2(H^s)}  )^2 \| \ulu-\ulv \|_{\ell^2(L^{\nu(\varepsilon)})} ,\end{equation}
		if $s\in(0,\frac34]$.

	\end{proof}
	It remains to verify that condition \ref{ass3:nonlinear_energy_derivative} holds. The energy associated with the Hartree term is given by
	\begin{equation}
		\mathscr{G}_{\mathrm{H}}(\ulu) = \frac{1}{4} \int_{\mathbb{R}^3} (w \ast \rho_{\ulu})(x) \, \rho_{\ulu}(x) \diff x.
	\end{equation}
	From the proof of Lemma \ref{lm:Hartree}, it follows that $\mathscr{G}_{\mathrm{H}}(\ulu)$ is well-defined under the conditions of Proposition \ref{prop:Hartree(N2)}, namely that $w\in\mathcal{M}_0+L^\infty$ and $\ulu\in \ell^2(L^2\cap L^{4})$ if $s\in(\frac34,\frac32)$, or that~$w$ is in~$L^{\frac{3}{4s}+\varepsilon}+L^\infty$ and $\ulu\in \ell^2(L^2\cap L^{\nu(\varepsilon)})$ if $s\in(0,\frac34]$, with $\nu(\varepsilon)$ given by \eqref{eq:nu(epsilon)}. The next proposition shows that \ref{ass3:nonlinear_energy_derivative} holds under the same conditions.
	\begin{prop}
		Assume that $w$ is real and even. Suppose that either:
		\begin{enumerate}[label=\roman*)]
			\item $s\in(\frac34,\frac32)$, $\nu=4$, $w\in\mathcal{M}_0+L^\infty$ and $\ulu,\ulv\in \ell^2(L^2\cap L^{4})$, or
			\item $s\in(0,\frac34]$, $w\in L^{\frac{3}{4s}+\varepsilon}+L^\infty$ for some $\varepsilon>0$, and $\ulu,\ulv\in \ell^2(L^2\cap L^{\nu(\varepsilon)})$ with $\nu=\nu(\varepsilon)$ given by \eqref{eq:nu(epsilon)}
		\end{enumerate}
		holds. Then
		\begin{equation}\label{eq:diffHartree}
			\frac{\diff}{\diff \tau} \mathscr{G}_{\mathrm{H}}(\ulu + \tau \ulv) \Big|_{\tau=0} = \mathrm{Re} \langle (w \ast \rho_{\ulu} ) \ulu, \ulv \rangle_{\ell^2(L^2+L^{\nu^\prime}),\ell^2(L^2\cap L^\nu)} .
		\end{equation}
		In particular, \ref{ass3:nonlinear_energy_derivative} holds with the same $\nu$ as in Proposition \ref{prop:Hartree(N2)}.
	\end{prop}
	\begin{proof}
		Consider for instance the case $(ii)$. For a.e. $x\in\mathbb{R}^3$, we have
		\begin{equation*}
			\frac{\diff}{\diff \tau} \big( (w \ast \rho_{\ulu + \tau \ulv} )(x) \, \rho_{\ulu + \tau \ulv}(x) \big)
			= (w \ast \frac{\diff}{\diff \tau} \rho_{\ulu + \tau \ulv} )(x)  \, \rho_{\ulu + \tau \ulv}(x) + (w \ast \rho_{\ulu +\tau \ulv } )(x)  \, \frac{\diff}{\diff \tau} \rho_{\ulu + \tau \ulv} (x) ,
		\end{equation*}
		with $ \frac{\diff}{\diff \tau} \rho_{\ulu + \tau \ulv} (x) = 2 \mathrm{Re} \langle {\ulu} + \tau {\ulv} , \ulv\rangle_{\ell^2(\mathbb{C})} (x) $. For $\tau\in(-1,1)$, we can then estimate
		\begin{multline*}
			\Big|\frac{\diff}{\diff \tau} \big( (w \ast \rho_{\ulu + \tau \ulv} )(x)  \, \rho_{\ulu + \tau \ulv}(x) \big)\Big|\\
			\le 2\langle|w| \ast ( |\ulu| + |\ulv|) , |\ulv| \rangle_{\ell^2(\mathbb{C})} (x)  \, \rho_{|\ulu| + |\ulv|}(x) + 2(|w| \ast \rho_{|\ulu| + |\ulv| } )(x)  \, \langle |\ulu| + |\ulv| , |\ulv| \rangle_{\ell^2(\mathbb{C})} (x) .
		\end{multline*}
		Using Hölder and Young inequalities as in the proof of Lemma \ref{lm:Hartree}, we see that the right-hand side is an integrable function.
		Therefore,
		\begin{multline*}
			\frac{\diff}{\diff \tau} \mathscr{G}_{\mathrm{H}}(\ulu + \tau \ulv) \\
			= \frac{1}{2} \int_{\mathbb{R}^3} \big( (w \ast \mathrm{Re} \langle {\ulu} + \tau {\ulv} , \ulv \rangle_{\ell^2(\mathbb{C})} ) (x)  \, \rho_{\ulu + \tau \ulv}(x) + (w \ast \rho_{\ulu +\tau \ulv } )(x)  \, \mathrm{Re} \langle {\ulu} + \tau {\ulv} , \ulv\rangle_{\ell^2(\mathbb{C})} (x) \big) \diff x .
		\end{multline*}
		Using that $w$ is real and even we deduce that
		\begin{align*}
			\frac{\diff}{\diff \tau} \mathscr{G}_{\mathrm{H}}(\ulu + \tau \ulv)  &=  \int_{\mathbb{R}^3} (w \ast \rho_{\ulu +\tau \ulv } )(x)   \, \mathrm{Re} \langle {\ulu} + \tau {\ulv} , \ulv\rangle_{\ell^2(\mathbb{C})} (x) \diff x \\
			& = \mathrm{Re} \langle (w \ast \rho_{\ulu +\tau \ulv } ) ( \ulu +\tau \ulv) , \ulv \rangle_{\ell^2(L^2+L^{\nu^\prime}),\ell^2(L^2\cap L^{\nu})} ,
		\end{align*}
		which implies \eqref{eq:diffHartree}.
	\end{proof}
	
	\subsection{External potential}\label{subsec:ext}
	
	We consider here the contribution of an external potential, which we recall is treated as a ``non-linearity'' for convenience. Recall that
	\begin{equation*}
		\ulg_V (\ulu) := V\ulu
	\end{equation*}
	Obviously, $\ulg_V$ satisfies \ref{assump1}, and it is not difficult to verify that \ref{ass2stri:nonlinearity_Lp_bound} and \ref{ass3:nonlinear_energy_derivative} also hold for~$V$ in suitable $L^p$ spaces, as we explain in this subsection.
	
	The next lemma is proved by the H{\"o}lder inequality.
	\begin{lemma}
		Let $\nu\ge2$ and $V \in L^{\frac{\nu}{\nu-2}}$. Then, for all $\ulu,\ulv\in\ell^2(L^\nu)$, 
		\begin{equation}\label{est:ext_pot_Lp}
			\| V \ulu - V \ulv \|_{\ell^2(L^{\nu'})} \leq \| V \|_{L^{\frac{\nu}{\nu-2}}} \| \ulu - \ulv \|_{\ell^2(L^\nu)} . 
		\end{equation}
	\end{lemma}
	Applying this lemma we easily obtain the following proposition.
	\begin{prop}\label{prop:external}
		Let $s\in(0,\frac32)$ and $\varepsilon>0$. Let $V=V_1+V_2\in L^{\frac{3}{2s}+\varepsilon}+L^\infty$. Then~$\ulg_V=\ulg_{V_1}+\ulg_{V_2}$ satisfies~\ref{ass2stri:nonlinearity_Lp_bound} with
		\begin{equation}\label{eq:nu(epsilon)2}
			\nu_1=\nu(\varepsilon)=\frac{6+4s\varepsilon}{3+2s\varepsilon -2s}
		\end{equation}
		and $\nu_2=2$.
	\end{prop}
	The energy related to the external potential is simply
	\begin{equation}
		\mathscr{G}_V(\ulu) = \frac{1}{2} \int_{\mathbb{R}^3} V(x) \rho_{\ulu}(x) \diff x. 
	\end{equation}
	Using \eqref{est:densities_Lp} and Hölder's inequality, one verifies that $\mathscr{G}_V(\ulu)$ is well-defined under the conditions of Proposition \ref{prop:external}, for any $\ulu\in\ell^2(L^2\cap L^{\nu(\varepsilon)})$. The next proposition follows from direct computations similarly as in the previous subsections. 
	\begin{prop}
		Let $s\in(0,\frac32)$, $\varepsilon>0$ and $\nu=\nu(\varepsilon)$ be given by \eqref{eq:nu(epsilon)2}. Assume $V$ real and~$V\in L^{\frac{3}{2s}+\varepsilon}+L^\infty$. Then, for all $\ulu,\ulv\in\ell^2(L^2\cap L^\nu)$,
		\begin{equation}
			\frac{\diff}{\diff \tau} \mathscr{G}_V(\ulu + \tau \ulv) |_{\tau=0} = \mathrm{Re} \langle V \ulu, \ulv \rangle_{\ell^2(L^2+L^{\nu^\prime}),\ell^2(L^2\cap L^\nu)} .
		\end{equation}
		Hence \ref{ass3:nonlinear_energy_derivative} is satisfied with $\nu=\nu(\varepsilon)$.
	\end{prop}
	
	\subsection{The class \texorpdfstring{$\Nglob$}{Nglob}}\label{subsec:Nglob}
	In this subsection we show how the conditions imposed on the non-linear terms considered in the previous examples can be restricted in order to obtain non-linearities belonging the the subclass $\Nglob$. The argument is based on the following proposition.
	
	\begin{prop}\label{lem:global_ex_estim}
		Let $s\in(0,\frac32)$, $\nu\in[2,2_s^*)$ and $\ulg\in\Nlocnu$. 
		Let
		\begin{equation}\label{eq:cond-delta}
			0\le\delta<2\times\frac{\frac1\nu-\frac1{2_s^*}}{\frac12-\frac1\nu}
		\end{equation}
		and suppose that there is $K\ge0$ such that
		\begin{equation}\label{est:nonlinearity_global_estimate} 
			\| \ulg (\ulu) \|_{\ell^2(L^2+ L^{\nu'})} \leq K \| \ulu \|_{\ell^2(L^2\cap L^\nu)}^{1+\delta},
		\end{equation}
		for all $\ulu\in \ell^2(L^2\cap L^\nu)$. 
		Then for all $0<a<1$, there are $C=C(a,\delta,s,\nu)\ge0$ and $q=q(\delta,s,\nu)\ge1$ such that 
		\begin{equation}
			|\mathscr{G}(\ulu)|  \leq \frac{a}{2} \| \ulu  \|_{\ell^2(H^{s})}^2 + C  ( \| \ulu  \|_{\ell^2(L^2)}^{q} + | \mathscr{G}(\boldsymbol{0}) | )
		\end{equation}
		for all $\ulu\in \ell^2(H^s)$. In particular, $\ulg\in\Nglob$. 
	\end{prop}

	Before proving this proposition, we apply it to the examples considered in the previous subsections. 
	\begin{corollary}
		Let $s\in(0,\frac32)$. 
		\begin{enumerate}[label=\roman*)]
			\item If ($\mu\geq 0$ and $0<\alpha<\frac{2s}{3-2s}$) or ($\mu<0$ and $0<\alpha < \frac{2s}{3} $), then $\ulg_{\mathrm{KS}}\in\Nglob$,
			\item Assume $w$ is real and even. If ($w^{+} \in L^{\frac{3^+}{4s}} + L^{\infty}$ for~$s\leq 3/4$ or $w^{+} \in L^{\frac{3^+}{4s}} + L^{\infty}$ for~$s> 3/4$) and $w^{-} \in L^{\frac{3^+}{2s}} + L^{\infty}$, then $\ulg_{\mathrm{H}}\in\Nglob$.
			\item Assume $V$ real. If $V \in L^{\frac{3^+}{2s}} + L^{\infty}$, then $\ulg_V\in\Nglob$.
		\end{enumerate}
	\end{corollary}
	\begin{proof}
		It suffices to apply Proposition \ref{lem:global_ex_estim}, observing that the assumptions of that proposition hold in each case. Indeed,
		\begin{enumerate}[label=\roman*)]
			\item  For the pure power non-linearity, by estimate \eqref{est:power_Lp}, the assumptions of Proposition \ref{lem:global_ex_estim} hold with $\nu=2(\alpha+1)$ and $\delta = 2 \alpha $, provided that $\alpha < \frac{2s}{3} $. 
			\item For the Hartree non-linearity, by estimate \eqref{est:Hartree_Lp_1}, the assumptions of Proposition \ref{lem:global_ex_estim} hold with $\nu=\nu(\varepsilon)$ (with $\varepsilon>0$, see \eqref{eq:nu(epsilon)}) and $\delta=2$, provided that
			$w \in L^{\frac{3}{2s}+\varepsilon} + L^{\infty}.$
			\item For the external potential, by estimate \eqref{est:ext_pot_Lp}, the assumptions of Proposition \ref{lem:global_ex_estim} hold with  $\nu=\nu(\varepsilon)$ (with $\varepsilon>0$, see \eqref{eq:nu(epsilon)2}) and $\delta = 0$, provided that
			$V \in L^{\frac{3}{2s}+\varepsilon} + L^{\infty}.$
		\end{enumerate}
	\end{proof}

	Now we turn to the proof of Proposition \ref{lem:global_ex_estim}
	\begin{proof}[Proof of Proposition \ref{lem:global_ex_estim}]
		Using \ref{ass3:nonlinear_energy_derivative} we can estimate 
		\begin{multline}
			| \mathscr{G}(\ulu) | \leq | \mathscr{G}(\boldsymbol{0}) | + \int_{0}^1 \big|\mathrm{Re} \langle \ulg(\tau \ulu) , \ulu \rangle_{\ell^2(L^2+L^{\nu^\prime}),\ell^2(L^2\cap L^{\nu})} \big| \diff \tau \\
			\leq | \mathscr{G}(\boldsymbol{0}) | + \sup_{\tau \in (0,1)} \| \ulg(\tau \ulu) \|_{\ell^2(L^2+L^{\nu^\prime})} \| \ulu \|_{\ell^2(L^2\cap L^{\nu})}\,.
		\end{multline}
		Therefore \eqref{est:nonlinearity_global_estimate} yields
		\begin{equation}
			| \mathscr{G}(\ulu) | \leq | \mathscr{G}(\boldsymbol{0}) | + K \| \ulu \|_{\ell^2(L^2\cap L^\nu)}^{2+\delta} .
		\end{equation}
		Applying Lemma \ref{lem:interpolation_ineq} below, more precisely the estimate \eqref{eq:bound-hoelder-sobolev-general-r}, we obtain
		\begin{equation}
			| \mathscr{G}(\ulu) | \leq | \mathscr{G}(\boldsymbol{0}) | + K \| \ulu \|_{\ell^2(H^s)}^{(2+\delta)\frac{\twoasts }{\twoasts -2} \frac{\nu-2}{\nu} } \| \ulu \|_{\ell^2(L^2)}^{(2+\delta)\frac{2}{\twoasts -2}\frac{\twoasts -\nu}{\nu} } .
		\end{equation}
		The condition \eqref{eq:cond-delta} implies that
		\begin{equation}
			b=b(\delta,s,\nu) \coloneq (2+\delta)\frac{\twoasts }{\twoasts -2} \frac{\nu-2}{\nu} < 2,
		\end{equation}
		which allows us to estimate, by Young's inequality,
		\begin{equation}
			| \mathscr{G}(\ulu) | \leq \frac{a}{2} \| \ulu \|_{\ell^2(H^s)}^2 + ( C \| \ulu \|_{\ell^2(L^2)}^{q} + | \mathscr{G}(\boldsymbol{0}) | ),
		\end{equation}
		where 
		\begin{align*} 
			&C = C(a,\delta,s,\nu) = K^{\frac{2}{2-b}}a^{-\frac{b}{2-b}} b^{\frac{b}{2-b}} , \\
			&q = q(\delta,s,\nu) = (2+\delta)\frac{2}{2-b}\frac{2}{\twoasts -2} \frac{\twoasts -\nu}{\nu} .
		\end{align*}
		We have proven that condition \ref{ass4:nonlinearity_global_estimate} holds with the locally bounded function
		\begin{equation}
			D(x) = Cx^{q(\delta,s,\nu)} + | \mathscr{G}(\boldsymbol{0}) | .
		\end{equation}
	\end{proof}

	\subsection{The gauge condition}\label{subsec:gauge} 
	
	The gauge condition \eqref{eq:gauge} is satisfied in all the examples considered here. In fact, these arise from density functionals, and 
	\begin{equation}
		\forall k \in \mathbb{N}, \;\ g_k(\ulu) = \tilde{g}(\rho_{\ulu}) u_k,
	\end{equation}
	for some function $\tilde{g} : L^\nu \rightarrow L^{\frac{\nu}{\nu-2}} $ where $\nu$ is the associated coefficient appearing in Condition~\ref{ass2stri:nonlinearity_Lp_bound}. This function can be verified to be well defined by the same estimates used in the previous sections. Therefore for any $t \in \mathbb{R}$ the relation 
	\begin{equation}
		\ulg(e^{it}\ulu) = \tilde{g}(\rho_{e^{it}\ulu}) e^{it} \ulu = e^{it} \tilde{g}(\rho_{\ulu}) \ulu 
	\end{equation}
	is proven by the invariance of the density $ \rho_{e^{it} \ulu} = \rho_{\ulu} $ with respect to the gauge transformation. 
	
	\section{Preliminary estimates}\label{sec:prelim}
	
	In this section we derive preliminary results that will be important tools in our proof of the existence of solutions to the Kohn-Sham equations. We begin with estimates on non-linearities belonging to the class $\Nloc$ in Subsection \ref{subsec:control-NL}. Next, in Subsection \ref{subsec:regul}, we introduce a regularization operator and establish several of its properties.  
	
	\subsection{Control of the non-linearities}\label{subsec:control-NL}
	The non-linearities we consider are constrained to be controlled by the kinetic energy term which is encoded in condition \ref{ass2:nonlinearity_Lp_bound}.
	In practice a sufficient condition is that $ \nu \in \left[ 2 , \twoasts  \right) $, $ \twoasts  = \frac{6}{3-2s}$, as implied by the Sobolev embeddings \eqref{eq:sobolev_embedding_plus}--\eqref{eq:sobolev_embedding_minus}. Here we prove some basic estimates in the norms of~$ \ell^2(L^2\cap L^\nu) $ and $ \ell^2(H^s) $ .

	Recall that the notations $\psnu$, $\eta_{s,\nu}$, $\thetasnu$ have been introduced in \eqref{eq:psnu}--\eqref{eq:eta-thetasnu}. 
	
	\begin{lemma}\label{lem:interpolation_ineq}
		Let $s\in(0,\frac32)$ and $ \nu \in [2,\twoasts )$. Then, for all $\ulf  \in \ell^2(H^s) $,
		\begin{equation}\label{eq:bound-hoelder-sobolev-general-r}
			\| \ulf  \|_{\ell^2(L^{\nu})} 
			\leq C_s^{1/\psnu'} \, \| \ulf  \|_{\ell^2(H^{s})}^{1/\psnu'} \, \| \ulf  \|_{\ell^2(L^2)}^{1/\psnu}
		\end{equation}
		and
		\begin{equation}\label{eq:bound-hoelder-sobolev-general-nu-and-two}
			\| \ulf  \|_{\ell^2(L^2\cap L^{\nu})} 
			\leq \sqrt{2}\,\eta_{s,\nu}\big(C_s \| \ulf\|_{\ell^2(H^{s})}\big) \, \thetasnu \big(\| \ulf  \|_{\ell^2(L^2)}\big)\,.
		\end{equation}
	\end{lemma}
	
	\begin{remark}
		In our application we need a strictly positive $1/\psnu$ exponent to be able to obtain some useful estimates when the kinetic energy of the sequences is assumed to be bounded. This is why we do not use the endpoint Sobolev embedding. 
	\end{remark}
	\begin{proof}[Proof of Lemma~\ref{lem:interpolation_ineq}]
		By the H\"older inequality, we have, for $ 2 \leq \nu < \twoasts  $ and for all $f\in L^2$,
		\begin{equation*}
			\| f \|_{L^{\nu}} \leq \| f \|_{L^{\twoasts }}^{1/\psnu'} \| f \|_{L^2}^{1/\psnu} .
		\end{equation*}
		Composed with the Sobolev embedding, this yields
		\begin{equation*}
			\| f \|_{L^{\nu}} \leq   C_s^{1/\psnu'} \| f \|_{H^{s}}^{1/\psnu'} \| f \|_{L^2}^{1/\psnu} . 
		\end{equation*}
		This inequality can be lifted to the $ \ell^2(L^p) $ norms by the H{\"o}lder inequality in $ \ell^2(\mathbb{C}) $: As $ \psnu $ and~$\psnu' $ are H{\"o}lder conjugate coefficients, we can estimate
		\begin{align*}
			\| \ulf  \|_{\ell^2(L^{\nu})} &\leq  \Big( \sum_{k \geq 0} C_s^{2/\psnu'} \| f_k \|_{H^{s}}^{2/\psnu} \| f_k \|_{L^2}^{2/\psnu} \Big)^{\frac{1}{2}}  \nonumber \\
			&\leq    C_s^{1/\psnu'} \| \ulf  \|_{\ell^2(H^{s})}^{1/\psnu'}  \| \ulf  \|_{\ell^2(L^2)}^{1/\psnu}
		\end{align*}
		which yields \eqref{eq:bound-hoelder-sobolev-general-r}.
		
		To estimate $\| \ulf  \|_{\ell^2(L^2\cap L^{\nu})}$, we use
		\begin{equation*}
			\| \ulf  \|_{\ell^2(L^2\cap L^{\nu})}
			\leq \sqrt{2} \|\ulf\|_{\ell^2(L^2)} +\sqrt{2} \|\ulf\|_{\ell^2(L^\nu)}
		\end{equation*}
		along with \eqref{eq:bound-hoelder-sobolev-general-r}, which implies \eqref{eq:bound-hoelder-sobolev-general-nu-and-two}.
	\end{proof}
	
	Now we turn to an Hölder continuity estimate in time for functions with values in $\ell^2(H^s)$.
	
	\begin{lemma}\label{lem:estimate_holder_t}
		Let $s\in(0,\frac32)$. Consider any interval $ I \subset \mathbb{R} $ and any sequence $\ulu $ that depends on~$t$ in~$I$ such that
		$${\ulu} \in L^\infty\big(I; \ell^2(H^{s})\big) \cap W^{1,\infty} \big(I; \ell^2(H^{-s})\big). $$ Then, for all $t$ and $t^\prime$ in $I$, 
		\begin{equation}\label{est:holder_time}
			\| {\ulu}(t') -  {\ulu}(t) \|_{\ell^2(L^{2})} 
			\leq  \sqrt{2} \max\big\{ \|{\ulu}\|_{L^\infty(I;\ell^2(H^{s}))} , \| \partial_t {\ulu}\|_{L^\infty(I;\ell^2(H^{-s}))} \big\}  |t'-t|^{\frac{1}{2}} .
		\end{equation}
	\end{lemma}
	\begin{proof}
		As $ \ulu$ belongs to $W^{1,\infty}(I;\ell^2(H^{-s})) $ we can write the fundamental theorem of calculus formula for all $t$ and $t'$ in $I$: 
		\begin{equation*}
			\ulu(t') -\ulu(t) = \int_{t}^{t'}  \partial_\tau \ulu (\tau) \diff  \tau  .
		\end{equation*}
		Recall the duality estimate
		\begin{equation*}
			\| f \|_{L^2} \leq \| f \|_{H^{s}}^{1/2} \| f \|_{H^{-s}}^{1/2} ,
		\end{equation*}
		for all $f\in H^s$, and observe that a similar estimate holds for any $ \ulf $ in $\ell^2(H^{s}) $:
		\begin{equation*}
			\| \ulf  \|_{\ell^2(L^2)} \leq \Big( \sum_{k \geq 0}  \| f_k \|_{H^{s}} \| f_k \|_{H^{-s}} \Big)^{\frac{1}{2}} \leq \| \ulf  \|_{\ell^2(H^{s})}^{\frac{1}{2}} \| \ulf  \|_{\ell^2(H^{-s})}^{\frac{1}{2}} ,
		\end{equation*}
		by the previous estimate and the Cauchy-Schwarz inequality in $ \ell^2$. 
		By making use of this inequality and the fundamental theorem of calculus above, we get
		\begin{align*}
			\| \ulu(t')-\ulu(t) \|_{\ell^2(L^2)} &\leq \| \ulu(t')-\ulu(t) \|_{\ell^2(H^{s})}^{\frac{1}{2}} \| \ulu(t')-\ulu(t) \|_{\ell^2(H^{-s})}^{\frac{1}{2}} \\
			&\leq \sqrt{2} \| \ulu \|_{L^{\infty}(I;\ell^2(H^{s}))}^{\frac{1}{2}} \Big\| \int_{t}^{t'}   \partial_\tau \ulu(\tau)\diff  \tau \Big\|_{\ell^2(H^{-s})}^{\frac{1}{2}} \\
			&\leq \sqrt{2} \| \ulu \|_{L^{\infty}(I;\ell^2(H^{s}))}^{\frac{1}{2}} \, \| \partial_{\tau} \ulu \|_{L^{\infty}(I;\ell^2(H^{-s}))}^{\frac{1}{2}} \, |t'-t|^{\frac{1}{2}} ,
		\end{align*}
		which implies \eqref{est:holder_time}.
	\end{proof}
	
	The next lemma provides a useful estimate for non-linearities in the class $\Nloc$.
	
	\begin{lemma}\label{lem:hoelder-regularity-non-linearity}Let $s\in(0,\frac32)$, $\nu\in[2,\twoasts)$ and $\ulg\in\Nlocnu$. 
		Then, for any $M>0$, there exists~$C=C(s,\nu,M)>0$
		such that, for all sequences $ \ulu$ and $\ulv$ in $B_{\ell^2(H^{s})}(0,M)$,
		\begin{equation}\label{est:Lnu_L2_norms}
			\| \ulg (\ulu) -\ulg (\ulv) \|_{\ell^2(L^2+L^{\nu^\prime})} \leq C \, \thetasnu \big(\| \ulu-\ulv \|_{\ell^2(L^2)}\big)
		\end{equation}
		and
		\begin{equation}\label{est:nonlinear_energy_L2}
			| \mathscr{G}(\ulu) - \mathscr{G}(\ulv) | \leq C \, \thetasnu \big(\| \ulu-\ulv \|_{\ell^2(L^2)}\big) .
		\end{equation}
	\end{lemma}
	\begin{proof}
		Using \ref{ass2:nonlinearity_Lp_bound} and composing with \eqref{eq:bound-hoelder-sobolev-general-nu-and-two}, we obtain
		\begin{align*}
			\big\| \ulg(\ulu)-\ulg(\ulv) \big\|_{\ell^2(L^2 +L^{\nu^\prime})} 
			&\leq   \Clip{M}\,   \eta_{s,\nu}(C_s \|  \ulu-\ulv \|_{\ell^2(H^{s})}) \, \thetasnu (\|  \ulu-\ulv  \|_{\ell^2(L^2)}) \\
			& \leq \Clip{M}\,\eta_{s,\nu}(C_s 2M) \, \thetasnu (\|  \ulu-\ulv \|_{\ell^2(L^2)}) \,.
		\end{align*}
		This proves \eqref{est:Lnu_L2_norms}.
		
		The second estimate is obtained similarly by using the fact that $\mathscr{G}$ is Gateau differentiable with differential $\ulg $ by \ref{ass3:nonlinear_energy_derivative}:
		\begin{align*}
			\big| \mathscr{G}(\ulu) -\mathscr{G}(\ulv) \big| 
			& =  \Big| \int_{0}^1 \frac{\diff}{\diff \tau}  \mathscr{G}( \tau \ulu + (1-\tau) \ulv) \diff  \tau \Big|  \\
			& = \Big| \int_{0}^1  \mathrm{Re}\langle \ulu -\ulv , \ulg(\tau \ulu + (1-\tau) \ulv ) \rangle_{\ell^2(L^2\cap L^{\nu}) , \ell^2(L^2+L^{\nu^\prime})}\diff  \tau \Big|  \\
			&\leq  \sup_{ \tau \in [0,1]} \| \ulg(\tau \ulu + (1-\tau) \ulv ) \|_{\ell^2(L^2+L^{\nu^\prime})} \, \|\ulu-\ulv \|_{\ell^2(L^2\cap L^{\nu})} ,
		\end{align*}
		where we used the duality product between $L^2+L^{\nu^\prime}$ and $L^2\cap L^{\nu} $ and the Cauchy-Schwarz inequality. By the non-linearity estimate \ref{ass2:nonlinearity_Lp_bound} (with the second sequence identically zero) and the bound
		\[\| \tau \ulu + (1-\tau) \ulv \|_{\ell^2(H^s)} \leq \tau\|  \ulu  \|_{\ell^2(H^s)}+ (1-\tau)\|  \ulv \|_{\ell^2(H^s)}\leq \tau M +  (1-\tau)M = M\,,\]
		we obtain
		%
		\begin{equation*}
			\| \ulg(\tau \ulu + (1-\tau) \ulv ) \|_{ \ell^2(L^2+L^{\nu^\prime})} 
			\leq  \Clip{M}  \| \tau \ulu + (1-\tau) \ulv \|_{\ell^2(L^2\cap L^{\nu})}  
			\leq  C_s M \Clip{M} ,
		\end{equation*}
		where we have used the Sobolev embedding and the fact that $ 0\leq \tau \leq 1$. The $\ell^2(L^2\cap L^{\nu})$ norm of~$\ulu-\ulv$ can be estimated by \eqref{eq:bound-hoelder-sobolev-general-nu-and-two} to get
		\begin{align*}
			\|\ulu-\ulv \|_{\ell^2(L^2\cap L^{\nu})} 
			&\leq   \eta_{s,\nu}(C_{s} \| \ulu - \ulv \|_{\ell^2(H^{s})}) \, \thetasnu ( \| \ulu - \ulv \|_{\ell^2(L^2)}) \\
			&\leq  \eta_{s,\nu}( C_{s} \,2M) \, \thetasnu ( \| \ulu-\ulv \|_{\ell^2(L^2)}) .
		\end{align*}
		Composing these estimates, we obtain
		\begin{equation*}
			\big| \mathscr{G}(\ulu) -\mathscr{G}(\ulv) \big| 
			\leq    C_{s} M \Clip{M}\,\eta_{s,\nu}( C_s\, 2M ) \, \thetasnu ( \| \ulu-\ulv \|_{\ell^2(L^2)}) .\label{estimate_energy_difference}
		\end{equation*}
		which implies \eqref{est:nonlinear_energy_L2}.
	\end{proof}
	
	
	\begin{remark}
		The class of non-linearities considered here vanish at $0$, $\ulg(\ulzero) = \ulzero $. Considering the previous estimates with $ \ulv = \ulzero $ and $ \|\ulu\|_{\ell^2(H^s)}\leq M $, we obtain the estimates
		\begin{align}
			\label{est:single_nonlinearity_Lp}
			& \| \ulg (\ulu)  \|_{\ell^2(L^2+L^{\nu^\prime})} \leq \Clip{M} \, \| \ulu \|_{\ell^2(L^2\cap L^\nu)}. 
		\end{align}
	\end{remark}

	\subsection{Regularization operator}\label{subsec:regul}
	
	
	Our construction of weak solutions to the Kohn-Sham equation relies on a regularization procedure adapted from e.g. \cite{cazenave1988cauchy,strauss1970weak}. 
	We introduce a family of operators with a number of properties that we recall in this subsection,  adapting \cite[Section 6.1.2]{grafakos2009modern} in a straightforward way. 
	
	For $s\in(0,\frac32)$, we define the regularization operator $R^{[n]}:L^p\to L^p$, for all $n\in\mathbb{N}, \, n \geq 1$ and $p\in[1,\infty]$, by setting
	\begin{equation}
		\forall f \in L^p, \quad (R^{[n]}f)(x) := ( \mathcal{R}^{[n]} \ast f ) (x) , \quad \mathcal{R}^{[n]}(y) := \frac{n^{3/2}}{2^{3} \pi^{3/2} \Gamma(\frac{s}{2})} \int_{0}^{\infty}  e^{-\tau} \tau^{-\frac{5-s}{2}} e^{-\frac{n}{4 \tau} |y|^2} \diff \tau .
	\end{equation}
	In the next lemma we recall that $R^{[n]}$ is indeed a well-defined bounded operator on $L^p$ for any~$p\in[1,\infty]$.

	\begin{prop}\label{prop:restriction-operator-on-Lp}
		Let $s\in(0,\frac32)$ and $p\in[1,\infty]$. For all $n\geq1$, the operators $R^{[n]}$ are uniformly bounded in $L^p$,
		\begin{equation}
			\label{est:reg_Lp_simple}
			\forall n \geq 1, \,\ \| R^{[n]} \|_{\mathcal{B}(L^p,L^p)} \leq 1 .
		\end{equation}
		Moreover, for $p \neq \infty$, they converge strongly to the identity as $n \rightarrow \infty $,
		\begin{equation}
			\label{lim:reg_strong_conv_simple}
			\forall f \in L^p, \quad \| (R^{[n]} - \mathbb{I} ) f \|_{L^p} \xrightarrow[n \rightarrow \infty]{} 0 .
		\end{equation}
	\end{prop}
	\begin{remark}
		Eq.~\eqref{est:reg_Lp_simple} implies both
		\begin{equation}
			\label{est:reg_Lp}
			\forall n \geq 1, \,\ \| R^{[n]} \|_{\mathcal{B}(L^2\cap L^p,L^2\cap L^p)} \leq 1 , \quad \text{and} \quad \| R^{[n]} \|_{\mathcal{B}(L^2+ L^p,L^2+ L^p)} \leq 1
		\end{equation}
		while \eqref{lim:reg_strong_conv_simple} yields, for $p\neq \infty$,
		\begin{equation}
			\label{lim:reg_strong_conv_cap}
			\forall f \in L^2\cap L^p, \, \| (R^{[n]} - \mathbb{I} ) f \|_{L^2\cap L^p} 
			\xrightarrow[n \rightarrow \infty]{} 0
		\end{equation}
		and
		\begin{equation}
			\quad \forall  f \in L^2+ L^p, \, \| (R^{[n]} - \mathbb{I} ) f \|_{L^2+L^p} 
			\xrightarrow[n \rightarrow \infty]{} 0.\label{lim:reg_strong_conv_plus}
		\end{equation}
	\end{remark}
	\begin{proof}[Proof of Proposition \ref{prop:restriction-operator-on-Lp}]
		One can derive the uniform bound in norm, Eq.~\eqref{est:reg_Lp_simple}, from Young's inequality: 
		\begin{align*}
			\| (R^{[n]}f) \|_{L^p} &\leq  \frac{n^{3/2}}{2^{3} \pi^{3/2} \Gamma(\frac{s}{2})} \int_{0}^{\infty}  e^{-\tau} \tau^{-\frac{5-s}{2}} \| e^{-\frac{n}{4 \tau} |\cdot|^2} \ast f \|_{L^p} \diff \tau \\
			&\leq  \Bigg( \frac{n^{3/2}}{2^{3} \pi^{3/2} \Gamma(\frac{s}{2})} \int_{0}^{\infty}  e^{-\tau} \tau^{-\frac{5-s}{2}} \| e^{-\frac{n}{4 \tau} |\cdot|^2}  \|_{L^1} \diff \tau \Bigg) \| f \|_{L^p}  \\
			&\leq  \Bigg( \frac{n^{3/2}}{2^{3} \pi^{3/2} \Gamma(\frac{s}{2})} \int_{0}^{\infty}  e^{-\tau} \tau^{-\frac{5-s}{2}} \big( \frac{4 \tau \pi}{n}\big)^{\frac{3}{2}} \diff \tau \Bigg) \| f \|_{L^p}  \\
			&=  \| f \|_{L^p} .
		\end{align*}
		
		The strong convergence in the sense of operators, Eq.~\eqref{lim:reg_strong_conv_simple}, follows from the fact that the~$\mathcal{R}^{[n]}$ are approximations of the identity, being superpositions of peaked gaussians: We have that
		\begin{equation*}
			\int_{\mathbb{R}^3} \mathcal{R}^{[n]}(y) \diff  y = (2 \pi)^{3/2} , 
		\end{equation*}
		and therefore, for a.e. $x\in\mathbb{R}^3$,
		\begin{align*}
			((R^{[n]}-\mathbb{I})f)(x) = & \frac{1}{(2 \pi)^{3/2} } \int_{\mathbb{R}^3}  \mathcal{R}^{[n]}(y) f(x-y)\diff  y - \frac{1}{(2 \pi)^{3/2} } \int_{\mathbb{R}^3} \mathcal{R}^{[n]}(y) f(x) \diff  y  \\
			= & \frac{1}{(2 \pi)^{3/2} } \int_{\mathbb{R}^3}  \mathcal{R}^{[n]}(y) (f(x-y)-f(x)) \diff  y.
		\end{align*}
		Recall that $ p \in [1,\infty) $ and $ \mathcal{R}^{[n]}(y) >0 $ for all $y\in\mathbb{R}^3$. Applying the Minkowski inequality, we obtain
		\begin{align*}
			\| (R^{[n]}-\mathbb{I})f \|_{L^p} \leq & \frac{1}{(2 \pi)^{3/2} } \big\| \| \mathcal{R}^{[n]}(y) |f(x-y)-f(x)| \|_{L^1_y} \big\|_{L^p_x}  \\
			\leq & \frac{1}{(2 \pi)^{3/2} } \big\| \| \mathcal{R}^{[n]}(y) |f(x-y)-f(x)| \|_{L^p_x} \big\|_{L^1_y}  \\
			\leq & \int_{\mathbb{R}^3}  \frac{\mathcal{R}^{[n]}(y)}{(2 \pi)^{3/2} } \Big( \int_{\mathbb{R}^3} |f(x-y) -f(x) |^p \diff x \Big)^{\frac{1}{p}} \diff y .
		\end{align*}
		We consider any $ \varepsilon >0 $ and a ball $ B_\delta(0)\subset\mathbb{R}^3$, with $\delta >0$ depending on $\varepsilon$ such that, by continuity of the translation operator in the $L^p$ space,
		\begin{equation*}
			\forall y \in B_\delta(0), \;\ \| f(\cdot-y)-f(\cdot) \|_{L^p} = \Big( \int_{\mathbb{R}^3} |f(x-y) -f(x) |^p \diff x \Big)^{\frac{1}{p}} < \varepsilon .
		\end{equation*}
		We split the integral in the $y$ variable in the integral over the ball $ B_\delta(0)$ and on the complement set $ \mathbb{R}^3\setminus B_\delta(0) $, to get
		\begin{align}
			\| (R^{[n]}-\mathbb{I})f \|_{L^p} &\leq   \int_{B_\delta(0)}  \frac{\mathcal{R}^{[n]}(y)}{(2 \pi)^{3/2} } \| f(x-y)-f(x) \|_{L^p} \diff y \notag \\
			& \quad + \int_{\mathbb{R}^3\setminus B_\delta(0)}  \frac{\mathcal{R}^{[n]}(y)}{(2 \pi)^{3/2} } \| f(x-y) - f(x) \|_{L^p} \diff y \notag \\
			&\leq  \varepsilon + 2 \| f \|_{L^p} \int_{\mathbb{R}^3\setminus B_\delta(0)} \frac{\mathcal{R}^{[n]}(y)}{(2 \pi)^{3/2} } \diff y. \label{eq:a0}
		\end{align}
		We prove that the integral of $\mathcal{R}^{[n]}(y)$ over the complement of the ball converges to 0 as $ n \rightarrow  \infty $, by direct computation,
		\begin{align}
			\int_{\mathbb{R}^3\setminus B_\delta(0)}  \frac{\mathcal{R}^{[n]}(y)}{(2 \pi)^{3/2} } \diff y & =  \frac{n^{3/2}}{2^{3} \pi^{3/2} \Gamma(\frac{s}{2})} \int_{0}^{\infty} e^{-\tau} \tau^{-\frac{5-s}{2}}  \diff \tau  \int_{\mathbb{R}^3\setminus B_\delta(0)}  e^{-\frac{n}{4 \tau} |y|^2}  \diff y \\
			& =  \frac{n^{3/2}}{2^{3} \pi^{3/2} \Gamma(\frac{s}{2})} \int_{0}^{\infty}  e^{-\tau} \tau^{-\frac{5-s}{2}} \diff \tau \, 4 \pi \int_{\delta}^{\infty}  R^2 e^{-\frac{n}{4 \tau} R^2} \diff R \\
			& =  \frac{n^{3/2}}{2^{3} \pi^{3/2} \Gamma(\frac{s}{2})} \int_{0}^{\infty}  e^{-\tau} \tau^{-\frac{5-s}{2}} \diff \tau\, 4 \pi \int_{ \big( \frac{n}{4 \tau} \big)^{1/2} \delta}^{\infty} \Big( \frac{n}{4 \tau} \Big)^{-\frac{3}{2}} \tilde{R}^2 e^{-\tilde{R}^2}  \diff \tilde{R} \\
			& =  \frac{4}{ \pi^{1/2} \Gamma(\frac{s}{2})} \int_{0}^{\infty}  e^{-\tau} \tau^{-\frac{2-s}{2}} \diff \tau\int_{ \big( \frac{n}{4 \tau} \big)^{1/2} \delta}^{\infty}  \tilde{R}^2 e^{-\tilde{R}^2}\diff \tilde{R} .
		\end{align}
		Notice that as a sequence, 
		\begin{equation}
			\int_{ \big( \frac{n}{4 \tau} \big)^{1/2} \delta}^{\infty}  \tilde{R}^2 e^{-\tilde{R}^2} \diff \tilde{R}
		\end{equation}
		converges pointwise to $0$ as $ n \rightarrow \infty $ and the integrand is uniformly bounded by
		\begin{equation}
			\frac{4}{\pi^{1/2} \Gamma(\frac{s}{2})} e^{-\tau} \tau^{-\frac{2-s}{2}}
		\end{equation}
		up to a finite constant, which integrates to $ 4/\pi^{1/2} $. Then by the dominated convergence theorem, we have that for any $ \varepsilon >0 $, $ \exists N >0: \, \forall n \geq N $ 
		\begin{equation}
			\int_{\mathbb{R}^3\setminus B_\delta(0)}  \frac{\mathcal{R}^{[n]}(y)}{(2 \pi)^{3/2}} \diff y  < \frac{\varepsilon}{2 \| f \|_{L^p}} .
		\end{equation}
		Together with \eqref{eq:a0}, this concludes the proof of~\eqref{lim:reg_strong_conv_simple}.
	\end{proof}


	We also remark the following basic estimates in $H^\sigma$ spaces.
	
	\begin{lemma}
		Let $s\in(0,\frac32)$. For all $n\geq1$, $R^{[n]}:L^2\to L^2$ identifies to the Fourier multiplier 
		\begin{equation}\label{eq:R(n)L^2}
			R^{[n]} \vcentcolon= \Big(1-\frac{1}{n} \Delta \Big)^{-\frac{s}{2}} = \Big\langle \frac{ - i \nabla}{\sqrt{n}} \Big\rangle^{-s} .
		\end{equation}
		In particular, for all $\sigma\in\mathbb{R}$ and all $n\geq1$, $R^{[n]}$ extends to the a bounded operator in $\mathcal{B}(H^\sigma)$ such that 
		\begin{equation}\label{eq:bound_R(n)_Hsigma}
			\|R^{[n]}\|_{\mathcal{B}(H^\sigma)}=1,
		\end{equation}
		and
		\begin{equation}\label{lim:reg_strong_conv_Hs}
			\forall f \in H^\sigma(\mathbb{R}^3) ,\;\ \lim_{n \rightarrow \infty} \| (R^{[n]} - \mathbb{I}) f \|_{H^\sigma} = 0 .
		\end{equation}
		Moreover the following estimates hold:
		\begin{align}
			&\forall f \in L^2(\mathbb{R}^3), \;\ \| R^{[n]} f \|_{H^{s}} \leq n^{s/2} \| f \|_{L^2} \;\ \label{est:Hs_reg}\\
			&\forall f \in H^{-s}(\mathbb{R}^3), \;\ \| R^{[n]} f \|_{L^2} \leq n^{s/2} \| f \|_{H^{-s}}  .\label{est:Hs_reg2}
		\end{align}
	\end{lemma}
	\begin{proof}
		The action of the operators $R^{[n]}:L^2\to L^2$ is a consequence of the formula 
		\begin{equation*}
			A^{-\frac{s}{2}} = \frac{1}{\Gamma(\frac{s}{2})} \int_{0}^{+\infty}  e^{-\tau A} \tau^{\frac{s}{2}} \frac{\diff \tau}{\tau}
		\end{equation*}
		and the 
		well-known form of the Fourier transform of a Gaussian. Indeed we have, for all $y\in\mathbb{R}^3$,
		\begin{equation*}
			\mathcal{R}^{[n]}(y) = \mathcal{F}^{-1} \Big( \frac{1}{\Gamma(\frac{s}{2})} \int_{0}^{+\infty}  e^{-\frac{\tau}{n} |\xi|^2} e^{-\tau} \tau^{\frac{s}{2}} \frac{\diff \tau}{\tau} \Big) (y) = \frac{n^{3/2}}{2^{3/2} \Gamma(\frac{s}{2})} \int_{0}^{+\infty}  e^{-\tau} \tau^{-\frac{5-s}{2}} e^{-\frac{n}{4 \tau} |y|^2} \diff \tau,
		\end{equation*}
		which yields
		\begin{equation}\label{eq:R(n)L^2a}
			\widehat{(R^{[n]}f)} (\xi) = ( \big(1 + |\cdot|^2/n \big)^{-\frac{s}{2}} \hat{f} ) (\xi), 
		\end{equation}
		for all $\xi\in\mathbb{R}^3$. This proves \eqref{eq:R(n)L^2}.
		
		
		By \eqref{eq:R(n)L^2a}, we see that $R^{[n]}$ acts as the multiplication operator by $(1 + |\cdot|^2/n )^{-\frac{s}{2}}$ in Fourier space. Since $\|(1 + |\cdot|^2/n )^{-\frac{s}{2}}\|_{L^\infty}=1$, this implies \eqref{eq:bound_R(n)_Hsigma}. The strong convergence properties follows from the pointwise convergence
		\begin{equation}
			\forall \xi \in \mathbb{R}^3, \;\ \lim_{n \rightarrow \infty} \langle  | \xi |/\sqrt{n} \rangle^{-1} = 1,
		\end{equation}
		the basic bound
		\begin{equation}
			\Bigg| (1+\frac{1}{n} | \xi |^2 )^{-s} -1 \Bigg| \leq 1 ,
		\end{equation}
		and the dominated convergence theorem.
		
		Finally the direct computation
		\begin{align}
			\| R^{[n]} f \|_{H^{s}}^2 = \int_{\mathbb{R}^3}  \frac{(1 + | \xi |^2)^s}{(1 + \frac{| \xi|^2}{n})^s} | \hat{f} (\xi) |^2 \diff \xi \leq n^s \int_{\mathbb{R}^3}  | \hat{f} (\xi) |^2 \diff \xi,
		\end{align}
		establishes \eqref{est:Hs_reg}. The estimate in the $L^2$ and $H^{-s} $ norms is obtained by applying \eqref{est:Hs_reg} to~$ \langle \nabla \rangle^{-\frac{s}{2}} f $.
	\end{proof}

	\section{Local existence} \label{sec:local-ex}
	
	This section is devoted to the proof of Theorem \ref{thm:local_existence}, using an approximation procedure adapted from the works of \cite{cazenave1988cauchy,strauss1970weak}. We will make use of some results from \cite{cazenave2003semilinear}.  The first step is to consider a sequence of Kohn-Sham equations depending on a parameter $n$ with regularized non-linearities. At fixed $n$, one can prove the existence of a solution, and then, up to a subsequence, show that the solutions converge as $n\to \infty$. Finally one verifies that the limit is a solution to the original Kohn-Sham equations. 
	
	\subsection{The regularized Kohn-Sham equations}
	
	We now introduce our regularized version of the Kohn Sham equations. The original interaction~$\ulg(\ulu)$ is defined for $ \ulu$  in $\ell^2 (L^2\cap L^{\nu}) $ by the sequence in \eqref{eq:def-interaction}. The regularized version~$\ulg^{[n]}(\ulu)=(g^{[n]}_{k}(\ulu))_{k\in\mathbb{N}}$ is defined, for $n\geq1$, by
	\begin{equation}
		\forall k \in \mathbb{N}\,,\quad g^{[n]}_{k}(\ulu) := R^{[n]} \, g_k( R^{[n]} \ulu ) ,
	\end{equation}
	where $R^{[n]}\ulu := (R^{[n]}u_k)_{k\in\mathbb{N}} $. 
	
	The action of the regularization is sufficient to make the regularized non-linearity locally Lipschitz in $\ell^2(L^2) $, by using the Sobolev embeddings \eqref{eq:sobolev_embedding_plus}--\eqref{eq:sobolev_embedding_minus}, and the continuity of $R^{[n]}$ from $H^s$ to $L^2$ or $L^2$ to $H^{-s}$ stated in \eqref{est:Hs_reg}. 
	\begin{lemma} Let $s\in(0,\frac32)$. If $\ulg\in\Nloc$ and $M>0$, for any $ \ulu$ and $\ulv$ in $B_{\ell^2(L^2)}(0,M)$ and any $ n \geq 1$,
		\begin{equation}
			\label{est:loc_lipsch_L2}
			\| \ulg^{[n]}(\ulu) - \ulg^{[n]}(\ulv) \|_{\ell^2(L^2)}
			\leq C_s^2n^s \Clip{n^{s/2}M}   \|\ulu-\ulv\|_{\ell^2(L^2)}.
		\end{equation}
	\end{lemma}
	\begin{proof}
		First, the boundedness of $R^{[n]}$ from $L^2$ to $H^{-s}$ stated in \eqref{est:Hs_reg} yields
		\begin{equation}
			\| \ulg^{[n]}(\ulu) -\ulg^{[n]}(\ulv)\|_{\ell^2(L^2)} \leq n^{\frac{s}{2}} \| \ulg(R^{[n]}\ulu) -\ulg(R^{[n]}\ulv) \|_{\ell^2(H^{-s})}. 
		\end{equation}
		Then, by the Sobolev embedding \eqref{eq:sobolev_embedding_minus}, 
		\begin{equation}
			\| \ulg^{[n]}(\ulu) -\ulg^{[n]}(\ulv)\|_{\ell^2(L^2)} \leq C_{s} n^{\frac{s}{2}} \| \ulg(R^{[n]}\ulu) -\ulg(R^{[n]}\ulv) \|_{\ell^2(L^2+L^{\nu^\prime})}\,,
		\end{equation}
		where $\nu$ is such that $\ulg\in\Nlocnu$. As $\|R^{[n]}\ulu\|_{\ell^2(H^{s})}\leq n^{s/2}M$ and $\|R^{[n]}\ulv\|_{\ell^2(H^{s})}\leq n^{s/2}M$ by \eqref{est:Hs_reg}, the condition \ref{ass2:nonlinearity_Lp_bound} on the non-linearity gives
		\begin{equation*}
			\| \ulg^{[n]}(\ulu) -\ulg^{[n]}(\ulv) \|_{\ell^2(L^2)}
			\leq C_{s} \, n^{\frac{s}{2}} \,
			\Clip{n^{s/2}M} \|R^{[n]}(\ulu-\ulv) \|_{\ell^2(L^2\cap L^{\nu})} \,.
		\end{equation*}
		Using again the boundedness of $R^{[n]}$ from $L^2$ to $H^s$ \eqref{est:Hs_reg}, and the Sobolev embedding \eqref{eq:sobolev_embedding_plus}, we obtain
		\begin{equation}
			\| \ulg^{[n]}(\ulu) -\ulg^{[n]}(\ulv) \|_{\ell^2(L^2)}	\leq C_s^2\, n^s \Clip{n^{s/2}M} \|\ulu-\ulv \|_{\ell^2(L^2)},
		\end{equation}
		which concludes the proof.
	\end{proof}
	
	We consider now an approximated problem given by the Kohn-Sham equations with a regularized non-linearity. The equations for the $n$-dependent solutions $ \ulpsi^{[n]} $ are
	\begin{equation}
		\label{eq:Kohn_Sham_reg}
		\begin{dcases}
			i \partial_t \ulpsi^{[n]}(t) = (1-\Delta)^s \ulpsi^{[n]}(t) +  \ulg^{[n]} (\ulpsi^{[n]}(t)) \, ,\\
			\ulpsi^{[n]}(t=0)= \ulphi  \, .
		\end{dcases}
	\end{equation}
	The definition of the energy has to be adapted for these regularized equations:
	\begin{equation*}
		\mathcal{E}^{[n]}(\ulpsi) = \frac{1}{2} \| \ulpsi \|_{\ell^2(H^{s})}^2 +  \mathscr{G}^{[n]}(\ulpsi) \quad \text{for} \quad \ulpsi \in \ell^2(H^{s})\,,
	\end{equation*}
	where the regularized non-linearity energy is simply obtained by applying the regularizing operator to each element of the sequence and computing the original non-linearity energy
	\begin{equation}\label{def:approx_nonlin_energy}
		\mathscr{G}^{[n]}(\ulpsi) = \mathscr{G} (R^{[n]}\ulpsi)= \mathscr{G} ((R^{[n]}\psi_k)_k) .
	\end{equation}

	We have the following theorem deduced from \cite[Theorem 3.3.1]{cazenave2003semilinear}.
	\begin{theorem}\label{thm:existence_regularized}
		Let $s\in(0,\frac32)$, $n \geq 1$ and $\ulg\in\Nloc$. Consider the equations \eqref{eq:Kohn_Sham_reg} with initial condition $ \ulphi  \in \ell^2(L^2) $. There exists a global strong solution
		\[
		\ulpsi^{[n]} \in \mathscr{C}\big(\mathbb{R};\ell^2(L^2)\big) \cap \mathscr{C}^1 \big(\mathbb{R};\ell^2(H^{-2s})\big)
		\]
		such that:
		\begin{enumerate}[label=\roman*)]
			\item The $ \ell^2(L^2) $ norm is conserved: $ \forall t \in \mathbb{R},\;\ \| \ulpsi^{[n]}(t) \|_{\ell^2(L^2)} = \| \ulphi  \|_{\ell^2(L^2)} $;
			\item \label{it:cons_en} If $ \ulphi  \in \ell^2(H^{s}) $, then 
			\begin{equation*}
				\ulpsi^{[n]} \in \mathscr{C}\big(\mathbb{R};\ell^2(H^{s})\big) \cap \mathscr{C}^1\big(\mathbb{R};\ell^2(H^{-s})\big)
			\end{equation*}
			and the energy is conserved: 
			\[ \forall t \in \mathbb{R}, \;\ \mathcal{E}^{[n]}\big(\ulpsi^{[n]}(t)\big) = \mathcal{E}^{[n]}(\ulphi ) ; \]
			\item If $ \ulphi  \in \ell^2(H^{2s}) $, then $ \ulpsi^{[n]} \in \mathscr{C}\big(\mathbb{R};\ell^2(H^{2s})\big) \cap \mathscr{C}^1\big(\mathbb{R};\ell^2(L^2)\big) $.
		\end{enumerate}
	\end{theorem}
	\begin{proof}[Proof of Theorem \ref{thm:existence_regularized}]
		This is an application of \cite[Theorem 3.3.1]{cazenave2003semilinear} with the operator~$ A = -\bigoplus_{k \geq 0} (1-\Delta)^s $ and the choice of spaces $ X = \ell^2(L^2) $, $ X_A = \ell^2(H^{s}) $ and $ \mathcal{D}(A) = \ell^2(H^{2s}) $ with the relative duals taken with respect to the $L^2$ duality product. 
		The crucial point is that the interaction is now locally Lipschitz in the base space $ \ell^2(L^2) $ by \eqref{est:loc_lipsch_L2}. As such the typical argument by the Banach contraction theorem can be applied directly in $ \ell^2(L^2) $. 
	\end{proof}
	
	\subsection{Proof of Theorem \ref{thm:local_existence}}
	
	We consider, for all $n\geq1$, a global solution $ \ulpsi^{[n]} \in \mathscr{C}(\mathbb{R};\ell^2(H^{s})) $ of the regularized problem~\eqref{eq:Kohn_Sham_reg} provided by Theorem \ref{thm:existence_regularized}. The strategy to prove Theorem~\ref{thm:local_existence} is to show that, up to a subsequence, these solutions converge in the limit $ n \rightarrow \infty $ to a local solution of the Kohn-Sham equations \eqref{eq:Kohn-Sham-vectorised}. 
	We divide the proof of Theorem~\ref{thm:local_existence} in several Lemmata.
	\begin{lemma}[Bounds on an interval independent of \texorpdfstring{$n$}{n}]\label{lem:interval-indep-n}
		Let $s\in(0,\frac32)$, $\ulg\in\Nloc$, $M>0$ and~$ \ulphi  \in \ell^2(H^{s}) $ such that $ \| \ulphi  \|_{\ell^2(H^{s})} \leq M$. There exist $T(M)>0$ and $\tilde{C}(M)>0$ such that, with $I(M)=(-T(M),T(M))$,
		\begin{equation}
			\begin{dcases} \label{est:uniform_est_theorem}
				\sup_{n \geq 1} \| \ulpsi^{[n]} \|_{L^{\infty}(I(M);\ell^2(H^{s}))} \leq 2M \, , \\
				\sup_{n \geq 1} \| \partial_t \ulpsi^{[n]} \|_{L^{\infty}(I(M);\ell^2(H^{-s}))} \leq \tilde{C}(M) ,
			\end{dcases}				
		\end{equation}
		where, for all $n \geq 1$, $\ulpsi^{[n]}$ is the solution to the regularized Kohn-Sham equations \eqref{eq:Kohn_Sham_reg} provided by Theorem \ref{thm:existence_regularized}.
	\end{lemma}
	\begin{proof}[Proof of Lemma~\ref{lem:interval-indep-n}]
		Define, for all $ n \geq 1$, 
		\begin{equation}\label{def:taun}
			\tau_n = \sup \left\{ \tau >0 \,\Big|\, \forall t \in (-\tau, \tau), \; \| \ulpsi^{[n]}(t) \|_{\ell^2(H^{s})} \leq 2M  \right\} .
		\end{equation}
		Taking the $\ell^2(H^{-s})$ norm of both sides of
		\begin{equation*}
			i \partial_t \ulpsi^{[n]}(t) = (1-\Delta)^s \ulpsi^{[n]}(t) +  \ulg^{[n]}(\ulpsi^{[n]}(t)) ,
		\end{equation*}
		yields, for all $t$ in  $(-\tau_n, \tau_n) $,
		\begin{equation*}
			\| \partial_t \ulpsi^{[n]}(t) \|_{\ell^2(H^{-s})} 
			\leq  \| (1-\Delta)^s \ulpsi^{[n]}(t) \|_{\ell^2(H^{-s})}
			+  \|  \ulg^{[n]}(\ulpsi^{[n]}(t)) \|_{\ell^2(H^{-s})} .
		\end{equation*}
		By the basic estimate $ \| (1-\Delta)^s f \|_{H^{-s}} \leq \| f \|_{H^{s}} $ and the Sobolev embedding \eqref{eq:sobolev_embedding_minus}, we obtain
		\begin{align*}
			\| \partial_t \ulpsi^{[n]}(t) \|_{\ell^2(H^{-s})} \leq \|  \ulpsi^{[n]}(t) \|_{\ell^2(H^{s})}
			+  C_{s} \| \ulg^{[n]}(\ulpsi^{[n]}(t)) \|_{\ell^2(L^2+L^{\nu^\prime})} .
		\end{align*}
		Next it follows from the non-linearity estimate \eqref{est:single_nonlinearity_Lp} and the boundedness of the regularization operator \eqref{est:reg_Lp} that, for all $t$ in $(-\tau_n, \tau_n) $,
		\begin{align*}
			\| \partial_t \ulpsi^{[n]}(t) \|_{\ell^2(H^{-s})} 
			& \leq  2M 
			+  C_{s} \Clip{2M}  \| \ulpsi^{[n]}(t) \|_{\ell^2(L^2\cap L^\nu)}  \\
			& \leq  2M + C_{s}^2 \Clip{2M} \, \| \ulpsi^{[n]}(t) \|_{\ell^2(H^s)}  .
		\end{align*}
		Therefore, we get a bound of the time derivative over the intervals $ (-\tau_n,\tau_n)$ uniformly with respect to $ n \geq 1$:
		\begin{equation}\label{eq:b0}
			\sup_{n \geq 1} \| \partial_t \ulpsi^{[n]}(t) \|_{L^{\infty}((-\tau_n,\tau_n);\ell^2(H^{-s}))} \leq 2M(1 +  C_s^2  \Clip{2M}  ) = \vcentcolon \tilde{C}(M) \,.
		\end{equation}
		From \eqref{est:holder_time} in Lemma \ref{lem:estimate_holder_t}, we deduce a bound for all $t$ and $t^\prime$ in $(-\tau_n, \tau_n)$, uniformly with respect to $ n \geq 1$: 
		\begin{equation}\label{est:L2_holder_continuity}
			\| \ulpsi^{[n]}(t)-\ulpsi^{[n]}(t^\prime) \|_{\ell^2(L^2)} \leq 2 \max\{2M, \tilde{C}(M)\}\, | t-t^\prime |^{\frac12} .
		\end{equation}
		From the conservation of energy of the regularized equations (see \ref{it:cons_en} in Theorem \ref{thm:existence_regularized}) we obtain, for all $t$ in $(-\tau_n, \tau_n) $,
		\begin{equation*}
			\mathcal{E}^{[n]}(\ulpsi^{[n]}(t)) = \frac{1}{2} \| \ulpsi^{[n]}(t) \|_{\ell^2(H^{s})}^2 +  \mathscr{G}^{[n]}(\ulpsi^{[n]}(t)) = \frac{1}{2} \| \ulphi  \|_{\ell^2(H^{s})}^2 + \mathscr{G}^{[n]}(\ulphi ),
		\end{equation*}
		which in turn allows us to estimate, for all $t$ in $(-\tau_n, \tau_n) $,
		\begin{equation*}
			\| \ulpsi^{[n]}(t) \|_{\ell^2(H^{s})}^2 \leq \| \ulphi  \|_{\ell^2(H^{s})}^2 +2  \big| \mathscr{G}^{[n]}(\ulpsi^{[n]}(t)) - \mathscr{G}^{[n]}(\ulphi ) \big| .
		\end{equation*}
		Using \eqref{est:nonlinear_energy_L2} in Lemma \ref{lem:hoelder-regularity-non-linearity}, we deduce that, for all $t$ in $(-\tau_n, \tau_n) $,
		\begin{equation*}
			\| \ulpsi^{[n]}(t) \|_{\ell^2(H^{s})}^2 \leq \| \ulphi \|_{\ell^2(H^{s})}^2 +2  C(M) \, \thetasnu (\| R^{[n]}\ulpsi^{[n]}(t)-R^{[n]}\ulphi \|_{\ell^2(L^2)})  ,
		\end{equation*}
		and by the boundedness of the regularizing operator \eqref{est:reg_Lp} and the H{\"o}lder in time estimate \eqref{est:holder_time} we obtain, for all $t$ in $(-\tau_n, \tau_n) $,
		\begin{align*}
			\| \ulpsi^{[n]}(t) \|_{\ell^2(H^{s})}^2 
			&\leq M^2 + 2  C(M) \, \thetasnu \big( 2 \max\{2M,\tilde{C}(M)\}  |t|^{1/2} \big) ,
		\end{align*}
		as we have the assumption $ \| \ulphi \|_{\ell^2(H^{s})} \leq M $.  Note that the right-hand side does not depend on $n$. As such, we have obtained an estimate,  uniform with respect to $ n \geq 1$ over all the intervals~$ (-\tau_n, \tau_n) $, of the energy norm for the solution $\ulpsi^{[n]}(t) \in \ell^2(H^{s}) $ to the approximated problem.
		
		We now construct a time interval independent of $n \geq 1$ such that a uniform bound holds for $\ulpsi^{[n]}(t)$. We define $ T(M) > 0 $ (independent of $n$) such that
		\begin{equation}
			C(M) \,\thetasnu \big( 2 \max\{2M,\tilde{C}(M)\}  |T(M)|^{1/2}\big) = M^2 \,,
		\end{equation}
		which is possible as $\thetasnu (x)\xrightarrow[x\to0]{}0$ and $\thetasnu (x)\xrightarrow[x\to\infty]{}\infty$. 
		Then, we obtain by the previous estimate that
		\begin{equation}
			\forall n \geq 1, \, \forall t \in \left( -T(M), T(M) \right) \cap (-\tau_n, \tau_n), \;\ \| \ulpsi^{[n]}(t) \|_{\ell^2(H^{s})} \leq \sqrt{3}M .
		\end{equation} 
		Note that we must have $ \forall n \geq 1, \, T(M) \leq \tau_n $ as otherwise the bound by $\sqrt{3}M$ would hold on the interval $ (-\tau_n,\tau_n) $, contradicting the definition of $\tau_n$ in \eqref{def:taun}, by continuity of $t\mapsto \ulpsi^{[n]}(t)$. Therefore there exists an interval $ I(M) = \left( -T(M),T(M) \right) $, independent of $n$, that is contained in all the intervals $(-\tau_n,\tau_n)$. On this interval the following uniform bound holds:
		\begin{equation}
			\sup_{n \geq 1} \| \ulpsi^{[n]}(t) \|_{L^{\infty}(I(M);\ell^2(H^{s}))} \leq 2M .
		\end{equation}
		Together with \eqref{eq:b0}, this concludes the proof of the two  estimates in \eqref{est:uniform_est_theorem}.
	\end{proof}

	The uniform bounds \eqref{est:uniform_est_theorem} are the starting points to apply a compactness argument, from which we will deduce the existence of a solution to the Kohn-Sham equations \eqref{eq:Kohn-Sham-vectorised}. 
	
	\begin{lemma}[Convergence of the solutions and  of the non-linearities by compactness]\label{lem:convergence-solutions}Let~$s$ in~$(0,\frac32)$, $\nu\in[2,\twoasts)$ and $\ulg\in\Nlocnu$. Let $ \ulphi  \in \ell^2(H^{s}) $ and $ M \geq \| \ulphi  \|_{\ell^2(H^{s})} $. Let $I(M)$ be the time-interval constructed in Lemma~\ref{lem:interval-indep-n}, such that the uniform bounds \eqref{est:uniform_est_theorem} hold for the sequence of solutions $(\ulpsi^{[n]})_{n \geq 1}$ to the regularized Kohn-Sham equations \eqref{eq:Kohn_Sham_reg} provided by Theorem \ref{thm:existence_regularized}.
		
		Then there exist
		\begin{enumerate}[label=\roman*)]
			\item \begin{equation}
				\ulpsi \in \mathscr{C}({I(M)}; \ell^2(H^{s})_{\mathrm{weak}}) \cap W^{1,\infty}(I(M);\ell^2(H^{-s}))
			\end{equation}
			such that
			\begin{equation}\label{est:energy_bound_solution}
				\| \ulpsi \|_{L^{\infty}(I(M); \ell^2(H^{s}))} \leq 2M , \quad \| \ulpsi \|_{W^{1,\infty}(I(M);\ell^2(H^{-s}))} \leq \tilde{C}(M),
			\end{equation}
			\item
			\begin{equation}
				\ulf \in \mathscr{C}(I(M); \ell^2(H^{-s})),
			\end{equation}
			such that $\| \ulf \|_{L^{\infty}(I(M);\ell^2(L^2+L^{\nu^\prime}))} < \infty$, 
			\item  a subsequence (still denoted by $\ulpsi^{[n]}$) such that, for all $t$ in $I(M)$,
			\begin{equation}\label{lim:weak_convergence_sequence}
				\ulpsi^{[n]}(t) \xrightarrow[n\to\infty]{\mathrm{w}-\ell^2(H^{s})} \ulpsi(t),
			\end{equation}
			the sequence $ \ulg^{[n]}(\ulpsi^{[n]}(t)) $ is bounded in $\ell^2(L^2+L^{\nu^\prime})$, and
			\begin{equation}\label{lim:weak_convergence_nonlince}
				\ulg^{[n]}(\ulpsi^{[n]}(t)) \xrightarrow[n\to\infty]{\mathrm{w}-\ell^2(L^2+L^{\nu^\prime})} \ulf(t) \,.
			\end{equation}
		\end{enumerate}
	\end{lemma}
	\begin{proof}[Proof of Lemma~\ref{lem:convergence-solutions}]
		Our compactness argument follows by an application of the Banach--Alaoglu theorem with a suitable density and diagonal argument for the time variable for which we refer to \cite[Proposition 1.1.2]{cazenave2003semilinear} applied with the spaces
		\begin{equation}\label{def:space_choices_XBY}
			X = \ell^2(H^{s}), \;\ B=\ell^2(L^2) , \;\ Y= \ell^2(H^{-s}).
		\end{equation}
		Note that the following embeddings hold:
		\begin{equation}
			\ell^2(H^{s}) \xhookrightarrow{} \ell^2(L^2) \xhookrightarrow{} \ell^2(H^{-s}) .
		\end{equation}
		The family $( \ulpsi^{[n]}(t) )_n $ is equicontinuous in the $ \ell^2(H^{-s}) $ norm by the easy previous embedding~$ \| \ulpsi^{[n]} \|_{\ell^2(H^{-s})} \leq \| \ulpsi^{[n]} \|_{\ell^2(L^2)} $ together with \eqref{est:holder_time} and \eqref{est:uniform_est_theorem}. Indeed, for all $t$ and $t^\prime$ in $I(M)$, we have
		\begin{equation}
			\| \ulpsi^{[n]}(t) - \ulpsi^{[n]}(t^\prime) \|_{\ell^2(H^{-s})} \leq 2 \max\{2M, \tilde{C}(M)\} |t-t^\prime|^{\frac{1}{2}} ,
		\end{equation}

		and hence the sequence $ (\ulpsi^{[n]}(t))_n$ belongs to $\mathscr{C}^{0,\frac{1}{2}}(\overline{I(M)};\ell^2(H^{-s}))$, uniformly with respect to~$ n \geq 1$. By~\eqref{est:uniform_est_theorem}, $ (\ulpsi^{[n]}(t))_n$ is also uniformly bounded in the $ \ell^2(H^{s}) $ norm for all $ t$ in~$\overline{I(M)}$ and~$ n \geq 1$. Moreover, as $ \ell^2(H^{s}) $ is reflexive, we can apply \cite[Proposition 1.1.2]{cazenave2003semilinear} to obtain that there exist
		\begin{equation}
			\ulpsi \in \mathscr{C}\big(\overline{I(M)}; \ell^2(H^{-s})\big)\cap \mathscr{C}\big(\overline{I(M)}; \ell^2(H^{s})_{\mathrm{weak}}\big),
		\end{equation}
		and a subsequence (which is still denoted by $\ulpsi^{[n]}(t)$) 
		which satisfies \eqref{lim:weak_convergence_sequence}. 
		Notice that, by lower semicontinuity of the norm in the weak topology, the first bound in \eqref{est:energy_bound_solution} holds.
		
		We can strengthen the convergence of $\ulpsi^{[n]}$ to $\ulpsi$ by using \cite[Remark 1.3.13 ii)]{cazenave2003semilinear}: As~$ \ell^2(H^{-s}) $ is reflexive,
		\begin{equation}
			\ulpsi^{[n]} \in W^{1,\infty}\big(I(M); \ell^2(H^{-s})\big), 
		\end{equation}
		by Theorem \ref{thm:existence_regularized}  (since for each $ n \geq 1$, $\ulpsi^{[n]}$ belongs to $\mathscr{C}^1(\mathbb{R};\ell^2(H^{-s})) $) and as the weak convergence in~$ \ell^2(H^{s}) $ implies the weak convergence in~$\ell^2(H^{-s})$, 
		we conclude from \cite[Remark 1.3.13 ii)]{cazenave2003semilinear} that  
		\begin{equation}
			\ulpsi \in W^{1,\infty}\big(I(M);\ell^2(H^{-s})\big) , 
		\end{equation}
		and that
		\begin{equation}
			\| \ulpsi \|_{W^{1,\infty}(I(M);\ell^2(H^{-s}))} \leq \tilde{C}(M) ,
		\end{equation}
		where we used in addition the uniformity with respect to $ n \geq 1$ of estimate \eqref{est:uniform_est_theorem}.
		
		From the uniform estimates \eqref{est:uniform_est_theorem} we derive a uniform estimate for the non-linearity in the interval $ I(M) $. From the non-linearity estimate \eqref{est:single_nonlinearity_Lp} and the bound of the regularization operator \eqref{est:reg_Lp}, we obtain that, for all $t \in I(M) $,
		\begin{equation}
			\| \ulg^{[n]}(\ulpsi^{[n]}(t)) \|_{\ell^2(L^2+L^{\nu^\prime})} \leq \Clip{2M} \| \ulpsi^{[n]}(t) \|_{\ell^2(L^2\cap L^{\nu})}
			\leq C_s \Clip{2M} \| \ulpsi^{[n]}(t) \|_{\ell^2(H^s)}
			\leq 2 C_s M \Clip{2M} ,
		\end{equation}
		as $\| \ulpsi^{[n]}(t) \|_{\ell^2(H^s)}\leq 2M$. This yields the following uniform estimate with respect to $n \geq 1 $:
		\begin{equation}\label{est:uniform_theorem_nonlin}
			\| \ulg^{[n]}(\ulpsi^{[n]}(t)) \|_{L^{\infty}(I(M);\ell^2(L^2+L^{\nu^\prime}))} \leq 2 C_s M \Clip{2M} .
		\end{equation}
		The non-linearity is also uniformly Hölder continuous in time. Indeed, by \eqref{est:Lnu_L2_norms} and the boundedness of the regularizing operator \eqref{est:reg_Lp} we get, for all $ t, \, t^\prime \in I(M) $,
		\begin{equation}
			\| \ulg^{[n]}(\ulpsi^{[n]}(t)) - \ulg^{[n]}(\ulpsi^{[n]}(t^\prime)) \|_{\ell^2(L^2+L^{\nu^\prime})} \leq C(M)\, \thetasnu  \big( \| \ulpsi^{[n]}(t)-\ulpsi^{[n]}(t^\prime) \|_{\ell^2(L^2)}\big),
		\end{equation}
		where $\thetasnu$ is defined in \eqref{eq:eta-thetasnu}.  Now, by \eqref{est:holder_time} we obtain
		\begin{equation}\label{eq:equicontinuity-approximated-non-linearity}
			\| \ulg^{[n]}(\ulpsi^{[n]}(t)) - \ulg^{[n]}(\ulpsi^{[n]}(t^\prime)) \|_{\ell^2(L^2+L^{\nu^\prime})}  
			\leq C(M)  \, \thetasnu \big(2 \max\{2M , \tilde{C}(M)\}  |t-t^\prime|^{1/2}\big) ,
		\end{equation}
		where we have used \eqref{est:uniform_est_theorem} and the fact that $\thetasnu $ is non-decreasing. 
		In particular, $t \mapsto \ulg^{[n]}(\ulpsi^{[n]}(t))$ is uniformly equicontinuous with respect to $n$ from the interval $ I(M) $ to $\ell^2(L^2+L^{\nu^\prime})$. 
		We can now proceed again by a compactness argument to obtain a weak limit of the non-linearity. We apply again \cite[Proposition 1.1.2]{cazenave2003semilinear}, now with the spaces 
		\begin{equation}
			X= \ell^2(L^2+ L^{\nu'}), \quad  Y = \ell^2(H^{-s}) .
		\end{equation}
		Note that $\ell^2(L^2+ L^{\nu'}) \xhookrightarrow{} \ell^2(H^{-s})$ by the Sobolev embedding \eqref{eq:sobolev_embedding_minus} and that $  \ell^2(L^2+ L^{\nu'})  $ is reflexive. Using the bound \eqref{eq:equicontinuity-approximated-non-linearity} and the uniform equicontinuity of $t \mapsto \ulg^{[n]}(\ulpsi^{[n]}(t))$ from $I(M)$ to $ \ell^2(H^{-s})$ (by the Sobolev embedding \eqref{eq:sobolev_embedding_minus}), we deduce from \cite[Proposition 1.1.2]{cazenave2003semilinear} that there exist a limit element 
		\begin{equation}
			\ulf \in \mathscr{C}(\overline{I(M)}; \ell^2(H^{-s})) ,
		\end{equation}
		and a subsequence, denoted again by $ \ulg^{[n]}(\ulpsi^{[n]}(t)) $, such that
		\begin{equation}
			\forall t \in \overline{I(M)}, \;\  \ulg^{[n]}(\ulpsi^{[n]}(t)) \xrightarrow{\mathrm{w}-\ell^2(L^2+ L^{\nu'})} \ulf(t) ,
		\end{equation}
		and $ \ulf \in L^\infty( I(M) ; \ell^2(L^2+ L^{\nu'}) )$. This concludes the proof of the lemma. 
	\end{proof}

	\begin{remark}\label{rmk:weak_convergence_ellp}
		Observe that any weak limit $\ulpsi^{[n]} \to \ulpsi$ in $ \ell^2(H^{-s}) $ as $n\to\infty$ implies a weak limit for any fixed $k$ in $\mathbb{N} $ by considering a sequence $ \ulw = (w_{k}  \delta_{k,\tilde{k}})_{\tilde{k}} $, $ w_k \in H^s $. More precisely, if $\ulpsi^{[n]} \to \ulpsi$ in $ \ell^2(H^{-s}) $ as $n\to\infty$, then we have that 
		\begin{equation}
			\forall k \in \mathbb{N} , \, \forall w_k \in H^{s}, \;\ \langle \psi^{[n]}_{k} , w_k \rangle_{H^{-s},H^s} \xrightarrow[n\to\infty]{} \langle \psi_k , w_k \rangle_{H^{-s},H^s} 
		\end{equation} 
		and, similarly, if $\ulf^{[n]}$ converges weakly to $\ulf$ in $ \ell^2(L^2+L^{\nu'}) $ with $\nu$ in $[2,2^*_s)$
		then
		\begin{equation}
			\forall k \in \mathbb{N}, \, \forall w_k \in H^{s}, \;\ \langle f^{[n]}_{k} , w_k \rangle_{L^2+L^{\nu'},L^2\cap L^\nu} \xrightarrow[n\to\infty]{} \langle f_k , w_k \rangle_{L^2+L^{\nu'},L^2\cap L^\nu} 
		\end{equation} 
		using the appropriate Sobolev embeddings \eqref{eq:sobolev_embedding_plus}--\eqref{eq:sobolev_embedding_minus} to obtain the weak convergence testing against functions in $ H^{s} $.
	\end{remark}
	
	To conclude the proof of Theorem~\ref{thm:local_existence}, it remains to show that the limit $\ulpsi$ constructed in Lemma~\ref{lem:convergence-solutions} satisfies the Kohn-Sham equations \eqref{eq:Kohn-Sham-vectorised}. Before delving into that, we recall a technical result involving the convergence of restrictions to regular bounded sets of weakly convergent sequences in Sobolev spaces.
	
	\begin{lemma}
		\label{lem:restriction}
		Let $R>0$ and $ y \in \mathbb{R}^3$. Denote by $ B_R(y) \subset \mathbb{R}^3 $ the open ball of radius $R$ centered at $y$. Then the restriction operator 
		\begin{equation}
			\mathrm{r}_{B_R(y)}: H^s(\mathbb{R}^3) \rightarrow H^s(B_R(y))
		\end{equation}
		is a continuous linear map, where the Sobolev space on the ball is defined as the interpolation space of $L^2(B_R(y))$ and $ H^1(B_R(y))$.
		Let $(u_n)_n \subset H^s(\mathbb{R}^3) $ be a weakly convergent subsequence to~$u \in H^s(\mathbb{R}^3)$, then 
		\begin{equation}
			\mathrm{r}_{B_R(y)} (u_n) \xrightarrow[n\to\infty] {\mathrm{w}-H^s(B_R(y))} \mathrm{r}_{B_R(y)} (u) ,
		\end{equation}
		and for any $\nu \in [2, \twoasts  ) $, there exists a subsequence $(u_{n_m})_m$ such that
		\begin{equation}
			\mathrm{r}_{B_R(y)} (u_{n_m}) \xrightarrow[m\to\infty]{\mathrm{s}-(L^2+L^\nu)(B_R(y))} \mathrm{r}_{B_R(y)}(u).
		\end{equation}
	\end{lemma}
	\begin{proof}
		The continuity of the restriction operator is proven in \cite[Theorem 9.1]{lions2012non} while a proof that the definition by interpolation of the fractional Sobolev spaces coincides with the definition using the Sobolev-Slobodeckij norms can be found in \cite[Remark 2, Section 4.4.1]{triebel1978interpolation}. Recall that the Sobolev embeddings in these spaces are compact, as shown in \cite[Theorem 7.1, Corollary 7.2]{dinezza2012521}.
		Let~$(u_n)_n \subset H^s(\mathbb{R}^3) $ be such that $ u_n \rightharpoonup u $ in $ H^{s}(\mathbb{R}^3) $. Then for any $ \tilde{v} \in H^{s}(B_R(y))^\prime $, 
		\begin{equation}
			\langle \tilde{v}, \mathrm{r}_{B_R(y)} u_n \rangle_{H^{s}(B_R)^\prime,H^{s}(B_R)} = \langle \mathrm{r}_{B_R(y)}^\prime \tilde{v}, u_n \rangle_{H^{s}(\mathbb{R}^3)^\prime,H^{s}(\mathbb{R}^3)} ,
		\end{equation}
		by definition of the adjoint operator of $ \mathrm{r}_{B_R(y)} $ in the continuous dual space and therefore
		\begin{equation}
			\langle \mathrm{r}_{B_R(y)}^\prime \tilde{v}, u_n \rangle_{H^{s}(\mathbb{R}^3)^\prime,H^{s}(\mathbb{R}^3)} 
			\xrightarrow[n\to\infty]{}  \langle \mathrm{r}_{B_R(y)}^\prime \tilde{v}, u \rangle_{H^{s}(\mathbb{R}^3)^\prime,H^{s}(\mathbb{R}^3)} 
			=  \langle \tilde{v}, \mathrm{r}_{B_R(y)} u \rangle_{H^{s}(B_R)^\prime,H^{s}(B_R)} .
		\end{equation}
		Hence  $ \mathrm{r}_{B_R(y)} u_n \rightharpoonup \mathrm{r}_{B_R(y)} u $ in $ H^{s}(B_R(y)) $. The compactness property of the Sobolev embedding then implies that there exists a subsequence $(u_{n_m})_m$ of $u_n$ such that $ \mathrm{r}_{B_R(y)} u_{n_m} \to \mathrm{r}_{B_R(y)} u  $ in~$L^{\tilde{\nu}}(B_R(y))$ for $\tilde{\nu}\in[2,\twoasts)$. Thus, in particular, taking both $\tilde{\nu}=2$ and $\tilde{\nu}=\nu$ and extracting subsequences successively, one gets the existence of a subsequence which converges in the space~$(L^2+L^{{\nu}})(B_R(y))$.
	\end{proof}
	
	We have now all the tools to prove the local existence theorem. 
	
	\begin{proof}[Proof of Theorem~\ref{thm:local_existence}]
		Let $(\ulpsi^{[n]})_{n \geq 1}$ be the sequence of solutions to the regularized Kohn-Sham equations \eqref{eq:Kohn_Sham_reg} provided by Theorem \ref{thm:existence_regularized} and let $\ulpsi$ be the limit of $\ulpsi^{[n]}$ constructed in Lemma~\ref{lem:convergence-solutions}. The last crucial steps in the proof of Theorem~\ref{thm:local_existence} involve showing that the limit~$ \ulpsi $ satisfy in a weak sense  \eqref{eq:Kohn-Sham-vectorised}.
		
		Rewriting the regularized Kohn-Sham equations \eqref{eq:Kohn_Sham_reg} in a weak sense gives that, for all~$k$ in~$\mathbb{N}$, $\chi \in \mathscr{C}_{c}^\infty (I(M);\mathbb{R})$ and $\ulw \in \ell^2(H^{s})$, 
		\begin{align}
			\int_{I(M)}  \Big(i\langle  \ulpsi^{[n]}(t) , \ulw \rangle \partial_{t} \chi(t) + \big\langle (1-\Delta)^s \ulpsi^{[n]}(t) +  \ulg^{[n]} (\ulpsi^{[n]}(t)) , \ulw \big\rangle \chi(t)  \Big)\diff  t = 0,
		\end{align} 
		where the bracket $ \langle \cdot,\cdot \rangle $ should be understood as the duality product between~$ \ell^2(H^{-s})$ and~$\ell^2(H^{s})$. 
		Remark~\ref{rmk:weak_convergence_ellp} applied to the sequences $(\ulpsi^{[n]}(t))_n$ and $(\ulg^{[n]}(\ulpsi^{[n]}(t)))_n$ allows us to consider each equation individually at fixed $k$. From the uniformity of the bounds \eqref{est:uniform_est_theorem} and \eqref{est:uniform_theorem_nonlin} and the boundedness of the interval, we can use Lebesgue's dominated convergence theorem to conclude that for all~$k$ in $\mathbb{N}$, $\chi$ in $\mathscr{C}_{c}^\infty (I(M);\mathbb{R})$ and $w_k$ in $H^{s}$, 
		\begin{align}
			\int_{I(M)} \Big( i\langle  \psi_{k}(t) , w_k \rangle \partial_{t} \chi(t) + \langle (1-\Delta)^s \psi_{k}(t) + f_{k}(t) , w_k \rangle  \chi(t)\Big)  \diff  t= 0 .
		\end{align} 
		This implies that the $ \psi_k$, which all belong to $H^s$, are weak solutions to the equations
		\begin{equation}
			\label{eq:Kohn_Sham_with_f}
			\begin{dcases}
				i \partial_t \psi_{k}(t,x) = (1-\Delta)^s \psi_{k}(t,x) + f_{k}(t,x) \text{ in } H^{-s} \, , \;\ \forall t \in I(M), \, \forall k \in \mathbb{N}, \\
				\psi_k(0,x)= \varphi_k (x) \, , \, \forall k \in \mathbb{N}.
			\end{dcases}
		\end{equation}
		Indeed, by integration by parts in time, as for any $k \geq 0$, $ \psi_k \in W^{1,\infty} (I(M);H^{-s})$, we obtain
		\begin{equation*}
			\int_{I(M)} \langle - i \partial_{t} \psi_{k}(t) +  (1-\Delta)^s \psi_{k}(t) + f_{k}(t) , w_k \rangle  \chi(t)  \diff  t= 0 .
		\end{equation*}
		Therefore, for all $k \geq 0$, for all $w_k \in H^s $, and for almost all $t$ in $I(M)$,
		\begin{equation*}
			\langle - i \partial_{t} \psi_{k}(t) +  (1-\Delta)^s \psi_{k}(t) + f_{k}(t) , w_k \rangle = 0.
		\end{equation*}
		Hence we conclude that for almost all $t$ in $I(M)$ and for all $k \geq 0$, 
		\[i \partial_{t} \psi_{k}(t) = (1-\Delta)^s \psi_{k}(t) + f_{k}(t)\]
		in $H^{-s}$. 
		
		The next step consists in proving that we can replace $ f_{k} (t) $ by $  g_{k}(\ulpsi(t)) $. 
		First we claim that
		\begin{equation}\label{eq:imaginary_part_vanishes}
			\forall t \in I(M), \, \forall k \in \mathbb{N}, \;\ \mathrm{Im} \big( \overline{\psi_k(t,x)} f_{k}(t,x) \big) = 0  \;\ \text{a.e. in } \mathbb{R}^3.  
		\end{equation}
		Indeed, consider any ball $ B_r(y)$, $ r >0$ and $ y \in \mathbb{R}^3 $. To prove that the imaginary part vanishes almost everywhere, we prove that the following integral vanishes for any choice of the ball:
		\begin{equation}
			\mathrm{Re} \langle i \psi_k(t) , f_{k}(t) \rangle_{(L^2\cap L^{\nu})(B_r(y)),(L^2+L^{\nu'})(B_r(y))} = 0 ,
		\end{equation}
		where the application of the restriction operator $r_{B_R(y)}$ from Lemma \ref{lem:restriction} is implicit and we omit it.
		We rewrite the term (where we omit to specify the duality product) in the following way:
		\begin{subequations}
			\begin{align}
				\mathrm{Re} \langle i \psi_k(t) , f_{k}(t) \rangle &=  \mathrm{Re} \langle i \psi_k(t) , f_{k}(t) -R^{[n]} g_{k}(R^{[n]}\ulpsi^{[n]}(t)) \rangle \, \label{subeq:vanishing-realpart-scalar-product1}\\ 
				&\quad + \mathrm{Re} \langle i \psi_k(t) , R^{[n]} g_{k}(R^{[n]}\ulpsi^{[n]}(t)) - g_{k}(R^{[n]}\ulpsi^{[n]}(t)) \rangle \,  \label{subeq:vanishing-realpart-scalar-product2}\\
				&\quad + \mathrm{Re} \langle i \psi_k(t) - i \psi^{[n]}_{k}(t) ,  g_{k}(R^{[n]}\ulpsi^{[n]}(t)) \rangle \, \label{subeq:vanishing-realpart-scalar-product3}\\
				&\quad + \mathrm{Re} \langle i \psi^{[n]}_{k}(t) -i R^{[n]} \psi^{[n]}_{k}(t) , g_{k}(R^{[n]}\ulpsi^{[n]}(t)) \rangle \, \label{subeq:vanishing-realpart-scalar-product4}\\
				&\quad + \mathrm{Re} \langle i R^{[n]} \psi^{[n]}_{k}(t) , g_{k}(R^{[n]}\ulpsi^{[n]}(t)) \, \rangle . \label{subeq:vanishing-realpart-scalar-product5}
			\end{align}
		\end{subequations}
		We now show that each of these terms converges to zero as $ n \rightarrow  \infty $ or is identically equal to zero. 
		\begin{itemize}
			\item For \eqref{subeq:vanishing-realpart-scalar-product1}: By the convergence \eqref{lim:weak_convergence_nonlince} and Lemma~\ref{lem:restriction} clearly the non-linearity converges weakly, also when restricted to the ball $B_r(y) $, and $ \psi_k(t) \in (L^2\cap L^{\nu})(B_r(y)) $. Therefore
			\begin{equation}
				\langle i \psi_k(t) , f_{k}(t) -R^{[n]} g_{k}(R^{[n]}\ulpsi^{[n]}(t)) \rangle_{(L^2\cap L^{\nu})(B_r(y)), (L^2+L^{\nu'})(B_r(y))} 
			\end{equation}
			vanishes as $ n \rightarrow \infty $.
			\item For \eqref{subeq:vanishing-realpart-scalar-product2}: Here we use the boundedness of $  (g_{k}(R^{[n]}\ulpsi^{[n]}(t))_n$ in $L^2+L^{\nu'} $ stated in Lemma~\ref{lem:convergence-solutions} (as follows from \eqref{lim:weak_convergence_nonlince}), together with the strong convergence of $ ((R^{[n]} - \mathbb{I}) \psi_k(t))_n $ given by \eqref{lim:reg_strong_conv_cap}. Thus,
			\begin{multline}
				\langle i \psi_k(t) , R^{[n]} g_{k}(R^{[n]}\ulpsi^{[n]}(t)) - g_{k}(R^{[n]}\ulpsi^{[n]}(t))  \rangle_{(L^2\cap L^{\nu})(B_r(y)), (L^2+L^{\nu'})(B_r(y))}   \\
				= \langle i (R^{[n]}- \mathbb{I})\psi_k(t) ,  g_{k}(R^{[n]}\ulpsi^{[n]}(t))  \rangle_{(L^2\cap L^{\nu})(B_r(y)), (L^2+L^{\nu'})(B_r(y))}
			\end{multline}
			vanishes as $ n \rightarrow \infty $.
			\item For \eqref{subeq:vanishing-realpart-scalar-product3}: By \eqref{lim:weak_convergence_sequence}, we have weak convergence in $\ell^2(H^{s})$ on $\mathbb{R}^3$ of $ \ulpsi^{[n]}(t) $ to $ \ulpsi(t)$ and, by Lemma~\ref{lem:restriction}, we have strong convergence in $ (L^2\cap L^{\nu})(B_r(y)) $. From the uniform bound on~$ g_{k}(R^{[n]}\ulpsi^{[n]}(t)) $ implied by \eqref{lim:weak_convergence_nonlince} in Lemma~\ref{lem:convergence-solutions} on the whole interval, we then deduce that
			\begin{equation}
				\langle i \psi_k(t) - i \psi^{[n]}_{k}(t) ,  g_{k}(R^{[n]}\ulpsi^{[n]}(t)) \rangle_{(L^2\cap L^{\nu})(B_r(y)), (L^2+L^{\nu'})(B_r(y))}
			\end{equation}
			vanishes as $ n \rightarrow \infty $.
			
			\item For \eqref{subeq:vanishing-realpart-scalar-product4}: The strong convergence in $ H^{-s} $, provided by \eqref{lim:reg_strong_conv_Hs}, of $ (R^{[n]}-\mathbb{I}) h$, for any $h$ in $H^{-s} $, implies that $ (R^{[n]}-\mathbb{I}) \psi^{[n]}_{k}(t) $ converges weakly in $H^{s} $. Indeed, we have
			\begin{equation}
				\langle (R^{[n]}-\mathbb{I}) \psi^{[n]}_{k}(t) , h \rangle_{H^{s} , H^{-s}} = \langle  \psi^{[n]}_{k}(t) , (R^{[n]}-\mathbb{I}) h \rangle_{H^{s} , H^{-s}},
			\end{equation} 
			and the right-hand side converges to zero by the fact that $ (\psi^{[n]}_{k}(t))_n $ is weakly convergent in $\ell^2(H^s)$ by Lemma \ref{lem:convergence-solutions}, and thus uniformly bounded, together with the aforementioned strong convergence property. Restricting to the ball and using Lemma~\ref{lem:restriction}, we can lift the weak convergence to the strong convergence of $ (R^{[n]}-\mathbb{I}) \psi^{[n]}_{k}(t) $ in $ (L^2\cap L^{\nu})(B_r(y)) $ as $n\to \infty$. Since in addition the sequence $ (g_{k} (R^{[n]} \ulpsi^{[n]}))_n $ is uniformly bounded by Lemma~\ref{lem:convergence-solutions}, we obtain that
			\begin{equation}
				\langle i \psi^{[n]}_{k}(t) -i R^{[n]} \psi^{[n]}_{k}(t) , g_{k}(R^{[n]}\ulpsi^{[n]}(t)) \rangle_{(L^2\cap L^{\nu})(B_r(y)),(L^2+L^{\nu'})(B_r(y))}
			\end{equation}
			vanishes as $ n \rightarrow \infty $.
			\item For \eqref{subeq:vanishing-realpart-scalar-product5}: This last term vanishes identically as it is purely imaginary by \ref{ass3:nonlinear_energy_derivative} applied with $\ulu = R^{[n]}\ulpsi^{[n]}$ and $\ulv = (\boldsymbol{1}_{B_r(y)}R^{[n]}\psi^{[n]}_{k'}\delta_{k,k'})_{k'\in\mathbb{N}}$, since
			\begin{equation*}
				\mathrm{Re}\langle i R^{[n]} \psi^{[n]}_k(t), g_{k}(R^{[n]}\ulpsi^{[n]}(t)) \rangle_{(L^2\cap L^{\nu})(B_r(y)) , (L^2+L^{\nu'})(B_r(y))} 
				= i \frac{\diff}{\diff \tau}\mathscr{G}(\ulu+\tau\ulv)\Big|_{\tau=0} ,
			\end{equation*}
			and $\mathscr{G}$ is a real valued function.
		\end{itemize}
		
		Hence \eqref{eq:imaginary_part_vanishes} is proven.
		
		Now, for each $ k \in \mathbb{N} $, we consider the real part of the duality product between $ H^{s}$ and $H^{-s} $ of the limit equations \eqref{eq:Kohn_Sham_with_f} with $ i \psi_k(t) $. We find, for all $t \in I(M)$ and $k \in \mathbb{N},  $ 
		\begin{multline}
			\mathrm{Re} \langle i \psi_k(t) , i \partial_t \psi_k(t) \rangle_{H^{s},H^{-s}} 
			=  \mathrm{Re} \langle i \psi_k(t) , (1-\Delta)^s \psi_k(t) \rangle_{H^{s},H^{-s}} + \mathrm{Re} \langle i \psi_k(t) , f_{k}(t) \rangle_{H^{s},H^{-s}} = 0\,,
		\end{multline}
		as the first term in the right-hand side of the first equality vanishes by self-adjointness of~$ (1-\Delta)^s $, and the second term by \eqref{eq:imaginary_part_vanishes}, since $ \mathrm{Re} \langle i \psi_k(t) , f_k(t) \rangle = \mathrm{Im} \langle \psi_k(t) , f_k(t) \rangle = 0 $. The previous equality holds for $t$ in the whole interval $I(M)$ and therefore implies that 
		\begin{equation}
			\forall t \in I(M), \,\ k \in \mathbb{N}, \;\ \frac{\diff }{\diff t} \Big( \frac{1}{2} \| \psi_k(t) \|_{L^2}^2 \Big) = 0 .
		\end{equation}
		Hence we have the conservation of the $L^2$-norm of $\psi_k(t)$ for any $k$ and $t\in I(M)$: 
		\begin{equation}
			\forall t \in I(M), \,\ \forall k \in \mathbb{N}, \;\  \| \psi_k(t) \|_{L^2} = \| \varphi_k \|_{L^2} ,
		\end{equation}
		and in particular the $ \ell^2(L^2) $ norm of $\ulpsi(t)$ is also conserved:
		\begin{equation}
			\forall t \in I(M),  \;\ \| \ulpsi(t) \|_{\ell^2(L^2)} = \| \ulphi \|_{\ell^2(L^2)} .
		\end{equation}
		This conservation of the norm, along with the conservation of the norm of the solutions to the approximated problem given by Theorem \ref{thm:existence_regularized}, imply
		\begin{equation}
			\forall t \in I(M),\, \forall n \geq 1,  \;\  \| \ulphi \|_{\ell^2(L^2)} = \| \ulpsi^{[n]}(t) \|_{\ell^2(L^2)} 
			= \| \ulpsi(t) \|_{\ell^2(L^2)}.
		\end{equation}
		The weak convergence of $\ulpsi^{[n]}(t)$ to $\ulpsi(t)$ in $\ell^2(L^2)$ (implied by the weak convergence in $\ell^2(H^s)$ given by Lemma \ref{lem:convergence-solutions}) combined with the equality of the $\ell^2(L^2)$-norms imply the strong convergence: 
		\begin{equation}\label{lim:strong_convergence_ell_L2}
			\forall t \in I(M),  \;\ \|\ulpsi^{[n]}(t)-\ulpsi(t)\|_{\ell^2(L^2)} \xrightarrow[n\to\infty]{}0  .
		\end{equation}
		
		The strong convergence \eqref{lim:strong_convergence_ell_L2} can now be to lifted to $ \ell^2(L^2\cap L^r) $, for any $2\leq r<\twoasts$, thanks the uniform bound in the energy norm $ \ell^2(H^{s}) $. Indeed, by using \eqref{eq:bound-hoelder-sobolev-general-nu-and-two}, we have, for any~$2\leq r<\twoasts$,
		\begin{multline}
			\| \ulpsi(t) - \ulpsi^{[n]}(t) \|_{\ell^2(L^2\cap L^{r})} \leq   \eta_{s,\nu}(C_s \| \ulpsi(t) - \ulpsi^{[n]}(t) \|_{\ell^2(H^{s})}) \, \thetasnu  (\| \ulpsi(t) - \ulpsi^{[n]}(t) \|_{\ell^2(L^2)})  \\
			\leq   \eta_{s,\nu}(C_s ( 4M )) \, \thetasnu (\| \ulpsi(t) - \ulpsi^{[n]}(t) \|_{\ell^2(L^2)}) ,
		\end{multline} 
		where $\eta_{s,\nu}$ and $\thetasnu$ are defined in \eqref{eq:eta-thetasnu} and we have used the uniform bound \eqref{est:uniform_est_theorem} in the second inequality. The previous estimate together with \eqref{lim:strong_convergence_ell_L2} implies that, for any $ 2 \leq r < \twoasts  $
		\begin{equation}
			\forall t \in I(M),  \;\ \ulpsi^{[n]}(t) \xrightarrow[n\to\infty]{\mathrm{s}-\ell^2(L^2\cap L^r)} \ulpsi(t) .
		\end{equation}
		
		Now we can prove that $ \ulf(t) = \ulg(\ulpsi(t)) $. Recall from Lemma \ref{lem:convergence-solutions} that
		\begin{equation}\label{lim:weak_convergence_nonlince-a}
			\forall t \in I(M), \;\ \ulg^{[n]}(\ulpsi^{[n]}(t)) \xrightarrow[n\to\infty]{\mathrm{w}-\ell^2(L^2+L^{\nu^\prime})} \ulf(t) \,.
		\end{equation}
		Let us prove that
		\begin{equation}\label{lim:weak_convergence_nonlince-b}
			\forall t \in I(M), \;\ \ulg^{[n]}(\ulpsi^{[n]}(t)) \xrightarrow[n\to\infty]{\mathrm{s}-\ell^2(L^2+L^{\nu^\prime})} \ulg(\ulpsi(t)) \,.
		\end{equation}
		We estimate, for all $ t \in I(M) $, the norm of the difference by
		\begin{subequations}
			\begin{align}
				&\big\| \ulg^{[n]}(\ulpsi^{[n]}(t)) - \ulg(\ulpsi(t)) \big\|_{\ell^2(L^2+ L^{\nu'})}\\
				&=\big\| R^{[n]} \ulg(R^{[n]}  \ulpsi^{[n]}(t)) - \ulg(\ulpsi(t)) \big\|_{\ell^2(L^2+ L^{\nu'})}  \nonumber\\
				& \leq \| R^{[n]} \|_{\mathcal{B}(L^2+L^{\nu'})}  \big\| \ulg(R^{[n]}\ulpsi^{[n]}(t) ) - \ulg( R^{[n]} \ulpsi(t) ) \big\|_{\ell^2(L^2+ L^{\nu'})} \label{subeq:identification-limit-non-linearity1}\\
				&\quad + \| R^{[n]} \|_{\mathcal{B}(L^2+L^{\nu'})}  \big\| \ulg( R^{[n]} \ulpsi(t) ) - \ulg( \ulpsi(t) ) \big\|_{\ell^2(L^2+ L^{\nu'})}  \label{subeq:identification-limit-non-linearity2}\\
				&\quad +  \big\| ( R^{[n]} - \mathbb{I}) \ulg( \ulpsi(t) ) \big\|_{\ell^2(L^2+ L^{\nu'})} . \label{subeq:identification-limit-non-linearity3}
			\end{align}
		\end{subequations}
		Recalling that $ \| R^{[n]} \|_{\mathcal{B}(L^2+L^{\nu^\prime})} \leq 1 $ by \eqref{est:reg_Lp}, let us consider each term on the right-hand side separately:
		\begin{itemize}
			\item For \eqref{subeq:identification-limit-non-linearity1}: By the non-linearity estimate \eqref{est:Lnu_L2_norms}, recalling the uniform bounds \eqref{est:uniform_est_theorem} and \eqref{est:energy_bound_solution}, we have that
			\begin{align}
				\big\| \ulg(R^{[n]}\ulpsi^{[n]}(t) ) - \ulg( R^{[n]} \ulpsi(t) ) \big\|_{\ell^2(L^2+ L^{\nu'})} 
				& \leq C(M) \, \thetasnu (\| R^{[n]}\ulpsi^{[n]}_{k}(t)-R^{[n]}\ulpsi_k(t) \|_{\ell^2(L^2)})  \\
				& \leq C(M) \, \thetasnu  (\| \ulpsi^{[n]}(t)-\ulpsi(t) \|_{\ell^2(L^2)}),
			\end{align}
			where we have used \eqref{est:reg_Lp} and \eqref{lim:reg_strong_conv_Hs} to drop the dependence on the regularizing operator. Therefore,  by the strong convergence in norm \eqref{lim:strong_convergence_ell_L2}, we have that for all $t \in I(M) $,
			\begin{equation}
				\lim_{n \rightarrow \infty} \big\| \ulg(R^{[n]}\ulpsi^{[n]}(t) ) - \ulg( R^{[n]} \ulpsi(t) ) \big\|_{\ell^2(L^2+ L^{\nu'})} = 0.
			\end{equation}
			\item For \eqref{subeq:identification-limit-non-linearity2}: Again, by the same non-linearity estimate \eqref{est:Lnu_L2_norms}, the uniform bounds \eqref{est:uniform_est_theorem}, \eqref{est:energy_bound_solution} and the bound of the regularization operator \eqref{est:reg_Lp},  we have
			\begin{equation}
				\big\| \ulg( R^{[n]} \ulpsi(t) ) - \ulg( \ulpsi(t) ) \big\|_{\ell^2(L^2+ L^{\nu'})} \leq C(M) \, \thetasnu  (\| (R^{[n]}-\mathbb{I})\ulpsi(t) \|_{\ell^2(L^2)}),
			\end{equation}
			which vanishes as $n\to\infty$ thanks to the strong convergence \eqref{lim:reg_strong_conv_Hs} of the regularization operator. Hence,
			\begin{equation}
				\forall t \in I(M), \;\ \lim_{n \rightarrow \infty} \big\| \ulg( R^{[n]} \ulpsi(t) ) - \ulg( \ulpsi(t) ) \big\|_{\ell^2(L^2+ L^{\nu'})} = 0 .
			\end{equation}
			\item For \eqref{subeq:identification-limit-non-linearity3}: It suffices to use the strong convergence of the regularization operator~\eqref{lim:reg_strong_conv_plus} 
			to obtain that 
			\begin{equation}
				\forall t \in I(M), \;\ \lim_{n \rightarrow \infty} \big\| ( R^{[n]} - \mathbb{I}) \ulg( \ulpsi(t) ) \big\|_{\ell^2(L^2+ L^{\nu'})} = 0 .
			\end{equation}
		\end{itemize}
		We have thus established \eqref{lim:weak_convergence_nonlince-b}. Together with \eqref{lim:weak_convergence_nonlince-a} and the uniqueness of the weak limit, this shows that 
		\begin{equation}
			\forall t \in I(M), \;\ \ulf(t) =  \ulg (\ulpsi(t) ) .
		\end{equation}
		
		Replacing $f_k$ by $g_k(\ulpsi)$ in \eqref{eq:Kohn_Sham_with_f}, we have proven that $ \ulpsi \in L^\infty(I(M); \ell^2(H^{s}) )$ is a weak solution to the Kohn-Sham equations \eqref{eq:Kohn-Sham-vectorised} on the interval $ I(M) $, in the sense of Definition \ref{def:solutions}, such that:
		\begin{equation}\label{eq:unif-bound-l2Hs}
			\forall t \in I(M), \;\ \| \ulpsi(t) \|_{ \ell^2(H^{s})} \leq 2M,
		\end{equation}
		and each $L^2$ norm is conserved:
		\begin{equation}
			\forall t \in I(M), \, \forall k \in \mathbb{N}, \;\ \| \psi_k(t) \|_{L^2} = \| \varphi_k \|_{L^2} . 
		\end{equation}
		
		Now we prove the energy inequality $ \ \mathcal{E}(\ulpsi(t)) \leq \mathcal{E}(\ulphi ) $ by using the lower semicontinuity of the $ \ell^2(H^{s}) $ norm in the weak topology and the strong convergence of the non-linearity energy as $n\to \infty$: Recall the inequality for the non-linearity energy \eqref{est:nonlinear_energy_L2} and~\eqref{def:approx_nonlin_energy}
		\begin{align}
			| \mathscr{G}^{[n]}(\ulpsi^{[n]}(t)) & - \mathscr{G}(\ulpsi(t)) | \\
			& \leq  C(M) \, \thetasnu (\| R^{[n]}\ulpsi^{[n]}(t) - \ulpsi(t) \|_{\ell^2(L^2)})  \\ 
			&\leq C(M)	\, \thetasnu \Big( \| R^{[n]} \|_{\mathcal{B}(L^{2})} \| \ulpsi^{[n]}(t) - \ulpsi(t) \|_{\ell^2(L^2)} + \| (R^{[n]}-\mathbb{I}) \ulpsi (t) \|_{\ell^2(L^2)} \Big) .
		\end{align}
		The first term on the right-hand side vanishes as $n\to\infty$ by \eqref{lim:strong_convergence_ell_L2}, while the second term vanishes by strong convergence of the regularization operator \eqref{lim:reg_strong_conv_Hs}. Therefore, 
		\begin{equation}\label{eq:conv-nonlinear}
			\forall t \in I(M), \;\ \lim_{n \rightarrow \infty} | \mathscr{G}^{[n]}(\ulpsi^{[n]}(t)) - \mathscr{G}(\ulpsi(t)) | = 0.
		\end{equation}
		From the convergence of the non-linear energy \eqref{eq:conv-nonlinear}, the lower semicontinuity of the $ \ell^2(H^{s}) $ norm, and the conservation of energy of the approximated problem (see Theorem \ref{thm:existence_regularized}), we conclude that
		\begin{equation}
			\mathcal{E}(\ulpsi(t)) \leq \liminf_{n \rightarrow  \infty} \mathcal{E}^{[n]}(\ulpsi^{[n]}(t)) = \liminf_{n \rightarrow  \infty} \Big( \frac{1}{2} \| \ulpsi^{[n]}(t) \|_{\ell^2(H^{s})}^2 + \mathscr{G}( R^{[n]}\ulpsi^{[n]}(t)) \Big) = \mathcal{E}(\ulphi) ,
		\end{equation}
		which is the energy inequality we stated.
		
		To conclude the proof of Theorem \ref{thm:local_existence}, the only thing that remains to be proven is that
		\begin{equation}\label{eq:last-step}
			\ulpsi\in\mathscr{C}^{0,1/(2p_{s,r})} \big(I(M);\ell^2(L^2\cap L^r)\big),
		\end{equation}
		for all~$r\in[2,\twoasts)$. Applying \eqref{eq:bound-hoelder-sobolev-general-r}, and using that $\|\cdot\|_{\ell^2(L^2)}\le\|\cdot\|_{\ell^2(H^s)}$, we can estimate, for all~$r\in[2,\twoasts)$ and $t,t'\in I(M)$,
		\begin{align}
			&\| \ulpsi(t)-\ulpsi(t^\prime)  \|_{\ell^2(L^2\cap L^{r})} \notag \\
			&\leq (1+C_s^{1/p_{s,r}'}) \| \ulpsi(t)-\ulpsi(t^\prime)  \|_{\ell^2(H^{s})}^{1/p_{s,r}'}  \| \ulpsi(t)-\ulpsi(t^\prime)  \|_{\ell^2(L^2)}^{1/p_{s,r}} \notag \\
			&\leq (1+C_s^{1/p_{s,r}'}) (4M)^{1/p_{s,r}'}  \| \ulpsi(t)-\ulpsi(t^\prime)  \|_{\ell^2(L^2)}^{1/p_{s,r}}, \label{eq:c0}
		\end{align}
		where we used \eqref{eq:unif-bound-l2Hs} in the second inequality. Taking the limit $n\to\infty$ in \eqref{est:L2_holder_continuity}, using \eqref{lim:strong_convergence_ell_L2}, gives
		\begin{equation*}
			\| \ulpsi(t)-\ulpsi(t^\prime)  \|_{\ell^2(L^2)}\le 2 \max\{2M, \tilde{C}(M)\} |t-t^\prime|^{1/2} .
		\end{equation*}
		Inserting this into \eqref{eq:c0} establishes \eqref{eq:last-step}, and thus concludes the proof of Theorem~\ref{thm:local_existence}.
	\end{proof}

	\section{Extension to global solutions}\label{sec:global}
	
	In this section we prove Theorem \ref{thm:global_existence}, namely that the local in time solution constructed in Theorem \ref{thm:local_existence} can be extended to a global in time solution, assuming that the nonlinearity belongs to~$\Nglob$. The argument is a rather straightforward adaptation  from e.g.~\cite{cazenave2003semilinear}.
	
	\begin{proof}[Proof of Theorem \ref{thm:global_existence}]\label{proof:global_esitence}
		Consider an interval $ I \ni 0 $ over which a local solution of \eqref{eq:Kohn-Sham-vectorised} with initial condition $ \ulphi $ exists by Theorem \ref{thm:local_existence}. By points $ii)$ and $iii)$ of Theorem \ref{thm:local_existence}, we have that 
		\begin{equation}
			\forall t \in I, \;\ \| \ulpsi(t) \|_{\ell^2(L^2)} = \| \ulphi \|_{\ell^2(L^2)} \;\  \text{ and } \;\ \mathcal{E}(\ulpsi(t)) \leq \mathcal{E} (\ulphi ).
		\end{equation}
		We can therefore estimate the $\ell^2(H^{s}) $ norm for all $ t \in I $ by
		\begin{align} \notag
			\| \ulpsi(t) \|_{\ell^2(H^s)}^2 & =  2 \mathcal{E}(\ulpsi(t)) - 2  \mathscr{G}(\ulpsi(t)) \leq  2 \mathcal{E}(\ulpsi(t)) - 2  \mathscr{G}^-(\ulpsi(t))
			\\ \label{H12_estimate_global_ex}
			&\leq  \| \ulphi \|_{\ell^2(H^s)}^2 + 2  |\mathscr{G}(\ulphi)| + 2|\mathscr{G}^-(\ulpsi(t))|  . 
		\end{align}
		By \ref{ass4:nonlinearity_global_estimate} and the conservation of the $\ell^2(L^2)$ norm (see \ref{it:cons-l2L2} in Theorem \ref{thm:local_existence}), we deduce that, for all $t \in I$,
		\begin{equation*}
			\| \ulpsi(t) \|_{\ell^2(H^s)}^2 
			\leq \| \ulphi \|_{\ell^2(H^s)}^2  + 2  |\mathscr{G}(\ulphi)|  + a  \| \ulpsi(t) \|_{\ell^2(H^{s})}^2 +2D(\|\ulphi \|_{\ell^2(L^2)}^2) .
		\end{equation*}
		Rearranging the terms, we obtain a bound on the $ \ell^2(H^{s}) $ norm which depends only on the initial condition $ \ulphi $, and is independent of time:
		\begin{equation}\label{est:bound_dep_only_initial_cond}
			\forall t \in I , \;\ 	\| \ulpsi(t) \|_{\ell^2(H^s)} 
			\leq (1-a)^{-\frac{1}{2}}\Big( \| \ulphi \|_{\ell^2(H^s)}^2 + 2  |\mathscr{G}(\ulphi)| +  2D(\|\ulphi \|_{\ell^2(L^2)}^2) \Big)^{\frac{1}{2}} =:M.
		\end{equation}
		Observe that, in particular, $ \| \ulphi \|_{\ell^2(H^s)} \leq M $.
		
		We can now use this uniform bound to extend iteratively the domain of definition of the solution. Considering the constant $M>0$, and the bound on the initial condition, we can apply Theorem \ref{thm:local_existence} to obtain a local solution $ \ulpsi(t) \in \ell^2(H^s) $ over the interval $ [0,T(M)] $. Notice that the inequality \eqref{est:bound_dep_only_initial_cond} holds over the whole interval $[0,T(M)]$ as by Theorem \ref{thm:local_existence} \ref{it:cons-l2L2} and \ref{it:en-ineq}, we have conservation of the $\ell^2(L^2)$ norm and the energy inequality, which imply that we can estimate as in \eqref{H12_estimate_global_ex} to get that 
		\begin{equation}
			\sup_{t \in [0,T(M)]} \| \ulpsi(t) \|_{\ell^2(H^{s})} \leq M. 
		\end{equation}
		Setting $ {\tilde{\ulphi}} \coloneqq \ulpsi(T(M)) $, we have that
		\begin{equation}
			\| \tilde{\ulphi} \|_{\ell^2(H^s)} \leq M,
		\end{equation}
		and we can use $\tilde{\ulphi} $ as an initial condition to construct a solution $ \tilde{\ulu}(t) $ over the interval $[T(M),2T(M)] $. Defining the glued solution, with the same notation, as 
		\begin{equation}
			\ulpsi(t) =  
			\begin{cases}
				\ulpsi(t) , \;\ t \in [0,T(M)], \\
				\tilde{\ulu}(t), \;\ t \in [T(M),2T(M)],	
			\end{cases}
		\end{equation}
		we have that the $ \ell^2(L^2) $ norm is conserved and the energy inequality holds on the whole interval. In particular, for all $t \in [T(M),2T(M)]$,
		\begin{align}
			\| \ulpsi(t) \|_{\ell^2(L^2)} &=  \| \tilde{\ulphi} \|_{\ell^2(L^2)} = \| \ulpsi(T(M)) \|_{\ell^2(L^2)} = \| \ulphi \|_{\ell^2(L^2)},   \\
			\mathcal{E}(\ulpsi(t)) &\leq  \mathcal{E}(\tilde{\ulphi}) = \mathcal{E}(\ulpsi(T(M))) \leq \mathcal{E}(\ulphi) .
		\end{align}
		Notice that only the estimate from Assumption \ref{ass4:nonlinearity_global_estimate}, the conservation of the $\ell^2(L^2)$ norm and the energy inequality are needed to prove \eqref{est:bound_dep_only_initial_cond}: it is then clear that it holds also for the glued solution over the interval $[0,2T(M)]$. In particular 
		\begin{equation}
			\| \ulpsi(2T(M)) \|_{\ell^2(H^s)} \leq M.
		\end{equation}
		Therefore we can use the solution at the maximal time of the interval as an initial condition and iterate this process indefinitely to construct a solution on the whole real line, as each time, by \eqref{est:bound_dep_only_initial_cond}, we always have the same bound $ \| \ulpsi(mT(M)) \|_{\ell^2(H^s)} \leq M $, for any $m \in \mathbb{N}, \, m \geq 1 $. In fact, denoting again by $ \ulpsi(t)$ the solution obtained by gluing the solutions on each interval, the bound holds on the whole line:
		\begin{equation}
			\| \ulpsi \|_{L^\infty(\mathbb{R};\ell^2(H^s))} \leq M.
		\end{equation}
		Note that while the constructed solution is continuous in time with the regularities given by Theorem \ref{thm:local_existence}, for any interval $[m T(M), (m+1) T(M)] $, the gluing procedure allows to extend the continuity properties only with a slightly weaker notion.
		The glued global solution is clearly continuous in time, as right and left limits coincide at the intersection points of the intervals; in contrast, the H{\"o}lder regularity is lost globally, but holds locally, as we can glue H{\"o}lder functions with the same exponent smaller than 1 a finite number of times:
		\begin{equation*}
			\ulpsi \in \mathscr{C}\big(\mathbb{R};\ell^2(H^{s})_{\mathrm{weak}}\big) \cap \bigcap_{\nu \in [2,\twoasts ) } \mathscr{C}_{\mathrm{loc}}^{0,1/\psnu}\big(\mathbb{R};\ell^2(L^2\cap L^\nu)\big).
		\end{equation*}
		Still, by the fact that functions in $ W^{1,\infty}(\ell^2(H^{s})) $ are just absolutely continuous functions with~$L^\infty$ weak derivative, and absolutely continuous functions on intervals can be glued to yield absolutely continuous functions, the global solution constructed this way  admits a weak derivative:
		\begin{equation}
			\ulpsi \in W^{1, \infty}\big(\mathbb{R};\ell^2(H^{-s})\big) .
		\end{equation}
		This concludes the proof.
	\end{proof}

	\section{Well-posedness}\label{sec:Strichartz}
	
	\subsection{Dispersive Estimates and Uniqueness for \texorpdfstring{$1\le s<\frac32$}{s between 1 and 3/2}}
	
	In this section, we suppose that $s\in[1,\frac32)$. 
	The dispersive properties without loss of derivatives that hold in this case allow us to prove the uniqueness of the solutions. Denote the dispersion relation in Fourier space by $ \omega(|k|) = (1+|k|^2)^s $. We use the results from \cite{GUO20081642} to obtain Strichartz estimates relative to $\omega(-i\nabla)$.
	
	\begin{prop}\label{prop:strichartz_general_s_geq_1}
		Let $s \in [1,\frac32) $. A pair of indices $(q,r) \in \left[ 2 , \infty \right] $ are called (Schrödinger) admissible if $ \frac{1}{q} = \frac{3}{2}(\frac{1}{2} - \frac{1}{r} ) $. For any admissible pairs $(q,r)$, $(\tilde{q},\tilde{r})$ and any interval $I \subset \mathbb{R} $, the following estimates hold: 
		\begin{align} \label{eq:strichartz_general_s_geq_1}
			\| e^{-it(1-\Delta)^s} f \|_{L^q_t(I;L^r_x(\mathbb{R}^3)} &  \lesssim \| f \|_{L^2_x(\mathbb{R}^3)} , \\
			\Bigg\| \int_{0}^{t} e^{-i(t-\tau)(1-\Delta)^s} F(\tau) \diff \tau \Bigg\|_{L_t^{q}(I;L^r_x(\mathbb{R}^3))} & \lesssim \| F \|_{L^{\tilde{q}^\prime}_t(I;L^{\tilde{r}^\prime}_x(\mathbb{R}^3))},\label{eq:strichartz_general_s_geq_2}
		\end{align}
		where $\tilde{q}^\prime, \, \tilde{r}^\prime $ denote the H\"older conjugate exponent of $\tilde q$, $\tilde r$, and $0,t \in \overline{I}$.
	\end{prop}
	\begin{proof}
		We consider the non-homogeneous Littlewood-Paley projections
		\begin{equation*}
			\forall k \geq 1, \;\ P_k = \phi(2^{-k} (-i \nabla) ), \;\ P_{\leq 0} = \chi(-i \nabla) ,
		\end{equation*}
		where the radial functions $\phi \in \mathscr{C}^\infty_c(B(0,2) \setminus \overline{B(0,\frac12)}) $ and $\chi \in \mathscr{C}^\infty_c(B(0,2))$ have the properties
		\begin{equation}\label{eq:dyadic_decomp_1}
			\forall \xi \in \mathbb{R}^3 , \;\ 1 = \chi(\xi) + \sum_{k\geq1} \phi(2^{-k} \xi) 
		\end{equation}
		and
		\begin{equation}\label{eq:dyadic_decomp_2}
			\mathrm{supp} (\chi) \cap \mathrm{supp}(\phi(2^{-k} \cdot ) ) = \emptyset, \,\ k \geq 2, \;\  \mathrm{supp} (\phi(2^{-k} \cdot )) \cap \mathrm{supp}(\phi(2^{-k^\prime} \cdot ) ) = \emptyset, \,\ |k-k^\prime| \geq 2,
		\end{equation}
		see e.g. \cite[Section 2.2]{bahouri2011fourier}. We abstain from using the usual notation $\Delta_k$ for the Littlewood-Paley projections to avoid confusion with the Laplace operator.
		Note the following properties of the radial function $ \omega(\eta) = (1+\eta^2)^s ,\; \eta \in \mathbb{R}^+$:
		\begin{itemize}
			\item[] $ |\omega^\prime(\eta)| \sim \eta^{2s-1}$ and $ \forall m \in \mathbb{N}, \, m \geq 2, \;\ | \omega^{(m)}( \eta) | \lesssim \eta^{2s-m}, \;\ \eta \geq 1$;
			\item[] $ |\omega^\prime(\eta)| \sim \eta $ and $ \forall m \in \mathbb{N}, \, m \geq 2, \;\  | \omega^{(m)}( \eta ) | \lesssim \eta^{\mathrm{mod}(m,2)} \lesssim \eta^{2-m}, \;\ \eta < 1$;
			\item[] $ |\omega^{\prime\prime}(\eta)| \sim \eta^{2s-2}, \,\ \eta \geq 1$ and $ |\omega^{\prime\prime}(\eta)| \sim 1, \,\ \eta < 1$.
		\end{itemize}
		These conditions allow us to apply \cite[Theorem 1 and Remark 2]{GUO20081642} to obtain the estimates 
		\begin{equation*}
			\forall k \geq1, \;\ \| e^{-it(1-\Delta)^s} P_k f \|_{L^\infty} \lesssim |t|^{-3/2} 2^{3k(1-s)} \| f \|_{L^1}, \;\
			\| e^{-it(1-\Delta)^s} P_{\leq0} f \|_{L^\infty} \lesssim (1+|t|)^{-3/2} \| f \|_{L^1} . 
		\end{equation*}
		As $s \geq 1$, we can always estimate $ 2^{3k(1-s)}$ by $1$ uniformly in $k \geq 1$. 
		By the Riesz-Thorin Theorem we immediately obtain the estimates for $ r \in [2,\infty] $:
		\begin{equation*}
			\forall k \geq 1, \;\ \| e^{-it(1-\Delta)^s} P_k f \|_{L^r} \lesssim |t|^{-\frac{3(r-2)}{2r}} \| f \|_{L^{r^\prime}}, \;\
			\| e^{-it(1-\Delta)^s} P_{\leq0} f \|_{L^r} \lesssim (1+|t|)^{-\frac{3(r-2)}{2r}} \| f \|_{L^{r^\prime}}  .
		\end{equation*}
		By \eqref{eq:dyadic_decomp_1} and \eqref{eq:dyadic_decomp_2}, we can always estimate in the following way:
		\begin{align}
			\| e^{-it(1-\Delta)^s} P_k f \|_{L^r} 
			& = \| e^{-it(1-\Delta)^s} P_k (P_{\leq0} + \sum_{k^\prime\geq0} P_{k^\prime}) f \|_{L^r} \\
			& = \| e^{-it(1-\Delta)^s} P_k ( \delta_{k,1} P_{\leq0} + \sum_{|k^\prime-k|\leq1} \Delta_{k^\prime}) f \|_{L^r} \\
			& \lesssim |t|^{-\frac{3(r-2)}{2r}} \Big( \delta_{k,1} \| P_{\leq0} f \|_{L^{r^\prime}} + \sum_{|k^\prime-k|\leq1} \|  P_{k^\prime} f \|_{L^{r^\prime}}  \Big) ,
		\end{align}
		and similarly
		\begin{align}
			\| e^{-it(1-\Delta)^s} P_{\leq0} f \|_{L^r} 
			& = \| e^{-it(1-\Delta)^s} P_{\leq0} (P_{\leq0} + \sum_{k\geq0} P_k) f \|_{L^r} \\
			& = \| e^{-it(1-\Delta)^s} P_{\leq0} ( P_{\leq0} + \Delta_{1}) f \|_{L^r} \\
			& \lesssim (1+|t|)^{-\frac{3(r-2)}{2r}} ( \| P_{\leq0} f \|_{L^{r^\prime}} + \|  P_{1} f \|_{L^{r^\prime}}  ) .
		\end{align}
		Now recall the definition of the Besov norm, for $r \in [1,\infty) $,
		\begin{equation*}
			\| f \|_{B^{0}_{r,2}}^2 = \| P_{\leq0} f \|_{L^r}^2 + \sum_{k \geq 1} \| P_k f \|_{L^r}^2 ,
		\end{equation*}
		and note that by Minkowski inequality and \eqref{eq:dyadic_decomp_1}, we have the following relations for $r \in [2,\infty]$:
		\begin{equation*}
			\|  f \|_{L^r} \leq \| f \|_{B^{0}_{r,2}},  \quad \| f \|_{B^{0}_{r^\prime,2}} \leq \| f \|_{L^{r^\prime}} . 
		\end{equation*} 
		We can therefore estimate
		\begin{align}
			\| e^{-it(1-\Delta)^s} f \|_{L^r}^2 
			& \leq \| e^{-it(1-\Delta)^s} f \|_{B^{0}_{r,2}}^2 \\
			& = \| e^{-it(1-\Delta)^s} P_{\leq0} f \|_{L^r}^2 + \sum_{k \geq 1} \| e^{-it(1-\Delta)^s} P_k f \|_{L^r}^2 \\
			& \lesssim (1+|t|)^{-\frac{3(r-2)}{r}} ( \| P_{\leq0} f \|_{L^{r^\prime}}^2 + \|  P_{1} f \|_{L^{r^\prime}}^2  ) \\
			& \qquad + |t|^{-\frac{3(r-2)}{r}}  \sum_{k \geq 1} ( \delta_{k,1} \| P_{\leq0} f \|_{L^{r^\prime}}^2 + \sum_{|k^\prime-k|\leq1} \|  P_{k^\prime} f \|_{L^{r^\prime}}^2  ) \\
			& \lesssim |t|^{-\frac{3(r-2)}{r}} \| f \|_{B^{0}_{r^\prime,2}}^2 \lesssim |t|^{-\frac{3(r-2)}{r}} \| f \|_{L^{r^\prime}}^2 .
		\end{align}
		We are now in position to apply \cite[Proposition 1]{GUO20081642} to obtain the Strichartz estimates \eqref{eq:strichartz_general_s_geq_1} and~\eqref{eq:strichartz_general_s_geq_2}. Here it should be noted that, in the sense of \cite[Definition 1a)]{GUO20081642}, our exponents are~$\theta_1=\theta_2= \frac{3(r-2)}{2r}$, which are $<1$ for $ r \in [2,6) $, and the corresponding time integrability is~$ \frac{1}{q} = \frac{3}{2}(\frac{1}{2}-\frac{1}{r})$, with $q \in (6/5,\infty] $ forming the admissible pair $ (q,r) $.
	\end{proof}
	
	Note that it could be possible to obtain slightly finer Strichartz estimates, but \eqref{eq:strichartz_general_s_geq_1} and~ \eqref{eq:strichartz_general_s_geq_2} are sufficient for our purpose. 
	\begin{remark} 
		For $s > 1$ (and $s<\frac32$), the range of space integrability $r$ of the admissible pairs~$(q,r)$ does not cover the whole range of $ r \in [2,\twoasts) $ as in \ref{ass2:nonlinearity_Lp_bound}, only $ r \in [2,6) $. This is the reason for the restriction of the integrability range in \ref{ass2stri:nonlinearity_Lp_bound} and therefore in the class $\NStri$.
	\end{remark}
	
	We further adapt the Strichartz estimates \eqref{eq:strichartz_general_s_geq_1}--\eqref{eq:strichartz_general_s_geq_2} to the vector case in order to prove the uniqueness of the solution to the Kohn-Sham equations for $ s \geq 1$. 
	
	\begin{prop}
		Let $ s \in [1,\frac32) $. For any admissible pairs $(q,r)$, $(\tilde{q},\tilde{r})$ and any interval $I \subset \mathbb{R} $, the following estimates hold: 
		\begin{align}
			&\| e^{-it(1-\Delta)^s} \ulu \|_{L^q_t(\mathbb{R};\ell^2(L^{r}_x))} \lesssim \| \ulu \|_{\ell^2(L^2_x)} \\
			&\left\| \int_{0}^t e^{i(t-\tau)(1-\Delta)^s} F(\tau) \diff \tau \right\|_{L^q_t(I;\ell^2(L^r_x))} \lesssim \| F \|_{L^{\tilde{q}^\prime}_t(I;\ell^2(L^{\tilde{r}^\prime}_x))} , \label{est:strichartz_retarded}
		\end{align}
		where $\tilde{q}^\prime, \, \tilde{r}^\prime $ denote the H\"older conjugate exponent of $\tilde q$, $\tilde r$, and $0,t \in \overline{I}$.
	\end{prop}
	\begin{proof}
		The estimates follow from the Strichartz estimates  \eqref{eq:strichartz_general_s_geq_1}--\eqref{eq:strichartz_general_s_geq_2}, observing that as $ q \geq 2 $ and $ 2 \geq \tilde{q}^\prime $ as the exponents are admissible, by Minkowski inequality,
		\begin{equation}
			\| \cdot \|_{L^q_t(I;\ell^2(L^r_x))} \leq  \| \cdot \|_{\ell^2(L^q_t(I;L^r_x))} \text{  and  } \| \cdot \|_{\ell^2(L^{\tilde{q}^\prime}_t(I;L^{\tilde{r}^\prime}_x))} \leq  \| \cdot \|_{L^{\tilde{q}^\prime}_t(I;\ell^2(L^{\tilde{r}^\prime}_x))} .
		\end{equation}
	\end{proof}
	
	We can now prove a uniqueness result which will imply Theorem \ref{thm:well_posedness}.
	
	\begin{theorem}\label{thm:uniqueness}
		Let $s \in [1,\frac32) $ and $\ulg\in \Nloc\cap\NStri$. Any weak solution to the Kohn-Sham equations \eqref{eq:Kohn-Sham-vectorised} with initial condition $\ulphi \in \ell^2(H^s) $ over an interval $I \ni 0$ is unique.
	\end{theorem}
	\begin{proof}
		Consider two solutions $\ulpsi$ and $ \tilde{\ulpsi}$ to the problem \eqref{eq:Kohn-Sham-vectorised} over the interval $I \ni 0 $, both with initial condition $ \ulphi \in \ell^2(H^s) $. By Duhamel's formula over any subinterval $I^\prime$ of $I$ containing $0$ and for any $t \in I' \subset I$, 
		\begin{equation}
			\ulpsi(t) - \tilde{\ulpsi}(t) = -i \sum_{j\in J}\int_{0}^t e^{i(t-\tau)\Delta}  \big(\ulg_j(\ulpsi(\tau)) - \ulg_j( \tilde{\ulpsi}(\tau)) \big) \diff \tau ,
		\end{equation}
		holds in $ \ell^2(H^{-s})$ for a.a.~$ t \in I^\prime $. This relation holds also in $L^q_t(I^\prime;\ell^2(L^{r}_x))$ for $(q,r)$ admissible by \eqref{est:strichartz_retarded}. For any admissible pair $(q_{j},\nu_{j})$, where $\nu_{j} \in \left[ 2, \twoastsone  \right) $ is the exponent of \ref{ass2stri:nonlinearity_Lp_bound} for the term $\ulg_{j} $, we can estimate by the inhomogeneous Strichartz estimate \eqref{est:strichartz_retarded} to obtain
		\begin{align}
			\| \ulpsi(t) - \tilde{\ulpsi}(t) \|_{L^{q_{j}}_t(I';\ell^2(L^{\nu_{j}}_x))}
			& \leq \sum_{j\in J} \| \int_{0}^{t}  \big(\ulg_j(\ulpsi(\tau)) - \ulg_j( \tilde{\ulpsi}(\tau)) \big) \diff \tau \|_{L^{q_j}_t(I^\prime;\ell^2(L^{\nu_j}_x))} \\
			& \leq  C_{\mathrm{Str}} \sum_{j\in J}\| \ulg_j(\ulpsi(t)) - \ulg_j( \tilde{\ulpsi}(t) ) \|_{L_t^{q_j^\prime}(I^\prime;\ell^2(L_x^{\nu_j^\prime}))}.
		\end{align}
		By \ref{ass2stri:nonlinearity_Lp_bound}, we further estimate
		\begin{equation}
			\| \ulg_j(\ulpsi(t)) - \ulg_j(\tilde{\ulpsi}(t)) \|_{L^{q'_{j}}_t(I^\prime;\ell^2(L^{\nu'_{j}}_x))} 
			\leq  C_{\mathrm{Str}} L_M \| \ulpsi(t) - \tilde{\ulpsi}(t) \|_{L_t^{q'_j}(I^\prime;\ell^2(L_x^{\nu_j}))} \,,
		\end{equation}
		where $M>0$ is such that $\| \ulpsi(t) \|_{L^{\infty}(I^\prime;\ell^2(H^1))}\|\le M$ and 
		$\| \tilde{\ulpsi} (t) \|_{L^{\infty}(I^\prime;\ell^2(H^1))}\leq M$. 
		Note that on the interval $I$ we may assume that both solutions are uniformly bounded by a constant $M>0$. Therefore, summing over all $j$ in $J$, we get
		\begin{equation}
			\| \ulpsi(t) - \tilde{\ulpsi}(t) \|_{L^{q_{j}}_t(I^\prime;\ell^2(L^{\nu_{j}}_x))} \leq  C_{\mathrm{Str}} L_M \| \ulpsi(t) - \tilde{\ulpsi}(t) \|_{L_t^{q'_j}(I^\prime;\ell^2(L_x^{\nu_j}))} .
		\end{equation}
		Recall that $q_j^\prime < q_j$ as $ r_j < \twoasts  $ and thus, by H\"older's inequality, 
		\begin{equation}
			\sum_{j\in J} \| \ulpsi(t) - \tilde{\ulpsi}(t) \|_{L^{q_{j}}_t(I^\prime;\ell^2(L^{\nu_{j}}_x))} \\ 
			\leq C_{\mathrm{Str}} \Clip{M}
			\max_{j'\in J}|I^\prime|^{\frac{1}{q'_{j'}}-\frac{1}{q_{j'}}}
			\sum_{j\in J}     \| \ulpsi(t) - \tilde{\ulpsi}(t) \|_{L^{q_{j}}_t(I^\prime;\ell^2(L^{\nu_{j}}_x))} ,
		\end{equation}
		which, for any choice of $|I^\prime|$ small enough yields a factor $<1$ and therefore $\ulpsi(t) \equiv \tilde{\ulpsi}(t) $ for a.e. in $I^\prime$. As this holds for any arbitrary small interval $I^\prime \subset I$, we can glue this result to obtain the uniqueness on the whole interval.
	\end{proof}
	
	Once uniqueness has been established, the next result follows from classical arguments.
	
	\begin{proof}[Proof of Theorem \ref{thm:well_posedness}]\label{proof:well_posedness}
		Consider from Theorem \ref{thm:local_existence} a weak $\ell^2(H^s)$ solution $ \ulpsi $ in any open interval~$I^\prime$ of size smaller than $2T(M)$. Let $\tau \in I^\prime$. Set $\ulpsi(\tau)$ as the initial condition to obtain a local solution~$\tilde{\ulpsi}(t)$ in a subinterval $\tilde{I}$ such that $ \tilde{\ulpsi}(0) = \ulpsi(\tau)$. Then, by uniqueness, $\ulpsi(\cdot+\tau) (0) = \tilde{\ulpsi}(0) $ and they must coincide on the whole interval $I^\prime$. The energy must therefore be conserved, in fact 
		\begin{equation} 
			\mathcal{E}(\ulpsi(\tau)) \leq \mathcal{E}(\ulpsi(0)) = \mathcal{E}(\tilde{\ulpsi}(-\tau)) \leq \mathcal{E}(\tilde{\ulpsi}(0)) = \mathcal{E}(\ulpsi(\tau)),
		\end{equation}
		and as $\tau$ is arbitrary, the energy must be constant on $I^\prime$. Recall that by \eqref{est:holder_time} and \eqref{est:nonlinear_energy_L2} the non-linearities energy is continuous in time, and as $ \| \ulpsi(t) \|_{\ell^2(H^s)}^2 = 2 \mathcal{E}(\ulpsi(t)) - 2 \mathscr{G}(\ulpsi(t)) $ and the total energy is constant, the map $ t \mapsto \| \ulpsi(t) \|_{\ell^2(H^s)} $ is continuous. Recall that $ \ulpsi$ belongs to $\mathscr{C}(I^\prime;\ell^2(H^s)_{\mathrm{weak}}) $ by Theorem \ref{thm:local_existence}. Therefore we obtain that $ \ulpsi \in \mathscr{C}(I^\prime;\ell^2(H^s)) $. As \eqref{eq:Kohn-Sham-vectorised} holds in $\ell^2(H^{-s}) $ by Definition \ref{def:solutions} of a weak solution, we have that $ i \partial_t \ulpsi \in \mathscr{C}(I^\prime;\ell^2(H^{-s})) $ and thus $ \ulpsi \in \mathscr{C}^1(I^\prime; \ell^{2}(H^{-s}))$. Therefore $\ulpsi$ is a strong solution to \eqref{eq:Kohn-Sham-vectorised} (see again Definition \ref{def:solutions}).
		
		The blow-up alternative is proven by a standard contradiction argument, using the local existence result provided by Theorem \ref{thm:local_existence}, yielding the maximal interval $(T_{\mathrm{min}},T_{\mathrm{max}})$.

		Assume that there is a sequence of initial conditions $ (\ulphi_{n})_n \subset \ell^2(H^1) $ converging to $\ulphi$ in~$\ell^2(H^s)$. Consider now any closed interval $ \bar{I} \subset (T_{\mathrm{min}},T_{\mathrm{max}})$ and set $M>0$ as 
		\begin{equation}
			M \coloneq 2 \sup_{t \in \bar{I} } \| \ulpsi(t) \|_{\ell^2(H^s)} .
		\end{equation}
		Consider the strong solutions $\ulpsi_{n}$ to the equations \eqref{eq:Kohn-Sham-vectorised} with initial conditions $\ulphi_{n}$.  As $ \ulphi_{n} $ converges to $\ulphi$, there exists $N\in\mathbb{N}$ such that $ \| \ulphi_{n} \|_{\ell^2(H^s)} < M $ for all $n \geq N$ as $\| \ulphi \|_{\ell^2(H^s)} < M$ by definition. As the interval of definition $ \left( -T(M),T(M) \right) $ of the weak solutions constructed by Theorem \ref{thm:local_existence} depends only on the norm of the initial condition, we have that for all $n \geq N$, the solutions $\ulpsi_{n}$ are well defined on an interval 
		\begin{equation}
			(-T(M),T(M) ) \subset \left( -T(\|\ulphi\|_{\ell^2(H^s)}),T(\|\ulphi\|_{\ell^2(H^s)}) \right),
		\end{equation}
		and we have that, for all $n \geq N$,
		\begin{equation}
			\ulpsi_{n} \in L^\infty((-T(M),T(M) );\ell^2(H^s)) \cap W^{1,\infty} ((-T(M),T(M) );\ell^2(H^{-s})) , 
		\end{equation}
		with a uniform bound in $n$ for both norms (for the second one, it suffices to combine the fact that~$\ulpsi_{n}$ satisfies \eqref{eq:Kohn-Sham-vectorised} in $\ell^2(H^{-s})$ together with \eqref{est:single_nonlinearity_Lp} and the Sobolev embeddings \eqref{eq:sobolev_embedding_plus} and~\eqref{eq:sobolev_embedding_minus}). As in the proof of Lemma \ref{lem:convergence-solutions}, we then deduce that, for a.e. $t\in(-T(M),T(M))$, there exists a subsequence $\ulpsi_{n_m}(t)$ which converges weakly in $\ell^2(H^s)$ to an element $ \ulu(t)$ in~$\ell^2(H^s) $. Moreover, as in the proof of Theorem \ref{thm:local_existence}, we can prove convergence of $\ulpsi_{n_m}(t) $ to weak solutions of \eqref{eq:Kohn-Sham-vectorised}, and therefore we conclude that $\ulu \equiv \ulpsi$ by uniqueness.
		
		Next we strengthen the weak convergence of the subsequence $\ulpsi_{n_m}(t)$ to $\ulpsi(t)$ in $\ell^2(H^s)$ to the strong convergence in $\ell^2(L^2)$. Observe that the uniqueness of solutions in fact assures that the sequence $\ulpsi^{[n]}(t)$ converges to $\ulpsi(t)$ weakly in $\ell^2(H^s)$. Indeed, the previous argument can be repeated, using as a starting point any subsequence of the original sequence of initial conditions~$\ulphi_{n}$, obtaining that for any subsequence there is a further subsequence $\ulpsi_{n_{m_l}}(t)$ that converges to $\ulpsi(t)$. This implies that the sequence $ \ulpsi_{n}(t) $ converges to $\ulpsi(t)$. From the conservation of energy and the continuity in time of the non-linearities energies, we then deduce that~$\| \ulpsi_{n}(t) \|_{\ell^2(H^s)}$ converges uniformly to $\| \ulpsi(t) \|_{\ell^2(H^s)}$ in the interval $(-T(M),T(M))$, since
		\begin{equation}
			\| \ulpsi_{n} (t) \|_{\ell^2(H^s)} =  \lim_{n \rightarrow \infty} (2\mathcal{E}(\ulpsi_{n}(t)) -2 \mathscr{G}(\ulpsi_{n}(t) )) 
		\end{equation}
		and the first term is constant while the second is H\"older continuous. 
		Then, by \cite[Proposition 1.1.2 ii)]{cazenave2003semilinear}, it follows that
		\begin{equation}
			\ulpsi_{n} \xrightarrow[n\to \infty]{} \ulpsi 
		\end{equation}
		in $\mathscr{C}((-T(M),T(M));\ell^2(H^s))$.
		Note that this construction depends only on $M$ defined through the arbitrary closed interval $\bar{I}$ and as such continuous dependence must hold on $\bar{I}$. From the fact that continuous dependence holds on any closed interval in the maximal interval of definition of the solution $ (T_{\mathrm{min}},T_{\mathrm{max}})$ we conclude.
	\end{proof}
	
	\subsection{The \texorpdfstring{$ 0 < s < 1 $}{ s smaller than 1} case}
	We conclude this section with a short discussion concerning the problem of uniqueness in the case where $s < 1$. The dispersive estimates for the Laplacian are the crucial step in the previous proof of uniqueness in Theorem \ref{thm:uniqueness}. If $ s \in (0,1) $, the dispersive estimates of the fractional operator $ (1 - \Delta)^{s}$ involve additional derivatives on the right-hand side of \eqref{est:strichartz_retarded}, see the lack of a uniform estimate of $ 2^{3k(1-s)} $ in the proof of Proposition \ref{prop:strichartz_general_s_geq_1} or \cite[Appendix B]{breteaux2025propagation} for the $s=\frac12$ case. This poses a problem in the case of a non-smooth non-linear term as $ g(\ulpsi) = \rho_{\ulpsi}^\alpha \ulpsi $ (with $\alpha$ small) covered by our analysis. In the case of smoother pure power non-linear terms (i.e. for $\alpha$ large enough), we could provide the estimates required by a contraction mapping argument, as is done for the fractional non-linear Schr\"odinger equation with cubic non-linearity in \cite{guo2010global}. Likewise, if $\alpha=0$, similar results hold for Hartree non-linearities, at least for pseudo-relativistic dispersion relation ($=\frac12$), see the analysis of the boson star equation in e.g. \cite{breteaux2025propagation,breteaux2025exponentialdecayoutsidelight,lenzmann2007well}.
	
	The same issue concerning uniqueness of solutions appears in the one-particle fractional Schr\"odinger equation
	\begin{equation}
		\label{eq:1particle_FRNLS}
		i \partial_t u(t,x) = (1-\Delta)^{s} u(t,x) + \mu | u(t,x) |^{2 \alpha} u(t,x). 
	\end{equation}
	Given the assumptions on the non-linearity we consider here, say $ \alpha = \frac13 $, there is a clear obstruction to prove that the Duhamel formula defines a contraction map in $L^p$ spaces as the typical estimate is of the form appearing in Condition \ref{ass2:nonlinearity_Lp_bound}. Our solution to this issue when~$s\in[1,\frac32)$ is to use dispersive estimates in Theorem \ref{thm:uniqueness} to trade space and time integrability and prove that the map is a contraction. In the case where $s \in (0,1) $, however, we have to introduce an amount of regularity to use dispersive estimates, and we are naturally led to address this question using Sobolev or Besov spaces to find a contraction. This strategy does not prove successful, in fact the Duhamel formula for the difference of two terms involves estimating the term
	\begin{equation}
		\| | u |^{2 \alpha} u - | v |^{2 \alpha} v \|_{B^{\tau}_{p^\prime,2}} .
	\end{equation}
	The most convenient way to rewrite this term, to extract the difference of $u$ and $v$ needed to show that the formula defines a contraction, seems to be 
	\begin{multline}
		| u |^{2 \alpha} u - | v |^{2 \alpha} v \\ 
		= \int_{0}^{1} \big( (\alpha+1) \overline{(u-v)} |\lambda u + (1-\lambda) v |^{2 \alpha} + \alpha (u-v) |\lambda u + (1-\lambda) v |^{2\alpha-2} (\lambda u + (1-\lambda) v s)^2 \big),
	\end{multline}
	which contains the main technical difficulty. In fact any analogue to the Leibniz rule (Kato-Ponce inequality, tame inequalities for products in Besov spaces and related properties of algebras of regular functions) all imply to estimate the term $ |\lambda u + (1-\lambda) v |^{2 \alpha} $ in $B^{\tau}_{p^\prime,2}$, or $H^{\tau,p}$, which is not well defined.
	
	It should also be noted that, while at a first look the case $ s= \frac12 $ with propagator $ e^{-it(1-\Delta)^\frac12} $ can be expected to be similar to the analysis for the propagator of the free Klein-Gordon equation, its dispersive properties are worse. For instance, the uniqueness argument for the nonlinear Klein-Gordon equation with a (possibly small) power non-linearity in \cite{GINIBRE198915} hinges on an estimate of the type (see for instance \cite{brenner1975p}, for $d = 3, \, 2 \leq p < \infty $)
	\begin{equation}
		\Big\| \frac{\sin{t|\nabla|}}{|\nabla|} u \Big\|_{B^{0}_{p,2}} \lesssim |t|^{1-2 \cdot 3 (\tfrac{1}{p^\prime} - \tfrac{1}{2})} \|  u \|_{B^{0}_{p^\prime,2}} ,
	\end{equation}
	which is not available in our case.
	
	The lack of a unique solution appears to be a technical problem due to the interplay between the lack of smoothness of the pure power non-linearity and the dispersive estimates that require additional smoothness in the case $s \in (0,1) $.

	\printbibliography
	
	\doclicenseThis
\end{document}